\documentclass[aos, preprint]{imsart}

\RequirePackage[OT1]{fontenc}
\RequirePackage{amsthm,amsmath}
\RequirePackage[numbers]{natbib}
\RequirePackage[colorlinks,citecolor=blue,urlcolor=blue]{hyperref}

\startlocaldefs
\numberwithin{equation}{section}
\theoremstyle{plain}
\newtheorem{thm}{Theorem}[section]
\endlocaldefs

\usepackage{siunitx}

\usepackage{xr-hyper}

\usepackage{amsfonts}
\usepackage{amsmath,amssymb,color}
\usepackage{graphicx}
\usepackage{subcaption}
\usepackage{array,multirow}
\usepackage{float}
\usepackage{afterpage}
\usepackage{arydshln}
\usepackage{natbib}
\usepackage{titlesec}

\usepackage{blkarray}
\usepackage{mathtools}
\usepackage[toc,page]{appendix}
\usepackage{extarrows}
\usepackage{xr}
\usepackage{makecell}

\makeatletter
\newcommand{\doublewidetilde}[1]{{%
  \mathpalette\double@widetilde{#1}%
}}
\newcommand{\double@widetilde}[2]{%
  \sbox\z@{$\m@th#1\widetilde{#2}$}%
  \ht\z@=.9\ht\z@
  \widetilde{\box\z@}%
}
\makeatother

\newcommand{\aaa}{\boldsymbol \alpha}

\newcommand{\ttt}{\boldsymbol \theta}
\newcommand{\TT}{\boldsymbol \Theta}
\newcommand{\nnu}{\boldsymbol \nu}

\newcommand{\rr}{\boldsymbol r}
\newcommand{\ms}{\mathcal S}
\newcommand{\mo}{\mathbf 1}
\newcommand{\mz}{\mathbf 0}
\newcommand{\mb}{\mathcal B}
\newcommand{\bq}{\boldsymbol q}

\newcommand{\mm}{\boldsymbol m}
\newcommand{\nn}{\boldsymbol n}
\newcommand{\uu}{\boldsymbol u}

\newcommand{\ba}{\mathbf A}

\newcommand{\ma}{\mathcal A}

\newcommand{\vv}{\boldsymbol v}

\newcommand{\mc}{\mathcal C}
\newcommand{\ee}{\boldsymbol e}

\newcommand{\pp}{{\boldsymbol p}}

\newcommand{\ww}{\boldsymbol w}

\newcommand{\RR}{\boldsymbol R}
\newcommand{\bo}{\boldsymbol}

\newcommand{\mcr}{\mathcal R}
\newcommand{\boa}{\boldsymbol a}
\newcommand{\bob}{\boldsymbol b}
\newcommand{\wt}{\widetilde}

\newcommand*{\Cdot}{\raisebox{-0.5ex}{\scalebox{1.5}{$\cdot$}}}

\newtheorem{corollary}{Corollary}[section]

\newtheorem{definition}{Definition}[section]

\newtheorem{example}{Example}[section]

\newtheorem{lemma}{Lemma}[section]

\newtheorem{proposition}{Proposition}[section]

\newtheorem{remark}{Remark}[section]

\begin{document}

\begin{frontmatter}

\title{Partial Identifiability of Restricted Latent Class Models}
\runtitle{Identifiability of Restricted Latent Class Models}
\tableofcontents

\thankstext{T1}{This research is partially supported by National Science Foundation grants   SES-1659328 and DMS-1712717, and Institute of Education Sciences
grant R305D160010.}

\author{\fnms{Yuqi} \snm{Gu}\ead[label=e1]{yuqigu@umich.edu}}
\and
 \author{\fnms{Gongjun} \snm{Xu}\ead[label=e2]{gongjun@umich.edu}}
 \affiliation{University of Michigan}

\address{Department of Statistics, \\ 
University of Michigan,\\
West Hall, 1085 S University Ave, \\
 Ann Arbor, MI 48109\\
\printead{e1,e2}}

\runauthor{Gu and Xu}




\begin{abstract}
Latent class models have wide applications in social and biological sciences.
In many applications, pre-specified restrictions are imposed on the parameter space of  latent class models, through a design matrix, to reflect practitioners'  assumptions about how the observed responses depend on subjects' latent traits. Though widely used in various fields, such restricted latent class models  suffer from  nonidentifiability due to their discreteness nature and complex structure of restrictions. 
This work  addresses the fundamental identifiability  issue of restricted latent class models  by developing a general framework for strict and partial identifiability of the model parameters. 
Under correct model specification, the developed identifiability conditions only depend on the  design  matrix and are easily checkable, which provide useful practical guidelines for designing statistically valid diagnostic tests. 
Furthermore, the new theoretical framework is applied to establish, for the first time,  identifiability of several designs from cognitive diagnosis applications.

\end{abstract}


\end{frontmatter}

\section{Introduction and Motivation}\label{sec-intro}
Latent class models are widely used in social and biological sciences to model unobserved discrete latent attributes. These models often assume each latent class represents a configuration  of the targeted  latent attributes that can explain the observed responses.
In many  applications, pre-specified restrictions are imposed on the parameter space of the latent class model, through a design matrix. These restrictions reflect practitioners' understanding about how the responses depend on the underlying latent attributes. 
This paper studies such a family of {\it restricted} latent class models, 
which have been widely employed in various fields. The following are several examples.
\begin{enumerate}
\item[(1)] \textit{Cognitive Diagnosis in Educational Assessment.} 
Restricted latent class models play a key role  in cognitive diagnosis modeling in educational and psychological assessment.  Cognitive diagnosis aims to make a classification-based decision   on an individual's latent attributes, based on his or her observed responses to a set of designed diagnostic items (questions).
 In the majority of models, the latent classes are characterized by profiles of the binary states of mastery/deficiency of the targeted ability attributes, while there are models that allow polytomous ordinal attributes \citep{davier2008general}.
  The restricted structure usually comes from the design matrix that specifies what latent attributes each item measures  \citep[e.g.,][]{junker2001cognitive, HensonTemplin09, rupp2010diagnostic, dela2011}. See Section \ref{sec-real-q} for several data examples, including the  Test of English as a Foreign Language (TOEFL) \citep[e.g.,][]{davier2008general} and  Trends in   International Mathematics and Science Study.

 \item[(2)] \textit{Psychiatric Evaluation.} Restricted latent class models have also been used in psychiatric evaluation. Here the  responses are manifested symptoms and the latent classes represent the profiles of presence/absence  of a set of underlying psychological or psychiatric disorders. The restricted structure results from the fact that each symptom may be shared by multiple disorders, which are specified by psychiatric diagnosis guidelines.
See   examples in \cite{Templin}, \cite{jaeger2006distinguishing}, and \cite{de2018analysis}. 
 \item[(3)] \textit{Disease Etiology Detection.}  Another   application of restricted latent class models is the diagnosis of disease etiology in epidemiology \citep{wu2017nested}. Here the observed responses are imperfect measurements of subjects' biological samples, and the latent classes are the configurations of existence or non-existence of a set of pathogens underlying a certain disease. The restricted structure naturally arises from the fact that each  measurement may only target certain pathogens.  
\end{enumerate}

Despite the popularity of the restricted latent class models, the fundamental identifiability issue is challenging to address.
Model identifiability is a prerequisite for making   statistical inferences. 
The study of identifiability   of  latent class models dates back to decades ago \citep{mchugh1956efficient, teicher1967identifiability, goodman1974exploratory}.
For unrestricted latent class models, \cite{gyllenberg1994non} showed the model is not identifiable in the sense that, there always exists some set of parameters, such that one can construct a different set of parameters which lead to the same distribution of the responses.
Such nonidentifiablity has likely
impeded statisticians from looking further into this problem \citep{allman2009}.
Due to the difficulty of establishing strict identifiability in such scenarios,  \cite{elmore2005application} and \cite{allman2009} studied the {\it generic} identifiability of these models.
The idea of generic identifiability is closely related to concepts in algebraic geometry and implies that the model parameters are identifiable almost everywhere in the parameter space, excluding only a Lebesgue measure zero set. \cite{allman2009} established generic identifiability results for various latent variable models, including the unrestricted latent class models.

 The complex constraints  of the restricted latent class models pose  additional   challenge to the study of model identifiability.
The existing results of generic identifiability in \cite{allman2009} do not apply to restricted latent class models, because  the restrictions imposed by the design matrix already constrain the model parameters of a restricted latent class model into a measure-zero (and hence potentially unidentifiable) subset of the parameter space of an unrestricted latent class model.
To address the identifiability issue under restrictions,
\cite{xu2017} proposed a set of sufficient conditions for   identifiability of a family of restricted latent class models.
However, a key assumption in \cite{xu2017}  is that the    design matrix has to satisfy a certain structural constraint and that the latent class space has to be saturated (see Section \ref{sec-idf-issue} for more details), which is often difficult to meet and may even be unrealistic in practice; 
see examples in \cite{de2004higher}, \cite{jaeger2006distinguishing}, \cite{HensonTemplin09}, \cite{dela2011}, \cite{lee2011cognitive} and many others.
The same strong assumption is also imposed in \cite{xu2018jasa} for identifiability of the design matrix.
Therefore, the existing theory is hardly applicable to popular designs in the literature, and the previously proposed conditions may not serve as good guidelines for future test designing.
Moreover, the techniques developed as in \cite{xu2017} for specific presumable design structure are not applicable to general  designs. 
The fundamental identifiability issues of the restricted latent class models remain largely underexplored and call for new identifiability theory.

This paper proposes a general framework of strict and partial identifiability for restricted latent class models.
Practical sufficient conditions for strict and partial identifiability are proposed and their necessity is discussed. 
In particular, depending on the two different types of algebraic structures of restricted latent class models, 
we introduce and study two useful notions of partial identifiability, respectively (see Sections \ref{sec-two-para} and   \ref{sec-multi-para}). 
The established identifiability results are widely applicable in practice, by relaxing most of the constraints imposed on the design matrix.
Moreover,  under correct model specification, all the identifiability conditions  only depend on the design matrix and are easily checkable by practitioners.
We apply the new theory to several existing designs  and establish identifiability under them for the first time in the literature. 

The rest of the paper is organized as follows. Section \ref{sec-model-exp} introduces the general model setup of  restricted latent class models, including model and  data examples in cognitive diagnosis applications; and then discusses the limitations of the existing studies.
Sections \ref{sec-two-para} and   \ref{sec-multi-para} present our main identifiability results.
 Section \ref{sec-extend} includes extensions of the new theory to some more complicated models.
Section \ref{sec-disc} gives a further discussion, and proofs of the theoretical results  are presented in the Supplementary Material.

\section{Model Setup, Examples and Identifiability Issues}\label{sec-model-exp}
We start with the setup for a latent class model with binary responses. 
Suppose there are $J$  dichotomous items, denoted by the item set $\ms =\{1,\ldots,J\}$.
For any subject, the observed variables are his/her binary responses to the $J$ items,   denoted by  $\RR=(R_1,\ldots,R_J)^\top\in \{0,1\}^J$. 
To model the distribution of the responses, we assume there are $m$ latent classes existing  in the population denoted by $\ma =\{\aaa_0,\ldots,\aaa_{m-1}\}$, where $m>1$ is assumed known.
For any $\aaa\in \ma$, we use $p_{\aaa } =P(\ba=\aaa )$  to denote  the proportion of subjects in the population that belong to class $\aaa$. 
Under this specification, we have $p_{\aaa}\in(0,1)$ and $ \sum_{\aaa\in \ma}p_{\aaa} =1$.
In the application of cognitive diagnosis,   a latent class $\aaa$ usually  denotes a knowledge state characterized by a profile of mastery/deficiency of a set of   latent attributes, and is represented by a binary   vector (see Section \ref{sec-cdm}).
 
Assume that a subject's latent class membership $\ba$ follows a categorical distribution with  population  proportion parameters $\pp =(p_{\aaa},\aaa\in\ma)$.   
Given a subject's latent class membership $\ba$, the responses  $\RR=(R_1,\ldots,R_J)$     are assumed to be  conditionally independent  and   each $R_j$ follows a Bernoulli distribution  with  the positive response probability  $ \theta_{j,\aaa}= P(R_j=1\mid \ba=\aaa)$. 
This local independence  is a common assumption in latent class modeling   \citep[e.g.,][]{agresti2003categorical,allman2009}.
We call these $\theta_{j,\aaa}$'s as the item parameters, and write $\TT =(\theta_{j,\aaa};\,j\in\ms,\,\aaa\in\ma)$, which is a $J\times m$ matrix. The rows of $\TT$ are indexed by the $J$ items in $\ms$, and the columns by the $m$ latent classes in $\ma$.
The model parameters are then characterized by   $\pp$ and $\TT$.

We focus on a general family of restricted latent class models that are popularly used in various social and biological applications. Under these restricted latent class models,
the item parameters in $\TT$     are  restricted  by certain prespecified  structures to reflect  experts' understanding or hypotheses on how the responses to each diagnostic  item depend on the latent classes.  
In particular, the restricted latent class models    assume that for any item $j$, there exists an item-specific set of latent classes $\mc_j$; and the classes in $\mc_j$ share the same value of positive response probability, which is higher than those of the other latent classes. We denote such a set of latent classes by 
\begin{equation}\label{eq-dcj}
\mc_j = \Big\{\aaa\in\ma:\, \theta_{j,\aaa} = \max_{\aaa^\star\in\ma}\theta_{j,\aaa^\star}\Big\}.
\end{equation}
The latent classes in $\mc_j$ then correspond to those  subjects who are ``most capable" of giving a positive response to item $j$, and for each $j\in\ms$,
\begin{equation}\label{eq-constraints}
\max_{\aaa\in\mc_j} \theta_{j, \aaa} 
= \min_{\aaa\in\mc_j} \theta_{j, \aaa} 
> \theta_{j,\aaa'},
\quad 
\forall \aaa'\notin\mc_j.
\end{equation}  
Additionally, it is assumed that there exists a  universal ``least capable" class $\aaa_0$ such that 
$\theta_{j, \aaa} \geq \theta_{j,\aaa_0}$ for any $\aaa\in\ma$ and $j \in \ms$. 
Note that a latent class $\aaa'$ satisfying $\aaa'\notin\mc_j$ and  $\theta_{j, \aaa'} 
> \theta_{j,\aaa_0}$ can be viewed as   ``partially capable".

 Different restricted latent class models specify the  $\TT$ and the constraint sets $\mc_j$'s differently to respect the underlying scientific assumptions. 
\color{black} To  illustrate this, we present various model examples and real data examples in Sections \ref{sec-cdm} and \ref{sec-real-q}.  In Section \ref{sec-idf-issue} we discuss the identifiability issue and limitations of the existing works, which call for the new identifiability theory.

\subsection{Restricted Latent Class Models in Cognitive Diagnosis}\label{sec-cdm}
The restricted latent class models have recently gained great interests in cognitive diagnosis with applications in educational assessment, psychiatric evaluation and many other disciplines \citep[e.g.,][]{rupp2010diagnostic, dela2011, culpepper2015bayesian,wang2018tracking,chen2018hidden}. 
Cognitive diagnosis is the process of arriving at a classification-based decision about an individual's latent attributes, based on the observed surrogate responses. Such diagnostic information plays an  important role in constructing efficient, focused remedial strategies for improvement in individual performance.

The restricted latent class models are important statistical tools   in cognitive diagnosis to detect the presence or absence of multiple fine-grained   attributes.
Cognitive diagnosis models in the psychometrics literature mostly consist of binary attributes, while general diagnostic models with categorical attributes were also considered in \cite{davier2008general}. In this work, we focus on the case of binary attributes.
Specifically, consider a cognitive diagnosis test with $J$ items designed to measure $K$  binary latent attributes.
  Under the introduced model setup, a latent class $\aaa$ is represented by a configuration of the $K$ latent attributes, denoted by a $K$-dimensional binary vector $\aaa=(\alpha_1,\ldots,\alpha_K)$, where $\alpha_k\in\{0,1\}$ denotes the deficiency or mastery of the $k$th attribute. A latent class $\aaa$ is also called an attribute profile. 
The latent class space $\ma$ is a subset of $\{0,1\}^K$.  If $\ma = \{0,1\}^K$, we say   $\ma$ is \textit{saturated}, which means the population contain subjects with all the possible configurations of attribute profiles. The universal least capable latent class $\aaa_0$ corresponds to the all-zero attribute profile, i.e., $\aaa_0 =(0,\ldots,0)$. 

The restrictions in cognitive diagnosis models is encoded by the so-called $Q$-matrix \citep{Tatsuoka1983}. A $Q$-matrix $Q = (q_{j,k})$ is a $J\times K$ matrix with binary entries $q_{j,k}\in\{0,1\}$ indicating the absence or presence of the dependence of the $j$th item on the $k$th attribute. Generally, $q_{j,k}=1$ means that item $j$ requires the mastery of attribute $k$ to solve and $q_{j,k}=0$ means the opposite. The $j$th row vector $\bq_j$ of   $Q$,   called the $\bq$-vector, gives the attribute requirements of item $j$. 
 See examples of $Q$-matrices in Section \ref{sec-real-q}.

In the following, we review some popular  cognitive diagnosis models   and illustrate how they fall into the family of  restricted latent class models. We first introduce some notations. For two vectors $\boldsymbol a=(a_1,\ldots,a_K)$, $\boldsymbol b=(b_1,\ldots,b_K)$ of the same dimension $K$, we write $\boldsymbol a\succeq\boldsymbol b$  if $a_i\geq b_i$ for all $i=1,\ldots,K$; and $\boldsymbol a\succneqq\boldsymbol b$ if $\boldsymbol a\succeq\boldsymbol b$ and $\boldsymbol a\neq \boldsymbol b$. Denote $\boldsymbol a-\boldsymbol b =(a_1-b_1,\ldots,a_K-b_K)$ and $\boldsymbol a\vee\boldsymbol b =(\max\{a_1,b_1\},\ldots,\max\{a_K,b_K\})$. We also denote the all-zero- and all-one vectors  by $\mathbf 0$ and $\mathbf 1$, respectively. 
\begin{example}[Conjunctive DINA and Disjunctive DINO]\label{exp-dina}
\normalfont{The Deterministic Input Noisy output ``And" gate (DINA) model proposed in \cite{junker2001cognitive} and the Deterministic Input Noisy output ``Or" gate (DINO) model proposed in \cite{Templin} are popular and basic diagnostic models, which adopt the conjunctive   and   disjunctive assumptions, respectively. 
Specifically, under DINA, a subject needs to master all the required attributes of an item to be ``capable" of it, and mastering the attributes not required by the item will not compensate for the lack of required ones. That is, the required attributes of an item act ``conjunctively" and the positive response probability is
\begin{eqnarray*}
	\theta_{j,\aaa}^{DINA} =\left\{ \begin{array}{cl}
 	1-s_j, &\mbox{ if } \aaa\succeq \bq_j,\\
 	g_j, &\mbox{ otherwise. }
 \end{array}\right.
\end{eqnarray*}
where   
$s_j$ is the  slipping parameter, which denotes the probability that a capable subject slips the positive response, and $g_j$ is the guessing parameter, which denotes the probability that a non-capable subject coincidentally gives the positive response by guessing.
Under  DINO, 
a subject only needs to master one of the required attributes to be ``capable" of an item. That is, the required attributes of an item act ``disjunctively" and
\begin{eqnarray*}
	\theta_{j,\aaa}^{DINO} =\left\{ \begin{array}{cl}
 	1-s_j, &\mbox{ if } \exists k~\text{s.t.}~ \alpha_k = q_{j,k}=1,\\
 	g_j, & \mbox{ otherwise.}
 \end{array}\right.
\end{eqnarray*}
where
$s_j$ and $g_j$ are the slipping and guessing parameters.    
Both the DINA and DINO models assume  $1-s_j>g_j$ for all $j$.

The DINA and DINO models are restricted latent class models with appropriately defined constraint sets $\mc_j$'s. Specifically, under the conjunctive DINA model, the $\mc_j$ defined in \eqref{eq-dcj} takes the form of
 \begin{equation}\label{eq-gamma-conj}
 \mc_j = \{\aaa\in\ma:\,  \aaa\succeq \bq_j \}, \quad j\in\ms;
 \end{equation} 
 while under the disjunctive DINO model, the $\mc_j$ defined in \eqref{eq-dcj} becomes
$
 \mc_j = \{\aaa\in\ma:\, \mbox{if } \exists k~\text{s.t.}~ \alpha_k = q_{j,k}=1\}$ for  $ j\in\ms.
$}
  \end{example}

\begin{example}[Main-Effect Cognitive Diagnosis Models]
\label{exp-main-eff}
\normalfont{An important  family of cognitive diagnosis models assume that the
$\theta_{j,\aaa}$ depends on the main effects of those attributes required by item $j$, but not their interactions. This family include the popular reduced Reparameterized Unified Model \citep[reduced-RUM;][]{dibello1995unified}, Additive Cognitive Diagnosis Models   \citep[ACDM;][]{dela2011}, the Linear Logistic Model  
\citep[LLM;][]{maris1999estimating}, and the  General Diagnostic Model \citep[GDM;][]{davier2008general}.
We call them the Main-Effect Cognitive Diagnosis Models.
In particular, under the reduced-RUM,
$
\theta^{\,RUM}_{j,\aaa} = \theta^+_j{\prod}_{k=1}^K r_{j,k}^{q_{j,k}(1-\alpha_k)},
$
where $\theta^+_j=P(R_j=1| \aaa\succeq\bq_j)$ represents the positive response probability of a capable subject of $j$, and $r_{j,k}\in(0,1)$ is the parameter penalizing not possessing attribute $k$ required by item $j$. Equivalently, the item parameter in reduced-RUM can be written as $\log \theta^{\,RUM}_{j,\aaa} = \beta_{j,0} + \sum_{k=1}^K\beta_{j,k}(q_{j,k}\alpha_k)$, where $\beta_{j,k}\geq 0$ for $q_{j,k}=1$.
Similarly, the ACDM  assumes the parameter $\theta_{j,\aaa} $ can be   written as the linear combination of the main effects of the required attributes: 
$
\theta^{\,ACDM}_{j,\aaa} = \beta_{j,0} + {\sum}_{k=1}^K\beta_{j,k}(q_{j,k}\alpha_k).
$
The LLM assumes a logistic link function  with
$ 
\mbox{logit}(\theta^{\,LLM}_{j,\aaa}) =  \beta_{j,0} + {\sum}_{k=1}^K\beta_{j,k}(q_{j,k}\alpha_k) .
$ 
These Main-Effect models  are  restricted latent class models, and under them, the $\mc_j$ defined in \eqref{eq-dcj} takes the form of $\mc_j=\{\aaa\in\ma:\,\aaa\succeq\bq_j\}$.
}
 \end{example}

\begin{example}[All-Effect Cognitive Diagnosis Models]\label{exp-all-eff}
\normalfont{Another popular type  of cognitive diagnosis models   assume that the positive response probability depends on the main effects and the interaction effects of the required attributes of the item. We call these models   the All-Effect cognitive diagnosis models, of which the GDINA model  \citep{dela2011}, the log-linear cognitive diagnosis models   \citep[LCDM;][]{HensonTemplin09}, and the  general diagnostic model  \citep[GDM;][]{davier2008general} are examples.
It was recently shown in \cite{von2014log} and \cite{davier2014dina} that the GDINA and LCDM can be rewritten as GDMs with extended skill space.
In particular, given a $Q$-matrix, denote the set of attributes required by item $j$ by $\mathcal K_{\bq_j} = \{1\leq k\leq K: q_{j,k}=1\}$, then the item parameter under GDINA is
\begin{equation}\label{eq-gdina}
\theta_{j,\,\aaa}^{\,GDINA} = {\sum}_{S\subseteq \mathcal K_{\bq_j}}\beta_{j,\,S}{\prod}_{k\in S}\alpha_k,
\end{equation}
where $\beta_{j,\,S}\geq 0$.
Note that the DINA model is a submodel of the GDINA model by setting all the $\beta_{j,\,S}$ coefficients in \eqref{eq-gdina}, other than $\beta_{j,\,\varnothing}$ and $\beta_{j,\,\mathcal K_{\bq_j}}$, to zero. 
Similar to the GDINA model, the LCDM adopts the logistic link function  and assumes that  
$
\mbox{logit}(\theta_{j,\,\aaa}^{\,LCDM}) = 
 {\sum}_{S\subseteq \mathcal K_{\bq_j}}\beta_{j,\,S}
 {\prod}_{k\in S}\alpha_k.$
The All-Effect models  are  restricted latent class models, and under them the $\mc_j$  in \eqref{eq-dcj} also takes the form $\mc_j=\{\aaa\in\ma:\,\aaa\succeq\bq_j\}$.
}
 \end{example}

When the latent class space $\ma$ is saturated with $\ma=\{0,1\}^K$, we have $m=|\ma|=2^K$. In practice, however, this may not always hold. 
For instance, researchers may assume there exist additional restrictions on the dependence structure among the latent attributes, such as an attribute hierarchy with some attributes being the prerequisite for some others \citep{leighton2004attribute, templin2014hierarchical}. 
A hierarchical structure among the $K$ attributes would reduce the number of possible attribute profiles from  $2^K$ to $m$ ($m<2^K$),   by excluding those   not respecting the hierarchy. 
For example, consider a diagnostic test with $K=2$ attributes. If it is scientifically reasonable to assume the first attribute is the prerequisite for the second one, then 
 the latent class space is reduced to  $\mathcal A=\{(0,0),(1,0), (1,1)\}$ with $m=|\ma|=3$, since the attribute profile $(0,1)$ does not respect this hierarchy.  
 Note that as shown in \citep{von2014located}, a  cognitive diagnosis model with such a linear hierarchy can   equivalently reduce  to a located latent class model with $m<2^K$ classes.

In this work, we assume the latent class space $\ma$ is pre-specified.
This would be the case when  practitioners have solid scientific reasons or prior knowledge from exploratory data analysis to assume certain   structure among attributes. 
This work aims to answer the question that for an arbitrary $\ma\subseteq\{0,1\}^K$, what kind of conditions would guarantee identifiability of $\TT$ and $\pp = (p_{\aaa},\aaa\in\ma)$.

 All the cognitive diagnosis models reviewed in Examples \ref{exp-dina}--\ref{exp-all-eff} are restricted latent class models. 
 We call them the {\it $Q$-restricted latent class models}, since the  $\mc_j$'s  and model constraints are further determined by the $Q$-matrix. 
Moreover, we call the DINA and the DINO models the {\it two-parameter Q-restricted latent class models}, since each item has exactly two item parameters, and we call the Main-Effect and All-Effect   models   as {\it multi-parameter Q-restricted latent class models}.

\subsection{Real Data Examples}\label{sec-real-q}
To further illustrate  the constraints induced by the design  matrix, we present  several  applications that utilize restricted latent class models as  cognitive diagnosis modeling tools. 
  
 \begin{example}[TOEFL Internet-based Testing Data]\label{exp-toefl}

\normalfont{TOEFL, short for Test of English as a Foreign Language, is a standardized test to measure English language ability of non-native speakers. 
Restricted latent class models have been   used   to analyze the TOEFL data by researchers at Educational Testing Service \citep[ETS;  e.g.,][]{davier2005ETS,davier2008general}. 
For instance,  \cite{davier2008general} proposed a general diagnostic model (GDM), which was used to analyze the TOEFL     reading   section  of two parallel forms,   A and B, with their  $Q$-matrices analyzed and specified by content experts.
In particular, the   forms A and B contain 39 and 40 items with
four latent   attributes: {\it $\alpha_1$: Word meaning, $\alpha_2$: Specific information, $\alpha_3$: Connect information, and $\alpha_4$: Synthesize and organize}. Table \ref{tab-toefl} gives the   summary of the two $Q$-matrices by presenting each $\bq$-vector's frequencies in them. For instance, the first line in Table \ref{tab-toefl} reads (1, 0, 0, 0) for the row $\bq$-vector and (9, 9) for the frequencies. This means that there are nine   items with $\bq$-vector (1, 0, 0, 0) in form A and nine in form B, respectively. 
 Under the restrictions induced by the $Q$-matrices, the  diagnostic models used to analyze the TOEFL data fall in the family of   restricted latent class models.
\begin{table}[H]
\centering
\caption{TOEFL iBT field test $Q$-matrix entry frequencies, Reading Forms A and B}
\resizebox{12cm}{!}{%
\begin{tabular}{cccc|cc}
\hline
\multicolumn{4}{c|}{$Q$-matrix row $\bq$-vectors} & \multicolumn{2}{c}{$\bq$-vector frequency} \\
\hline
\multirow{2}{2cm}{\centering Word meaning} & \multirow{2}{2cm}{\centering Specific information} & \multirow{2}{2cm}{\centering Connect information} & \multirow{2}{2cm}{\centering Synthesize and
organize} & \multirow{2}{*}{\centering Form A} & \multirow{2}{*}{\centering Form B} \\
& & & & &\\
\hline
1 & 0 & 0 & 0 & 9 & 9\\
0 & 1 & 0 & 0 & 8 & 11\\
1 & 1 & 0 & 0 & 1 & 1\\
0 &  0  & 1  & 0  & 10 &  10\\
1 & 0 & 1 & 0 & 0 & 1\\
0 & 1 & 1 & 0 & 2 & 0\\
0 & 1 & 0 & 1 & 1 & 0\\
0 & 0 & 1 & 1 & 7 & 8\\
1 & 0 & 1 & 1 & 1 & 0\\
\hline
\end{tabular}
}
\label{tab-toefl}
\end{table}
}
\end{example}

\begin{example}[Trends in International Mathematics and Science Study]\label{exp-timss-8th}
\normalfont{Trends in International Mathematics and Science Study (TIMSS) is a large scale cross-country assessment, administered by the International Association for the Evaluation of Educational Achievement. TIMSS evaluates the mathematics and science abilities of fourth and eighth graders  every four years since 1995 and  covers more than 40 countries.  
The TIMSS data allows one to analyze trends in student progress that can provide feedback for future improvement in areas needing further instruction \citep{lee2011cognitive}. 
Researchers have used the cognitive diagnosis models to analyze the TIMSS data \citep[e.g.,][]{lee2011cognitive, choi2015cdm, yamaguchi}.
For instance, a $43\times 12$ $Q$-matrix constructed by mathematics educators and  researchers was specified for the TIMSS 2003 eighth grade mathematics assessment \citep{choi2015cdm}. A total number of $12$ fine-grained attributes are identified, which fall in five big categories of skill domains measured by the eighth grade exam, Number, Algebra, Geometry, Measurement, and Data. 
The  $Q$-matrix   is presented in Table 1 in the Supplementary Material.
\cite{choi2015cdm} used  DINA model  to fit the dataset   containing responses  sampled from   8912 U.S. and 5309 Korean students. 
Main-Effect and All-Effect diagnostic models have also been applied to analyze the TIMSS data \citep[e.g.,][]{yamaguchi}.
}
\end{example}

\begin{example}[Fraction Subtraction Data]\label{exp-frac}
\normalfont{
The dataset contains $536$ middle school students' binary responses to $20$ fraction subtraction items that were designed for diagnostic assessment.
The $Q$-matrix contains eight attributes (the $20\times 8$  $Q$-matrix   is presented in Table 2 in the Supplementary Material).
Many researchers have used various restricted latent class models models to fit this dataset   \citep[e.g.,][]{de2004higher,decarlo2011analysis,HensonTemplin09,dela2011}.
}
\end{example}

\subsection{Concept of Identifiability and Issues with Existing Works}\label{sec-idf-issue}
Though widely used in various applications, the identifiability issue of restricted latent class models remains largely unaddressed. 
We next introduce the concept of identifiability and discuss the limitations of the exiting theory.

For the introduced restricted latent class models, the probability mass function of the response pattern $\RR$ is 
\begin{equation}\label{eq-pmf}
P(\RR=\rr\mid \TT,\pp) = \sum_{\aaa\in\ma}p_{\aaa}\prod_{j=1}^J\theta_{j,\aaa}^{r_j}(1-\theta_{j,\aaa})^{1-r_j},
\quad \rr\in\{0,1\}^J.
\end{equation}
 Following the definition of identifiablity in the   literature \citep[e.g.,][]{casella2002statistical},  the model parameters $(\TT,\pp)$ of a restricted latent class model are identifiable  if for any $(\TT,\pp)$ in the parameter space $\mathcal T$, there is no $(\bar\TT,\bar\pp)\neq (\TT,\pp)$ such that
\begin{equation}\label{eq-orig}
\mathbb P(\RR=\rr\mid \TT,\pp) =\mathbb  P(\RR=\rr\mid \bar \TT,\bar\pp) ~\mbox{ for all } ~ \rr\in \{0,1\}^J. 
\end{equation}
 In the following, we also say that the model parameters are {\it strictly} identifiable if the above condition holds.

To establish model identifiability,
a strong and often impractical assumption made by previous works is that the $Q$-matrix must contain at least one $K\times K$ identity submatrix $I_K$ up to some row permutation, that is,   the $Q$-matrix must contain all $K$ distinct single-attribute $\bq$-vectors \citep{chen2015statistical,Xu15, xu2017, id-dina}.
A $Q$-matrix satisfying this requirement is also said to be complete under the DINA model \citep{Chiu}.  
For  general $Q$-restricted latent class models including the multi-parameter models, \cite{xu2017} requires at least two disjoint  $K\times K$  identity submatrices in $Q$ to establish identifiability.
However, in practice, 
in the existence of a large number of fine-grained attributes and complex cognitive process, a $Q$-matrix   rarely satisfies such requirements. For the TOEFL   data   in Example \ref{exp-toefl}, there does not exist any item that solely requires the fourth skill attribute  in both $Q$-matrices. 
 For the $Q$-matrix of the TIMSS data in Example \ref{exp-timss-8th},   only three attributes  (1, 7 and 8) out of twelve are measured by some single-attribute items. 
For the $Q$-matrix   in Example \ref{exp-frac},  there are only two attributes (2 and 7) out of eight measured by some single-attribute items.
Many other examples can be found in the literature \citep[e.g.,][]{jaeger2006distinguishing,HensonTemplin09, dela2011, lee2011cognitive}.
Moreover, another strong assumption   made in existing works \cite{xu2017,id-dina} is that $\ma=\{0,1\}^K$, i.e., $p_{\aaa}>0$ for any $\aaa\in\{0,1\}^K$, which   fails    when some attribute profiles  are deemed impossible to exist.

Such identifiability issues  of cognitive diagnosis models have   long been recognized \citep{de2004higher,davier2008general, TatsuokaC09, decarlo2011analysis, MarisBechger,zhang2013non,davier2014dina}. 
For instance, \cite{davier2008general} pointed out in the study of  the TOEFL data that larger numbers of skills (i.e., $K$) very likely pose problems with identifiability, unless the number of items per skill  is ``sufficiently" large. 
But given the complicated structure of constraints, 
how   the number of items and the form of the design matrix influence identifiability is still an open problem in the literature.

This work addresses  this open problem by developing a general theoretical framework based on a key technical tool, the \textit{indicator matrix} $\Gamma$.
Under an arbitrary restricted latent class model,
we define $\Gamma$ to be a $J\times m$ matrix using the sets $\mc_j$'s.
The $\Gamma$-matrix has the same size as the matrix $\TT$, with rows indexed by items in $\ms$, and columns by latent classes in $\ma$. 
The $(j,\aaa)$th entry of $\Gamma$ is
\begin{equation}\label{eq-def-gamma}
\Gamma_{j,\aaa}=I(\aaa\in \mc_j), \quad j\in\ms,~\aaa\in\ma, 
\end{equation}
which is a binary indicator of whether   $\aaa$ is   ``most capable" to give a positive response to  $j$.  
For $\aaa\in\ma$, denote the $\aaa$th column vector of $\Gamma$ by $\Gamma_{\Cdot,\aaa}$. 
The $\Gamma$-matrix defined this way turns out to be a useful tool for developing the identifiability theory, and it helps to relax many of the existing strong assumptions, as shown later in Sections \ref{sec-two-gamma} and \ref{sec-gen}. Indeed, most of our identifiability conditions can be  represented as requirements on the structure of $\Gamma$, since the information of which latent classes achieve the highest level of $\theta_{j,\aaa}$ of item $j$ is  what our theoretical derivations essentially rely on.
 Depending on two different algebraic structures of the restricted parameter spaces, we next consider two types of restricted latent class models and present  their identifiability  results  in Sections \ref{sec-two-para} and \ref{sec-multi-para}, respectively.

\section{Identifiability Results for Two-Parameter Models}\label{sec-two-para}
This section considers two-parameter restricted latent class models where each item $j$ has two item parameters, i.e., $|\{\theta_{j,\aaa}:\,\aaa\in\ma\}|=2$. 
Specifically, a two-parameter model assumes that for each item $j$, the latent classes in  $\mc_j$ share a same positive response probability, denoted by $\theta_j^+$,  and the latent classes in the complement set $\ma\setminus\mc_j$ share another same  positive response probability, denoted by $\theta_j^-$. We assume $\theta_j^+>\theta_j^-$.
Note that the unique item parameters in $\TT$   reduce to $(\ttt^+,\ttt^-)$, where  $\ttt^+=(\theta_1^+,\ldots,\theta_J^+)^\top$ and $\ttt^-=(\theta_1^-,\ldots,\theta_J^-)^\top$.
 The motivation for studying the two-parameter   models comes from the popular DINA   and DINO models in cognitive diagnosis, as introduced in Example \ref{exp-dina}. 
Moreover, the study of the two-parameter models  provides insight into understanding other restricted latent class models, as they serve as submodels for many multi-parameter models.

Under a two-parameter model, the $\Gamma$-matrix fully captures the model structure, in the sense that $\theta_{j,\aaa}=\theta_j^+$ if $\Gamma_{j,\aaa}=1$ and $\theta_{j,\aaa}=\theta_j^-$ if $\Gamma_{j,\aaa}=0$.
So in this scenario, 
if $\Gamma$ contains two identical columns, then the corresponding latent classes have the same item parameters across all items. Namely, if $\Gamma_{\Cdot,\aaa}=\Gamma_{\Cdot,\aaa'}$, then $\TT_{\Cdot,\aaa}=\TT_{\Cdot,\aaa'}$. Thus from an identifiability perspective, these two latent classes are equivalent and can not be distinguished based on their observed responses. 
This implies that in order to achieve strict identifiability of the proportion parameters $\pp=(p_{\aaa},\aaa\in\ma)$, it is necessary that each latent class in $\ma$ should correspond to a distinct column vector of $\Gamma$.
We shall call   such a $\Gamma$-matrix   separable.

\begin{definition}\label{def-gamma-com}
A $\Gamma$-matrix is said to be separable, if any two column vectors of $\Gamma$ are distinct. Otherwise, we say $\Gamma$ is inseparable.
\end{definition}

  To see how the separability of the $\Gamma$-matrix influences model identifiability, we start with an ideal case  with all the item parameters $(\ttt^+,\ttt^-)$ known.  
The following proposition characterizes the importance of  $\Gamma$'s separability.

\begin{proposition} \label{prop1}
Consider a two-parameter restricted latent class model with  known   $(\ttt^+,\ttt^-)$. 
Then the proportion parameters $\pp$ are  identifiable if and only if the $\Gamma$-matrix is separable.
\end{proposition}

 We use the following example as an illustration.
\begin{example}\label{ex7}
\normalfont{
	Consider the   $Q$-matrix in \eqref{eq-cc} with $K=2$ attributes. Under the DINA  model  with $\mc_j$ in the form of \eqref{eq-gamma-conj}, if $\ma=\{0,1\}^2=\{\aaa_0=(0,0)$, $\aaa_1=(1,0)$, $\aaa_2=(0,1)$, $\aaa_3=(1,1)\}$, then $\Gamma^{(1)}$ in \eqref{eq-cc} represents the corresponding $\Gamma$-matrix, which is inseparable. Specifically, we can see that $\Gamma_{\Cdot,\aaa_0}=\Gamma_{\Cdot,\aaa_2}$ and the two classes $\aaa_0$ and $\aaa_2$ have the same item parameters, $\TT_{\Cdot,\aaa_0}=\TT_{\Cdot,\aaa_2}=\ttt^-$. Thus  $\aaa_0$ and $\aaa_2$ are not distinguishable and equivalently, their proportion parameters $p_{\aaa_0}$ and $p_{\aaa_2}$ are not identifiable. 
\begingroup
\setlength{\belowdisplayskip}{0pt}
\begin{eqnarray}\label{eq-cc}
\quad Q =
\begin{pmatrix}
    1 & 0 \\
    1 & 1 \\
\end{pmatrix}
\begin{aligned}
&~~~\xLongrightarrow[\ma=\{0,1\}^2]{DINA} &
\Gamma^{(1)} = 
\begin{blockarray}{cccc}
 \aaa_0 & \aaa_1 & \aaa_2 & \aaa_3 \\
\begin{block}{(cccc)}
  0 & 1 & 0 & 1\\
  0 & 0 & 0 & 1\\
\end{block}
\end{blockarray}~; \\
& \xLongrightarrow[\ma=\{0,1\}^2\backslash\{0,1\}]{DINA} &
\Gamma^{(2)}  = 
\begin{blockarray}{ccc}
 \aaa_0 & \aaa_1   & \aaa_3 \\
\begin{block}{(ccc)}
  0 & 1   & 1\\
  0 & 0   & 1\\
\end{block}
\end{blockarray}~. 
\end{aligned}
\end{eqnarray}
\endgroup
On the other hand, if prior knowledge shows that the first attribute is the prerequisite for the second, then $\ma$ reduces to $\{0,1\}^2\backslash\{(0,1)\}$ and the   $\Gamma$-matrix becomes $\Gamma^{(2)}$ in \eqref{eq-cc}.   The  $\Gamma^{(2)}$ is separable, with each $\aaa$ having a distinct column vector in $\Gamma$ and $\TT_{\Cdot,\aaa_0}\neq \TT_{\Cdot,\aaa_1}\neq \TT_{\Cdot,\aaa_3}$.  Therefore Proposition \ref{prop1} gives that $\pp$ is identifiable in the ideal case with known $\TT$.
}
\end{example}

An inseparable $\Gamma$-matrix violates the necessary condition for identifying $\pp$ under the two-parameter models.
To study the ``partial" identifiability of $\pp$ when $\Gamma$ is inseparable, we next define an equivalence relation ``$\stackrel{}{\sim}$" of latent classes   induced  by the column vectors of $\Gamma$. 
Specifically, we define  $\aaa\stackrel{}{\sim}\aaa'$  if and only if $\Gamma_{\Cdot,\,\aaa} = \Gamma_{\Cdot,\,\aaa'}.$  
Let $C$ be the number of distinct column vectors of $\Gamma$  and $\ma_1,\ldots,\ma_C$ be the  $C$ equivalence classes under $\stackrel{}{\sim}$. 
Let $\aaa_{\ma_i}$ be a representative of $\ma_i$ and we write $[\aaa_{\ma_i}]=\ma_i$.
We define the grouped population proportion parameters to be 
\begin{equation}\label{eq-group}
\nu_{[\aaa_{\ma_i}]} :=  \sum_{\aaa:\,\aaa\in \ma_i }p_{\aaa},\quad 
 \mbox{ for } i=1,\ldots,C,
\end{equation}
and write  $\nnu= (\nu_{[\aaa_{\ma_1}]},\ldots, \nu_{[\aaa_{\ma_C}]})^\top $.
  When  $\Gamma$ is separable, we have $C=m$, $\nnu=\pp$ and
  each $\aaa$ represents a unique equivalence class.

The following result shows that under an inseparable $\Gamma$-matrix, though $\pp$ are not identifiable, the parameters $\nnu$ are identifiable. 

\begin{proposition}\label{prop1'}
Consider a two-parameter   model 
 with  known   $(\ttt^+,\ttt^-)$. 
When the $\Gamma$-matrix is inseparable, $\nnu$ is identifiable.
Moreover, the latent classes in the same equivalence class can not be distinguished in the sense that for any model parameters $ \pp\neq \bar\pp$, if 
 $\nu_{[\aaa_{\ma_i}]} = \bar \nu_{[\aaa_{\ma_i}]}, \mbox{ where }\bar \nu_{[\aaa_{\ma_i}]} = \sum_{\aaa:\,\aaa\in \ma_i}\bar p_{\aaa}$ for   $i=1,\ldots,C,$
 then $\mathbb P(\RR\mid\TT,\pp) = \mathbb P(\RR\mid\TT,\bar\pp)$. 
\end{proposition}
\noindent
\color{black} 
When $\Gamma$ is inseparable, Proposition \ref{prop1'} implies that even in the ideal case with known $(\ttt^+,\ttt^-)$, the identification of  $\nnu$ is the strongest identifiability result one can obtain for two-parameter restricted latent class models. 
This therefore motivates us to introduce  the following definition of the $\pp$-partial identifiability  when both $(\ttt^+,\ttt^-)$ and $\pp$ are unknown.

\begin{definition}[$\pp$-partial identifiability]
For a two-parameter restricted latent class model with a given $\Gamma$-matrix,  the model parameters $(\ttt^+,\ttt^-,\pp)$ are said to be $\pp$-partially identifiable  if   $(\ttt^+,\ttt^-,\nnu)$ are  identifiable.
\end{definition}

 We point out that when the $\Gamma$-matrix is separable, the $\pp$-partial identifiability exactly becomes the strict identifiability. 
When $\Gamma$ is inseparable, the definition of $\pp$-partial identifiability here refers to partially identifying the proportion parameters $\pp$, while the item parameters still need to be strictly identifiable. 
Such definition suits for the needs of cognitive diagnosis applications, by ensuring   the  identification of   the  equivalent   attribute profiles of interest, and also ensuring the estimability of all   item parameters  so that the  quality of the items can be accurately evaluated and validated.

In the framework of $\pp$-partial identifiability, the following Section \ref{sec-two-gamma} presents a general  identifiability result,  allowing  $\ma$ to be arbitrary and $\Gamma$ to be inseparable. 
Section \ref{sec-2p-q} further focuses on the   family of  $Q$-restricted latent class models
and discusses the necessity of the proposed conditions.  Section \ref{sec-2p-exp} includes the applications of the new theory.

\begin{remark}\label{remark1}
 For the   family of two-parameter $Q$-restricted latent class models, the $\Gamma$-induced equivalence classes can be obtained as follows. 
 We define two sets of attribute profiles under the conjunctive DINA  and disjunctive  DINO assumptions, respectively:
\begin{eqnarray}\label{eq-group-q}
   &&  ~~\mathcal R^{Q,conj} =
\{\aaa=\vee_{h\in S}\,\bq_h:  S \subseteq \mathcal S\},~
\mathcal R^{Q,disj}  =\{\mo-\aaa: \aaa\in\mathcal R^{Q,conj}\}, 
\end{eqnarray}
where $\vee_{h\in S\,}\bq_h=(\max_{h\in S}\{q_{h,1}\},\ldots,\max_{h\in S}\{q_{h,K}\})$, and $\vee_{h\in\varnothing}\,\bq_h$ is defined to be the all-zero vector.
We claim that when  $\ma=\{0,1\}^K$, the  $\mathcal R^{Q,conj}$ or $\mathcal R^{Q,disj}$ is a complete set of representatives of the conjunctive or disjunctive equivalence classes, respectively; 
the proof of this result is given in Section B of the Supplementary Material.
Moreover, for any latent class space $\ma\subseteq\{0,1\}^K$, define a map $f(\Cdot):\, \ma \rightarrow \mathcal R^{Q,conj}$ (or $R^{Q,disj}$)   which sends each attribute pattern $\aaa\in\ma$ to the element in $\mathcal R^{Q,conj}$  (or $R^{Q,disj}$)  equivalent to $\aaa$. Then $f(\ma)$ forms a complete set of conjunctive  or disjunctive representatives. 
A similar grouping operation in the saturated and conjunctive case   was introduced in \cite{zhang2013non}. 
Consider Example \ref{ex7} for an illustration. If $\ma=\{0,1\}^2$,  $\Gamma^{(1)}$ is inseparable. The equivalence class representatives are $\mathcal R^{Q,conj} =\{(0,0),\allowbreak (1,0),\allowbreak (1,1)\}$  by \eqref{eq-group-q}  
and  $\nnu  =(\nu_{[0,0]},\allowbreak\nu_{[1,0]},\allowbreak \nu_{[1,1]})$ with $\nu_{[0,0]} = p_{(0,0)}+p_{(0,1)}$, $\nu_{[1,0]} = p_{(1,0)}$, $\nu_{[1,1]} = p_{(1,1)}$. 
On the other hand, $\Gamma^{(2)}$ is separable with latent class space $\ma=R^{Q,conj}$.
 This also illustrates that   a separable $\Gamma$-matrix does not necessarily correspond to a $Q$-matrix containing an identity submatrix $I_K$.
Therefore, compared with existing theory, 
 the $\Gamma$-matrix provides a more suitable tool than the $Q$-matrix for studying identifiability of $Q$-restricted models.
 \end{remark}

\subsection{Strict and Partial Identifiability}\label{sec-two-gamma}
This subsection presents conditions depending on the $\Gamma$-matrix that lead to the $\pp$-partial identifiability of a   two-parameter restricted latent class model.
 We first introduce some notations.
Based on the constraint sets $\mc_j$'s, we categorize the entire set of items $\ms=\{1,\ldots,J\}$ into two subsets, the set of \textit{non-basis} items $S_{non}$ and that of \textit{basis} items $S_{basis}$ as follows,
 \begin{equation}\label{basis}
	S_{non} = \{j: \exists h \in \ms\setminus\{j\}, \text{ s.t. } \mc_{h} \supseteq \mc_j\} ~\mbox{ and }~
S_{basis} = \ms\setminus S_{non}.
\end{equation}
By this definition, an item $j$ is a non-basis item if the capability of item $j$ implies capability of some other item, and a basis item otherwise. 
With a slight abuse of notation, for any subset of items $S\subseteq\ms$, let   
$\mc_{S} = \cap_{j\in S\,}\mc_j$ denote the set of latent classes that are most capable of all the items in $S$. 
We introduce the next definition of $S$-\textit{differentiable} to describe  the relation between an item and a set of items.

 \begin{definition}\label{def-diff}
For an item $j$ and a set of items $S$ that does not contain $j$, item $j$ is said to be \textit{$S$-differentible} if there exist two subsets $S^+_j$, $S^-_j$ of $S$, which are not necessarily non-empty or disjoint, such that 
\begin{equation}\label{eq-defdiff}
\mc_{S^+_j}\varsubsetneqq \mc_{S^-_j} ~\mbox{ and }~  \mc_{S^-_j}\setminus \mc_{S^+_j}  \subseteq \ma\setminus\mc_j.\end{equation}
\end{definition}

When $j$ is $S$-differentiable, the set $S$ is said to be a \textit{separator set} of item $j$. 
An item $j$ is $S$-differentiable indicates that the items in the separator set $S$ can differentiate at least one incapable latent class of $j$ (i.e., one latent class in $\ma\setminus{\cal C}_j$) from the universal least capable class $\aaa_0$.

We need the following two conditions to establish identifiability.
\begin{enumerate}
\item[(C1)] \emph{Repeated Measurement Condition}: For each item $j$, there exist two disjoint sets of items $S_j^1$, $S_j^2\subset \ms\backslash\{j\}$ such that
$\mc_j\supseteq \mc_{S_j^1}$ and $\mc_j\supseteq \mc_{S_j^2}$.

\item[(C2)] \emph{Sequentially Differentiable Condition}: 
Start with the set $S_{sep}=S_{non}$. Expand $S_{sep}$ by including all items in $\ms\setminus S_{sep}$ that are $S_{sep}$-differentiable, and repeat the expanding procedure until no items can be added to $S_{sep}$. The sequentially expanding procedure ends up with $S_{sep}=\ms$. 
\end{enumerate}
Before presenting the formal theorem, we first give a simple illustration of how Condition (C2) can be checked.
 
\begin{example}
 \normalfont{
 Consider the following $3\times 4$  $\Gamma$-matrix,
\begingroup
\setlength{\belowdisplayskip}{0pt}
\begin{eqnarray*}
\Gamma = 
\begin{blockarray}{cccc}
\aaa_0 & \aaa_1 & \aaa_2 & \aaa_3 \\
\begin{block}{(cccc)}
0 & 0 & 1 & 1 \\
0 & 0 & 0 & 1 \\
0 & 1 & 0 & 0 \\
\end{block}
\end{blockarray}~,
\end{eqnarray*}
\endgroup then $\mc_1=\{\aaa_2,\aaa_3\}$, $\mc_2=\{\aaa_3\}$ and $\mc_3=\{\aaa_1\}$.
By \eqref{basis}, $S_{non}=\{2,3\}$ and $S_{basis}=\{1\}$.
To check condition (C2), we start with the separator set $S_{sep}=S_{non}=\{2,3\}$.
For basis item $1$, we define $S_1^+=\varnothing$ and $S_1^-=\{3\}$. Then $\mc_{S_1^+} = \{\aaa_0,\aaa_1,\aaa_2,\aaa_3\}$  and $\mc_{S_1^-}=\{\aaa_0,\aaa_2,\aaa_3\}$, so $\mc_{S_1^+}\setminus \mc_{S_1^-}=\{\aaa_1\}\subseteq\mc_1^c=\{\aaa_0,\aaa_1\}$, which means \eqref{eq-defdiff} holds for $j=1$. Besides, $S_1^+\cup S_1^- \subseteq S_{non}$. So by Definition \ref{def-diff}, item 1 is $S_{non}$-differentiable. 
Now we can expand the separator set $S_{sep}$ to be $S_{non}\cup\{1\}=\ms$. So the sequentially expanding procedure described in condition (C2) ends in one step with $S_{sep}=\ms$, and (C2) is satisfied.
}
 \end{example}

\vspace{-0.15in}
\begin{thm}
\label{thm1}
 Under the two-parameter restricted latent class models, 
condition (C1) is sufficient for  identifiability of $(\ttt^+,\ttt^-_{non})$, where $\ttt^-_{non} =(\theta^-_j,~j\in S_{non})$.
Moreover, conditions (C1) and (C2) are sufficient for $\pp$-partial identifiability of the model parameters $(\ttt^+,\ttt^-,\pp)$.
 \end{thm}

Theorem \ref{thm1} presents a general identifiability result with  strict identifiability being a special case. For instance, in the   case of $\mathcal A=\{0,1\}^K$, if the $J\times 2^K$  $\Gamma$-matrix is separable, then  $\nnu=\pp$ and the $\pp$-partial identifiability in Theorem \ref{thm1} exactly ensures strict identifiability of all the   parameters $(\ttt^+,\ttt^-,\pp)$. Similarly, in the case of $\mathcal A\subset\{0,1\}^K$, if the $J\times |\mathcal A|$ $\Gamma$-matrix  is separable, the $\pp$-partial identifiability ensures $(\ttt^+,\ttt^-)$ and $(p_{\aaa},\,\aaa\in\mathcal A)$ are strictly identifiable.
 Conditions (C1) and (C2) only depend on the structure of the $\Gamma$-matrix and are easily checkable.
Condition (C1)  implies that  at least one capable class  of each item is repeatedly measured by other items. 
Condition (C2) requires that for each basis item, at least one of its incapable classes should be differentiated from the universal least capable class  through a sequential procedure. 
From the proof of Theorem \ref{thm1},   (C1) suffices for identifiability of $(\ttt^+,\ttt^-_{non})$;  furthermore, the sequential procedure in condition (C2) ensures that as $S_{sep}$ sequentially expands its size, for any item $h$ included in $S_{sep}$, the parameter $\theta_h^-$ is  identifiable.
 If (C2) holds, i.e., the sequential procedure ends up with  $S_{sep} =\ms$,  we have the entire $\ttt^-$  identifiable, which further leads to  identifiability of $\nnu$. 
The sequential statement of (C2)  accurately 
characterizes the underlying structure of the  $\Gamma$-matrix needed for identifiability.
In particular, if there are no basis items, i.e., $\mathcal S = S_{non}$, then  (C2) automatically holds with zero  expanding step; while if there do exist basis items and each basis item is $S_{non}$-differentiable, then (C2) holds with one  expanding step.

The next proposition  further extends the result in Theorem \ref{thm1} to the case where the $\Gamma$-matrix may not satisfy (C1) and (C2). 
For any subset of items $S\subseteq \mathcal S$, define the \textit{$S$-adjusted $\Gamma$-matrix} $\Gamma(S)$ as follows, which has the same size as the original $\Gamma$. Its $j$th row $\{{\Gamma}(S)\}_{j,\Cdot}$ equals $\mo_m^\top-{\Gamma}_{j,\Cdot}$ if $j\in S$, and equals ${\Gamma}_{j,\Cdot}$ if $j\notin S$.
Here $\mathbf 1^\top_m$ denotes an all-one row vector of length $m$. 
\begin{proposition}\label{prop-after1}
 Consider a two-parameter restricted latent class model associated with a $\Gamma$-matrix. If there exist a subset of items $S\subseteq  \mathcal S$ such that the $S$-adjusted $\Gamma$-matrix ${\Gamma}(S)$
satisfies conditions (C1) and (C2), then the   two-parameter model is $\pp$-partially identifiable.
\end{proposition}
 
Proposition \ref{prop-after1}   relaxes the conditions of Theorem \ref{thm1}, by only requiring that (C1) and (C2) can be satisfied after switching the zeros and ones for some rows of in the $\Gamma$. 
The identifiability conditions in Theorem \ref{thm1} and Proposition \ref{prop-after1}    allow for a non-saturated latent class space $\ma$ and inseparability of the $\Gamma$-matrix, which relaxes the existing identifiability conditions in the literature.  
Moreover, the proposed conditions (C1) and (C2) would become necessary and sufficient in certain scenarios to be discussed in  the following subsection.

\subsection{Results for $Q$-Restricted Latent Class Models}\label{sec-2p-q}
To further illustrate the result in Theorem \ref{thm1}, we focus on the two-parameter $Q$-restricted latent class model with a saturated latent class space $\ma=\{0,1\}^K$. 
This includes the conjunctive DINA and disjunctive DINO models in Example  \ref{exp-dina} as special cases. 
Without loss of generality, we next only consider the two-parameter conjunctive model.
Nevertheless, all the  $\pp$-partial identifiability results presented in this subsection hold for both the conjunctive and the disjunctive models, due to the duality between them \citep{chen2015statistical}.

We introduce the following definitions adapted from Section \ref{sec-two-gamma}.
Under the conjunctive model assumption with $\mc_j$ taking the form of \eqref{eq-gamma-conj}, the  non-basis and basis items defined earlier in \eqref{basis} can be equivalently expressed in terms of the $\bq$-vectors as follows
\begin{equation}\label{eq-defdina}
S_{non} = \{j: \exists h\in \ms\setminus\{j\} \text{ s.t. } \bq_{h} \preceq \bq_j \} ~\mbox{ and }~ 
S_{basis} = \ms\setminus S_{non}.
\end{equation}
Moreover, item $j$ is set $S$-differentiable if there exist $S^+$, $S^-\subseteq S$ such that \begin{equation}\label{eq-defdiff-dina}
\mz\precneqq \vee_{h\in S^+}\bq_h - \vee_{h\in S^-}\bq_h\preceq \bq_j.
\end{equation}
In addition, conditions (C1) and (C2) are equivalent to:
\begin{enumerate}
\item[(C1$^*$)] \emph{Repeated Measurement Condition}: For each $j\in\ms$, there exist two disjoint item sets $S^1_j$, $S^2_j\subseteq \ms\setminus\{j\}$ such that 
$\bq_j\preceq \vee_{h\in S^1_j}\,\bq_h$ and $\bq_j\preceq \vee_{h\in S^2_j}\,\bq_h.$
\item[(C2$^*$)] \emph{Sequentially Differentiable Condition}: The same as condition (C2), but using definition \eqref{eq-defdiff-dina} of $S$-differentiable regarding the $\bq$-vectors.
\end{enumerate}

{Following Theorem \ref{thm1}, 
the next corollary shows that the derived  conditions on the $Q$-matrix suffice for the $\pp$-partial identifiability of both the conjunctive and  disjunctive two-parameter models.} 
\begin{corollary}\label{cor-dina-main}
Under the two-parameter $Q$-restricted latent class models, 
assuming $\nu_{[\aaa]}>0$ for any equivalence class $[\aaa]$, (C1$^*$) and (C2$^*$) are sufficient for the $\pp$-partial identifiability of $(\ttt^+,\ttt^-,\pp)$.
\end{corollary}

We use the following example as an illustration of the identifiability result; see also real data examples   in Section \ref{sec-2p-exp}. 
\begin{example}\label{exp-q53}
\normalfont{
Under the DINA model, consider the following   $Q$-matrix.
\begin{equation}\label{eq-q53}
Q=
\begin{pmatrix}
1 & 0 & 0 \\
0 & 1 & 0 \\
\hdashline[2pt/2pt]
1 & 1 & 1 \\
0 & 1 & 1 \\
1 & 0 & 1
\end{pmatrix}
\end{equation}
 This $Q$-matrix lacks the single-attribute item $(0,0,1)$, and the corresponding $\Gamma$-matrix under $\ma=\{0,1\}^3$ is inseparable.
In this case, we have   the following $7$ equivalence classes  $ \{[0,0,0]$, $[1,0,0]$, $[0,1,0]$, $[1,1,0]$, $[0,1,1]$, $[1,0,1]$, $[1,1,1]\},$
with the equivalence class $[0,0,0]$ containing attribute profiles $(0,\,0,\,0)$ and $(0,\,0,\,1)$, while each of the other equivalence classes   contains one attribute profile.
 Following the definition in   \eqref{eq-defdina}, items $1$ and $2$ are basis items, and items $3$, $4$ and $5$ are non-basis items.
For all the five items, condition (C1$^*$) is satisfied by taking $(S_1^1,~ S_1^2)=(\{3\},~ \{5\})$, $(S_2^1,~ S_2^2)=(\{3\}, ~\{4\})$, $(S_3^1, ~S_3^2)=(\{1,4\},~ \{2,5\})$, $(S_4^1, ~S_4^2)=(\{3\},~ \{2,5\})$, and $(S_5^1,~ S_5^2)=(\{3\}, ~\{1,4\})$.
In addition, condition (C2$^*$) is also satisfied since  the basis items $1$ and $2$ are $(S_1^+\cup S_1^-)$- and $(S_2^+\cup S_2^-)$-differentiable, respectively, where $(S_1^+, S_1^- )=(\{3\},\{4\})$ and $(S_2^+,S_2^-)=(\{3\},\{5\})$.  
By Corollary \ref{cor-dina-main}, the DINA model parameters are $\pp$-partially identifiable. 
}
\end{example}

 As shown above, conditions (C1$^*$) and (C2$^*$) are sufficient conditions to ensure $\pp$-partial identifiability. 
 In the following, we discuss the necessity of   (C1$^*$) and (C2$^*$) and further provide procedures to establish  identifiability in certain cases when these conditions  fail to hold.

 For a general $Q$-matrix, condition (C1$^*$) implies that each attribute is required by at least three items.
In the next theorem, we show that it is necessary for each attribute to be required by at least two items; in particular,  if some attribute is required by only two items, the  identifiability conclusion would depend on the structure of the $\bq$-vectors of those two items.

\begin{thm}[Discussion of C1$^*$] \label{thm_cond}
Consider a two-parameter $Q$-restricted latent class model. 
\begin{enumerate}
\item[(a)] If some attribute is required by only one item, then the model is not $\pp$-partially identifiable.
\item[(b)] If some attribute is required by only two items, without loss of generality, suppose the first attribute is required by the first two items  and the $Q$-matrix takes the following form
\begin{equation}\label{eq2v}
Q=
\begin{pmatrix}
1 & \vv_1^\top \\
1 & \vv_2^\top \\
\mathbf 0 & Q'
\end{pmatrix}_{J\times K},
\end{equation}
where $Q'$ is a $(J-2)\times (K-1)$ sub-matrix of $Q$ and $\vv_1$, $\vv_2$ are $(K-1)$-dimensional vectors. 
\begin{enumerate}
\item[(b.1)] If $\vv_1= \mz$ or $\vv_2= \mz$, the model   is not $\pp$-partially identifiable.
\item[(b.2)] If $\vv_1\neq \mz$ and $\vv_2\neq \mz$, the model is $\pp$-partially identifiable if the sub-matrix $Q'$ satisfies conditions (C1$^*$) and (C2$^*$), and either (a) or (b) below holds  for $i=1$ and $2$: (a) There exists some $j\geq 3$ such that $\bq_{j,\,2:K}\nsucceq \vv_i$; (b) There does not exist any $j\geq 3$ such that $\bq_{j,\,2:K}\nsucceq \vv_i$, and among the attributes required by $\vv_i$, there exists at least one attribute $k$ that is not required by every item $j\in\{ 3,\ldots,J\}$.
\end{enumerate}
\end{enumerate}
\end{thm}

{Theorem \ref{thm_cond} characterizes the different situations when condition (C1$^*$) fails to hold for some attribute, 
 and provides sufficient conditions for identifiability when the $Q$-matrix falls in the scenario (B).
In addition, the result in Theorem \ref{thm_cond} can be easily extended to the case where there are multiple attributes that are required by only two items. 
 
The next theorem discusses the necessity of Condition (C2$^*$) and states that
if there exists some basis item that does not have any separator set, then the model parameters are not $\pp$-partially identifiable.  
\begin{thm}[Discussion of C2$^*$]\label{thm-com-necc}
Under the two-parameter $Q$-restricted   models, 
the condition that each basis item $j$ is $\left(\ms\setminus\{j\}\right)$-differentiable, is necessary for the $\pp$-partial identifiability.
\end{thm}

 Furthermore, under the two-parameter $Q$-restricted   models with a separable $\Gamma$-matrix and a saturated latent class space $\ma$, the following theorem shows conditions (C1$^*$) and (C2$^*$) are exactly the minimal requirement for strict identifiability of the model.

\begin{thm}[Result on the Necessary and Sufficient Condition]\label{cor-suffnece}
Under the two-parameter $Q$-restricted   models, if $\ma$ is saturated and $\Gamma$ is separable, then conditions (C1$^*$) and (C2$^*$) are  {necessary and sufficient} for the strict identifiability of $(\ttt^+,\ttt^-,\pp)$.
\end{thm}

Under the assumptions of Theorem \ref{cor-suffnece}, conditions (C1$^*$) and (C2$^*$) are equivalent to the following explicit conditions on the structure of the $Q$-matrix: (C1$'$)  Each attribute is required by at least three items;  (C2$'$)   With $Q$ in the form $Q=(I_K^\top, (Q')^\top)^\top$, 
any two different columns of the submatrix $Q'$ are distinct.
Please see the proof of Theorem \ref{cor-suffnece} for details.

\subsection{Applications}\label{sec-2p-exp}

One important implication of the established identifiability theory is the consistent estimability of the model parameters.  
Consider a sample of size $N$ and denote the $i$th subject's multivariate binary responses by $\bo R_i=(R_{i,1},\ldots,R_{i,J})^\top$. 
Assume $\bo R_1,\ldots,\bo R_N$ identically and independently follow the categorical distribution with the probability mass function \eqref{eq-pmf}. 
The likelihood based on the sample can be written as
$
L(\TT,\pp\mid \bo R_1,\ldots,\bo R_N) = \prod_{i=1}^N \mathbb P(\RR=\bo R_i\mid \TT,\pp).
$
We denote the true parameters by $(\TT^0,\pp^0)$ and the maximum likelihood estimators (MLE) by $(\widehat\TT,\widehat\pp)$, which may not be unique. We further define the corresponding parameters $\nnu^0$ and $\widehat\nnu$ as in \eqref{eq-group}.
We have the following conclusion on the estimability of a two-parameter model.
\begin{proposition}\label{prop-con1}
If a two-parameter model is $\pp$-partially identifiable, then   $(\widehat\TT,\widehat\nnu)\to(\TT^0,\nnu^0)$ almost surely as $N\to\infty$. 
In addition, if $\Gamma$-matrix is also separable, then $(\widehat\TT,\widehat\pp)\to(\TT^0,\pp^0)$ almost surely. 
On the other hand,  if $\Gamma$-matrix is inseparable, $\pp$ can not be  consistently estimated. 
 \end{proposition}
 \color{black} With the consistency result, we can directly establish the asymptotic normality of  $(\widehat\TT,\widehat\nnu)$ when the  model is $\pp$-partially identifiable, following a standard argument of asymptotic statistics \cite{van2000asymptotic}.

We next apply the newly developed theory to the data examples  introduced in Section \ref{sec-real-q}, and establish the $\pp$-partial identifiability of the two-parameter restricted latent class model under the $Q$-matrices.

For the TOEFL iBT data introduced in Example \ref{exp-toefl}, the two-parameter restricted latent class models associated with the $Q$-matrices corresponding to reading forms A and B, denoted by $Q_A$ and $Q_B$ respectively, are both $\pp$-partially identifiable. Specifically, under the conjunctive DINA model, the $Q_A$ and $Q_B$ specified in Table \ref{tab-toefl} induce 14 and 12 equivalence classes of attribute profiles respectively, for which the sets of  representatives  are 
$\mathcal R^{Q_A} = \{0,1\}^4\setminus\{(0,0,0,1),\allowbreak~(1,0,0,1)\}$ and 
$\mathcal R^{Q_B} = \{0,1\}^4\setminus\{(0,0,0,1),\allowbreak~(1,0,0,1),\allowbreak~(0,1,0,1),\allowbreak~(1,1,0,1)\}$.
The $\mathcal R^{Q_A}$ and $\mathcal R^{Q_B}$ are calculated following the procedure introduced in Remark \ref{remark1}.
It is straightforward to check that for both $Q_A$ and $Q_B$, condition (C1$^*$) holds and there is no basis item, which further implies the satisfaction of  condition (C2$^*$). Therefore Corollary \ref{cor-dina-main} gives the $\pp$-partial identifiability of the two-parameter models associated with both $Q_A$   and $Q_B$.
Furthermore, Proposition \ref{prop-con1} implies the consistent estimability of  $(\ttt^+,\ttt^-,\nnu)$.  
In particular, 
 the proportion parameters of the   equivalence classes $\nnu=(\nu_{[\aaa]},\,\aaa\in\mathcal R^{Q_A})$ can   be consistently estimated, while the proportion parameters of   attribute profiles in a same equivalent class cannot.  
For instance,  under $Q_A$,   attribute patterns $\aaa^\star=(0,0,0,1)$ and $\aaa^{\star\star}=(0,0,0,0)$ share the same equivalent class; so $p_{\aaa^\star}$ and $p_{\aaa^{\star\star}}$ are not estimable, and it is only possible and meaningful to estimate  $\nu_{[\aaa^\star]} = p_{\aaa^\star}+p_{\aaa^{\star\star}}$.

Other than the TOEFL data, our new results in Section \ref{sec-2p-q} also guarantee the $\pp$-partial identifiability of two-parameter models associated with the $Q_{43\times 12}$ for the TIMSS data, and the $Q_{20\times 8}$ for the fraction subtraction data. The details of checking our conditions for $Q_{43\times 12}$ and $Q_{20\times 8}$ are included in Section A of the Supplementary Material.

\section{Identifiability Results for Multi-Parameter   Models}\label{sec-multi-para}

 This section considers multi-parameter restricted latent class models where each item $j$ allows for  more than two item parameters, i.e., $|\{\theta_{j,\aaa}:\aaa\in\ma\}| \geq 2$. 
In a multi-parameter model, those latent classes in $\mc_j$ still have the same level of positive response probability, according to the definition of $\mc_j$ in \eqref{eq-dcj}; however, the classes in $\ma\setminus\mc_j$ can have multiple levels of positive response probabilities, depending on the extents of their ``partial" capability of item $j$.
Examples of multi-parameter models include the Main-Effect and the All-Effect models introduced in Examples \ref{exp-main-eff} and \ref{exp-all-eff}, respectively.

We would like to point out that the $\Gamma$-matrix defined in \eqref{eq-def-gamma} still provides a useful technical tool for studying identifiability of multi-parameter models, despite the fact that the entry $\Gamma_{j,\aaa}$ only indicates whether $\aaa$ belongs to the most-capable-set $\mc_j$ and it does not summarize all the structural assumptions in multi-parameter models.

On the one hand,   similar to  the two-parameter case, under a multi-parameter   model, the  separability  of the $\Gamma$-matrix is  still \textit{necessary} for the {\it strict} identifiability of $(\TT,\pp)$. 
This is   because a two-parameter model,  such as   DINA,   can be viewed as a submodel of a multi-parameter model, such as GDINA or GDM, by constraining  certain parameters in the multi-parameter model to zero. 
So in order to ensure identifiability of all possible model parameters in the parameter space of  a multi-parameter model, Proposition \ref{prop1} implies the $\Gamma$ must be separable.

On the other hand, 
 when the $\Gamma$-matrix is inseparable and contains identical columns, the item parameter vectors associated with different latent classes may still be distinct. 
This is because under the general constraints \eqref{eq-constraints}, when $\Gamma_{j,\aaa}=0$ under a multi-parameter model,   $\aaa$ could be either least capable or partially capable of item $j$, and hence the latent classes in the set $\ma\setminus\mc_j=\{\aaa:\,\Gamma_{j,\aaa}=0\}$ can still have different  positive response probabilities, as shown in Examples \ref{exp-main-eff} and \ref{exp-all-eff}. 
 Such a difference from the two-parameter models makes the $\pp$-partial identifiability theory    developed  in Section \ref{sec-two-para} not applicable to multi-parameter models.
To study   identifiability  of  multi-parameter models when $\Gamma$ is inseparable, we therefore need an alternative partial identifiability notion and technique.  
We use the next example to illustrate this and show how the separable requirement of the $\Gamma$-matrix in Proposition \ref{prop1} could be relaxed under multi-parameter models.

\color{black}
\begin{example}\label{incomplete-eg}
\normalfont{
Consider  the   $Q$-matrix in \eqref{eq-cc}.
Under a two-parameter conjunctive restricted latent class model, we have shown  attribute profiles $\aaa_0 =(0,0)$ and $\aaa_2 =(0,1)$ are  not distinguishable. However, a multi-parameter model   models the main effect of each required attribute for an item. Consider the Main-Effect   model with the identity link function as introduced in Example \ref{exp-main-eff} (the ACDM), 
one has 
$\TT_{\Cdot, \aaa_0}  = (\beta_{1,0},~\beta_{2,0})^\top$ and  
$\TT_{\Cdot, \aaa_2}  = (\beta_{1,0},~\beta_{2,0}+\beta_{2,2})^\top;
$
then $\TT_{\Cdot, \aaa_0} \neq \TT_{\Cdot, \aaa_2}$ as long as $\beta_{2,2}\neq 0$.
When this inequality constraint $\beta_{2,2}\neq 0$ holds, $\TT_{\Cdot, \aaa_0} \neq \TT_{\Cdot, \aaa_2}$
despite that $\Gamma_{\Cdot,\aaa_0} = \Gamma_{\Cdot,\aaa_2}$.
In such scenarios, the grouping operation of the proportion parameters introduced in Section \ref{sec-two-para}
is not appropriate, and one needs to treat these two latent classes $\aaa_0$ and $\aaa_2$ separately. Consider any possible $\TT$ for which the inequality constraint $\beta_{2,2}\neq 0$ does not hold, then all such $\TT$ indeed fall into a subset of the parameter space $\mathcal T$ with smaller dimension than $\mathcal T$, characterized by $\mathcal V=\{(\TT,\pp): \beta_{2,2}= 0\}$. This implies that for almost all valid model parameters $(\TT,\pp)$ in $\mathcal T$, except a Lebesgue measure zero set $\mathcal V$, the $\TT$ satisfy $\TT_{\Cdot, \aaa_0} \neq \TT_{\Cdot, \aaa_1}$. This observation naturally leads to   the following notion of generic identifiability.
}
\end{example}

Motivated by Example \ref{incomplete-eg}, when the $\Gamma$-matrix is inseparable, we shall study the 
\textit{generic identifiability} of the restricted latent class model. 
Let ${\mathcal T}$ denote the restricted parameter space of $(\TT,\pp)$ under the general constraints \eqref{eq-constraints}, and let $d$ denote the number of free parameters in $(\TT,\pp)$, so ${\mathcal T}$ is of full dimension in $\mathbb R^d$. Generic identifiability means that identifiability holds for almost all points except a subset of  $ {\mathcal T}$ that has Lebesgue measure zero.
Generic identifiability is closely related to the concept of algebraic variety in algebraic geometry. 
Following the definition in \cite{allman2009}, an algebraic variety $\mathcal V$ is defined as the simulateneous zero-set of a finite collection of multivariate polynomials $\{f_i\}_{i=1}^n\subseteq \mathbb R[x_1,x_2,\ldots,x_d]$,
$
\mathcal V =\mathcal V(f_1,\ldots,f_n) = \{\boldsymbol x\in\mathbb R^d\mid f_i(\boldsymbol x) = 0,~ 1\leq i\leq n.\}
$
 An algebraic variety ${\mathcal V}$ is all of $\mathbb R^d$ only when all the polynomials defining it are zero polynomials; otherwise, ${\mathcal V}$ is called a \textit{proper subvariety} and is of dimension less than $d$, hence necessarily of Lebesgue measure zero in $\mathbb R^d$. The same argument holds when $\mathbb R^d$ is replaced by the parameter space ${\mathcal T}\subseteq\mathbb R^d$ that has full dimension in $\mathbb R^d$.
We next present the   definition of generic identifiability for restricted latent class models.

\begin{definition}[Generic Identifiability]\label{def-gen}
A restricted latent class model is said to be generically identifiable on the  parameter space ${\mathcal T}$, if  $(\TT,\pp)$ 
are strictly identifiable on $\mathcal T\setminus \mathcal V$ where $\mathcal V$ is a proper algebraic subvariety of $\mathcal T$.
\end{definition}

 Generic identifiability  could be viewed as some ``partial" identification of  model parameters in the sense that, the non-identifiable parameters fall in a subset of the parameter space that can be characterized as solutions to some nonzero polynomial equations.
 As can be seen from the form of \eqref{eq-constraints}, the constraints on the parameter space introduced by the $\Gamma$-matrix already force the parameters fall into a proper algebraic subvariety of the unrestricted parameter space, 
 so previous results established in \cite{allman2009} for unrestricted latent class models do not apply to the models considered in this work. 

\begin{remark}\label{rmk-mult-pid}
Under  multi-parameter models, it is still possible that two latent classes  $\aaa$ and $\aaa'$  always have the same positive response probabilities, i.e.,  $\TT_{\Cdot,\aaa} = \TT_{\Cdot,\aaa'}$ and $\aaa$, $\aaa'$ are not distinguishable even generically. In this case one could have $\pp$-partial identifiability of the model.
However, this happens only when $\Gamma_{\Cdot, \aaa}= \Gamma_{\Cdot, \aaa'} =\mathbf 1$; moreover, under $Q$-restricted models,  
  this happens only if the  $Q$-matrix contains an all-zero column, which is a trivial case with a redundant column in $Q$. 
Under such a   $Q$-matrix,   we can simply remove these all-zero columns and study the (generic) identifiability under the reduced $Q$-matrix. Therefore, without loss of generality, in the following discussion we assume the  $Q$-matrix does not contain any all-zero column such that $\TT_{\Cdot,\aaa} = \TT_{\Cdot,\aaa'}$ would not happen.
\end{remark}
 
Based on the above discussions, to study  identifiability of multi-parameter restricted latent class models, we consider two   situations in Section \ref{sec-gen}: first, when the $\Gamma$-matrix is separable, we study the \textit{strict} identifiability of model parameters; second, when the $\Gamma$-matrix is inseparable, we study the \textit{generic} identifiability of  model parameters. 
Furthermore, in Section \ref{sec-gen-q} we present sufficient conditions for generic identifiability of the   family of $Q$-restricted latent class models, and discuss the necessity of the proposed conditions.

\color{black} 

\subsection{Strict and Generic Identifiability}\label{sec-gen}
First consider the case where the $\Gamma$-matrix is separable.
For a subset of items $S$, denote the corresponding $|S|\times m$ indicator matrix by $\Gamma^S=(\Gamma_{j,\aaa},~j\in S,~\aaa\in\ma)$, which is a submatrix of the previously defined $\Gamma$-matrix. 
We say  $\aaa$ succeeds $\aaa'$ with respect to $S$ and denote it by $\aaa\succeq_{S}\aaa'$, if
$ \Gamma_{j,\aaa}\geq\Gamma_{j,\aaa'}$ for any $j\in S;$
this means $\aaa$ is at least as capable as $\aaa'$ of items in set $S$. 
With this definition, any subset of items $S$ induces a partial order ``$\succeq_{S}$" on the set of latent classes $\ma$.
When two sets $S_1$ and $S_2$  induce the same partial order on $\ma$, that is, for any $\aaa'$ and $\aaa\in\ma$, 
$
\aaa'\succeq_{S_1}\aaa  \text{ if and only if }  \aaa'\succeq_{S_2}\aaa,
$     we write $``\succeq_{S_1}" = ``\succeq_{S_2}"$.
The following theorem gives conditions that lead to strict identifiability of multi-parameter restricted latent class models.
\begin{thm}\label{thm-order}
For a multi-parameter restricted latent class model,
if the $\Gamma$-matrix satisfies the following conditions, then the parameters $(\TT,\pp)$ are strictly identifiable.
\begin{enumerate}
\item[(C3)] There exist two disjoint item sets $S_1$ and $S_2$, such that $\Gamma^{S_i}$ is separable for $i=1,2$ and $``\succeq_{S_1}" = ``\succeq_{S_2}"$.

\item[(C4)] $\Gamma^{(S_1\cup S_2)^c}_{\Cdot,\aaa}\neq \Gamma^{(S_1\cup S_2)^c}_{\Cdot,\aaa'}$ for any $\aaa$, $\aaa'$ such that $\aaa'\succeq_{S_i} \aaa$ for $i=1$ or $2$.
\end{enumerate}
\end{thm}
Condition (C3) implies the entire $\Gamma$-matrix is separable, and it requires two disjoint sets of items $S_1$ and $S_2$  to have enough information to distinguish the latent classes, and it serves as a \textit{Repeated Measurement Condition} for  the  identifiability of multi-parameter restricted latent class models. Condition (C4) states that, for those pairs of latent classes $\aaa$ and $\aaa'$ such that $\aaa$  is  more capable than $\aaa'$ uniformly on either $S_1$ or $S_2$, the remaining items in $(S_1\cup S_2)^c$ should differentiate $\aaa$ and $\aaa'$ by  their column vectors in $\Gamma^{(S_1 \cup S_2)^c}$.

Strict identifiability can be achieved with a relaxation of Condition (C4) together with a stronger version of Condition (C3).
Before presenting this result, 
we define a latent class $\aaa$ as a basis latent class under an item set $S$,  if there does not exist
$\aaa'\in \ma$ such that
$\aaa'\preceq_{S}\aaa$.
Denote the set of all basis latent classes under $S$ by $\mb_{S}$.
Then ``$\succeq_{S_1}" = ``\succeq_{S_2}"$ implies $\mb_{S_1} = \mb_{S_2}$.

\begin{proposition}
\label{new}
Under a multi-parameter restricted latent class model,
if the $\Gamma$-matrix satisfies the following conditions, then $(\TT,\pp)$ are identifiable.
\begin{enumerate}
\item[(C3$^*$)] 
There exist two disjoint item sets $S_1$ and $S_2$, such that $\Gamma^{S_i}$ is separable for $i=1,2$ and $``\succeq_{S_1}" = ``\succeq_{S_2}"$.
Moreover, for any $j\in S_1\cup S_2$, there exists $\aaa\in \mb_{S_1}$ such that $\Gamma_{j,\aaa}=1$.

\item[(C4$^*$)] $\Gamma^{(S_1\cup S_2)^c}_{\Cdot,\aaa}\neq \Gamma^{(S_1\cup S_2)^c}_{\Cdot,\aaa_0}$ for any $\aaa\in \mb_{S_1}$ and $\aaa\neq\aaa_0$, where $\aaa_0$ is the universal least capable class.
\end{enumerate}
\end{proposition}

\begin{remark}
Theorem \ref{thm-order} and Proposition \ref{new} show the trade-off between the conditions on the separable submatrices part of $\Gamma$ and   on the remaining part. They establish identifiability  for a wide range of restricted latent class models, with the $\Gamma$-matrix ranging in the spectrum of different extents of inseparability. Specifically, for a  $Q$-restricted latent class model that lacks many single-attribute items,  (C3) is easier to satisfy than (C3$^*$)   and Theorem \ref{thm-order} would be more applicable; while for a $Q$-restricted  model     that   lacks few single-attribute items, Proposition \ref{new} would become more applicable as (C4$^*$) imposes a weaker condition on the set $(S_1\cup S_2)^c$.
\end{remark}

\begin{remark}\label{rmkstrctid}
 Theorem \ref{thm-order} and Proposition \ref{new} extend the existing work \cite{xu2017}. 
 Compared with the identifiability result in \cite{xu2017} that requires two copies of the identity submatrix $I_K$ to be included in the $Q$-matrix, in the special case  with $\ma = \{0,1\}^K$, the   proposed conditions (C3$^*$) and (C4$^*$)   reduce to the conditions in \cite{xu2017}. Furthermore, in general cases of an unsaturated latent class space with $|\ma|<2^K$, the conditions in Theorem \ref{thm-order} and Proposition \ref{new} impose much weaker requirements   than those in \cite{xu2017}, because a $Q$-matrix lacking some single-attribute items may suffice for a separable $\Gamma$-matrix and further suffice for strict identifiability.
 \end{remark}
 
Next, we consider the case where the multi-parameter restricted latent class model is associated with an inseparable $\Gamma$-matrix, which violates condition (C3). We study the generic identifiability of the model parameters.

\begin{thm}\label{thm-order-gen}
Consider a multi-parameter restricted latent class model. If there exist two disjoint item sets $S_1$ and $S_2$, such that altering some entries of zero to one in $\Gamma^{S_1\cup S_2}$ can yield a $\widetilde\Gamma^{S_1\cup S_2}$ that satisfies Condition (C3); and that the $\Gamma^{(S_1\cup S_2)^c}$ satisfies condition (C4),
then the model parameters $(\TT,\pp)$ under the original $\Gamma$-matrix are generically identifiable.
\end{thm}

  Theorem \ref{thm-order-gen} is established based on the theoretical development of Theorem  \ref{thm-order}. By relaxing the condition (C3) and allowing $\Gamma$ to be inseparable, we may not have strict identifiability, as discussed in Example  \ref{incomplete-eg}. We   use the following example to further illustrate the  results of Theorems  \ref{thm-order}--\ref{thm-order-gen}.

\begin{example}\label{exp-multi-hcdm}

\normalfont{
For  a multi-parameter restricted latent class model, if  $\Gamma = ((\Gamma^{sub})^\top, \, (\Gamma^{sub})^\top,\, (\Gamma^{sub})^\top)^\top$ contains three copies of the following $\Gamma^{sub}$, then  (C3) and (C4) are satisfied and $(\TT,\pp)$ under $\Gamma$ are strictly identifiable.  
$$
\Gamma^{sub}=
\begin{pmatrix}
0 &  1 & 1 & 1 \\
0 &  0 & 1 & 1 \\
0 &  0 & 0 & 1 \\
\end{pmatrix};\quad
\Gamma^{S_1}=
\begin{pmatrix}
0 &  \mathbf 0 & 1 & 1 \\
0 &  0 & 1 & 1 \\
0 &  0 & 0 & 1 \\
\end{pmatrix},\quad
\Gamma^{S_2}=
\begin{pmatrix}
0 &  1 & 1 & 1 \\
0 &  0 & \mathbf 0 & 1 \\
0 &  0 & 0 & 1 \\
\end{pmatrix}.
$$
Instead, consider $\Gamma_{\text{new}}=((\Gamma^{S_1})^\top, \, (\Gamma^{S_2})^\top,\, (\Gamma^{sub})^\top)^\top$ with two submatrices in the forms of $\Gamma^{S_1}$ and $\Gamma^{S_2}$ above, then neither of $\Gamma^{S_i}$ is separable. But by changing the $(1,2)$th entry of $\Gamma^{S_1}$ and $(2,3)$th entry of $\Gamma^{S_2}$ from zero to one, the resulting $\widetilde \Gamma^{S_1}$ and $\widetilde \Gamma^{S_2}$ are separable, so the conditions of Theorem \ref{thm-order-gen} are satisfied and   $(\TT,\pp)$ under $\Gamma_{\text{new}}$ are generically identifiable.
} 
\end{example}

\subsection{Results for $Q$-Restricted Latent Class Models}\label{sec-gen-q}
In this subsection we characterize how the $Q$-matrix    impacts the    identifiability of multi-parameter  models.
Similarly to Section \ref{sec-2p-q}, we consider the case $\ma=\{0,1\}^K$.
For strict identifiability, the result of either Theorem \ref{thm-order} or Proposition \ref{new} implies the result  of Theorem 1  in \cite{xu2017}, as discussed in Remark \ref{rmkstrctid}.
 Our next result gives a flexible structural condition on  $Q$ that leads to generic identifiability. 
\begin{thm}
\label{thm-gen-q}
Under a multi-parameter $Q$-restricted latent class model, 
if the $Q$-matrix satisfies the following conditions, then the model parameters are generically identifiable, up to label swapping among those latent classes that have identical column vectors in $\Gamma$.

\begin{enumerate}
\item[(C5)] $Q$ contains two $K\times K$ sub-matrices $Q_1$, $Q_2$, such that for $i=1, 2$, 
\begin{equation}\label{eq-diag}
Q = \begin{pmatrix}
Q_1 \\
Q_2 \\
Q'
\end{pmatrix}_{J\times K};\quad
Q_i =
\begin{pmatrix}
    1 & * & \dots  & * \\
    * & 1 & \dots  & * \\
    \vdots & \vdots & \ddots & \vdots \\
    * & * & \dots  & 1
\end{pmatrix}_{K\times K},\quad i=1,2,
\end{equation}
where each `$*$'   can be either zero or one.
\item[(C6)] 
With the $Q$-matrix taking the form of \eqref{eq-diag}, in the submatrix $Q'$ each attribute is required by at least one item.
\end{enumerate}
\end{thm}

The above identifiability result  does not require the $Q$ to contain an identity submatrix $I_K$ and provides a flexible new condition for generic identifiability that are satisfied by various $Q$-matrix structures; see   examples in Section \ref{sec-gen-app}. 
Under a multi-parameter restricted latent class model with all entries of the $Q$-matrix being ones,   conditions (C5) and (C6) in Theorem \ref{thm-gen-q} equivalently reduce to $J\geq 2K+1$, which is consistent with the result in \cite{allman2009} for unrestricted latent class models.

Next we discuss the necessity of the proposed sufficient conditions for generic identifiability.
Conditions (C5) and (C6)  imply that each attribute is required by at least three items.  
The next theorem  shows that it is necessary for each attribute to be required by at least two items. 

\begin{thm}\label{prop-vio-gen}
Consider a multi-parameter $Q$-restricted latent class model.
\begin{itemize}
\item[(a)]
If some attribute is required by only one item, then the model is not generically identifiable. 

\item[(b)]
If some attribute is required by only two items, without loss of generality assume $Q$ takes the following form 
\begin{equation}\label{eq-gen-v1v2}
Q = 
\begin{pmatrix}
1 & \vv_1^\top \\
1 & \vv_2^\top \\
0 & Q'
\end{pmatrix},
\end{equation}
then as long as $\vv_1\vee\vv_2\neq \mo_{K-1}$ and the sub-matrix $Q'$ satisfies conditions (C5) and (C6), then the model parameters $(\TT,\pp)$ are generically identifiable, up to label swapping among those latent classes that have identical column vectors in $\Gamma$.

\end{itemize}

\end{thm}

\begin{remark}\label{rmk-after-thm8}
As a notion of partial identification of model parameters, generic identifiability does not imply strict identifiability. For instance, if the $Q$-matrix is in the form of \eqref{eq-gen-v1v2} and $\vv_i=\mz$ for $i=1$ and $2$, then the model is not strictly identifiable, but generic identifiability can still hold as stated in Theorem \ref{prop-vio-gen}. This is also an analogue to the situations discussed in Theorem \ref{thm_cond} for two-parameter restricted latent class models.
Based on Theorems \ref{thm-gen-q} and \ref{prop-vio-gen},  
we would recommend  practitioners in diagnostic test designs to ensure each attribute is measured by {at least three items}.
\end{remark}

\subsection{Applications}\label{sec-gen-app}
Similar to the discussion in Section \ref{sec-2p-exp}, 
our results of generic identifiability also lead to the estimability of the model parameters. 
\begin{proposition}\label{prop-con2}
Suppose a restricted latent class model is generically identifiable on the parameter space $\mathcal T$ with a measure-zero non-identifiable set $\mathcal V$. If the true parameters $(\TT^0,\pp^0)\in\mathcal T\setminus\mathcal V$, then $(\widehat\TT,\widehat\pp)\to(\TT^0,\pp^0)$ almost surely as $N\to\infty$.
\end{proposition}

We apply the new theory of generic identifiability to the designs introduced in Section \ref{sec-real-q}, and establish  generic identifiability of the multi-parameter restricted latent class models. 
Consider the TOEFL iBT Data.
Both $Q$-matrices corresponding to TOEFL reading forms A and B can be transformed into the form of \eqref{eq-diag} through some row rearrangements, with the corresponding  $Q'$  requiring each attribute at least once. Therefore  both $Q$-matrices  satisfy conditions (C5) and (C6) and any multi-parameter $Q$-restricted models  associated with them are generically identifiable and estimable.
Our   results in this section also guarantee the generic identifiability of multi-parameter models associated with the $Q_{43\times 12}$ for the TIMSS data, and the $Q_{20\times 8}$ for the fraction subtraction data; please see Section A in the Supplementary Material for details of checking the conditions.

\section{Extensions to More Complex Models}
\label{sec-extend}
In this section, we   extend   our identifiability theory   to some more complicated latent variable models.

\subsection{Mixed-items Restricted Latent Class Models}\label{sec-mixed}
Our identifiability theory based on $\Gamma$ directly applies to the case of mixed types of items, where the $J$ items can conform to different models, including two-parameter conjunctive, two-parameter disjunctive, or multi-parameter.

First consider the \textit{two-parameter-mixed} restricted latent class model, where each item is either two-parameter conjunctive or disjunctive. For any $Q$-matrix and   latent class space $\ma$,  denote the $\Gamma$-matrix under the two-parameter conjunctive model  by $\Gamma^{conj}(Q,\ma)$, and that under the two-parameter disjunctive model by $\Gamma^{disj}(Q,\ma)$.
The following is a corollary of Theorem \ref{thm1}.
\begin{corollary}\label{prop-ao}
Consider a two-parameter-mixed restricted latent class model with $Q=(Q_{disj}^\top, \allowbreak Q_{conj}^\top)^\top$,
where  $Q_{disj}$ and $Q_{conj}$ correspond to disjunctive and conjunctive items, respectively. 
If the following condition (E1) holds, then $(\ttt^+,\ttt^-,\pp)$ are $\pp$-partially identifiable.
\begin{itemize}
\item[(E1)] The $J\times |\ma|$ matrix $\Gamma=(\Gamma^{disj}(Q_{disj},\ma)^\top,~ \allowbreak\Gamma^{conj}(Q_{conj},\ma)^\top)^\top$
satisfies conditions (C1) and (C2) in Theorem \ref{thm1}.
\end{itemize}
In particular, if $\ma=\{0,1\}^K$ and the $\Gamma$ defined in (E1) is separable, then $(\ttt^+,\ttt^-,\pp)$ are strictly identifiable.
\end{corollary}

One implication of Corollary \ref{prop-ao} is that when a diagnostic test contains both conjunctive and disjunctive items, the underlying $Q$-matrix does not need to 
include a submatrix $I_K$
for $(\ttt^+,\ttt^-,\pp)$ to be strictly identifiable. This is in contrary to the case of a purely conjunctive or purely disjunctive two-parameter model, where this requirement is indeed necessary \cite{Xu15,id-dina}.
The following application of Corollary \ref{prop-ao} illustrates this point.
\begin{example}\label{ex-22}
\normalfont{
Consider a diagnostic test with 4 conjunctive items and 2 disjunctive items with the following  $Q$-matrix
\begingroup
\setlength{\belowdisplayskip}{0pt}
\begin{align*}
Q = 
\begin{pmatrix}
Q^{conj}_{ 4\times 2} \\
\hline
Q^{disj}_{ 2\times 2}
\end{pmatrix}
=
\begin{pmatrix}
1 & 0\\
1 & 1\\
1 & 1\\
1 & 1\\
\hline
1 & 1\\
1 & 1\\
\end{pmatrix}
~\Rightarrow~
\Gamma
= 
\begin{blockarray}{cccc}
(0,0) & (0,1) & (1,0) & (1,1)\\
\begin{block}{(cccc)}
0 & 0 & 1 & 1\\
0 & 0 & 0 & 1\\
0 & 0 & 0 & 1\\
0 & 0 & 0 & 1\\
\cline{1-4}
0 & 1 & 1 & 1\\
0 & 1 & 1 & 1\\
\end{block}
\end{blockarray}~.
\end{align*}
\endgroup
Then if $\ma=\{0,1\}^2$, the corresponding $\Gamma$-matrix as shown above is separable, and conditions (C1$^*$) and (C2$^*$) are satisfied. So $\ttt^+=(\theta^+_1,\ldots,\theta^+_6)^\top$, $\ttt^-=(\theta^-_1,\ldots,\theta^-_6)^\top$ and $\pp=(p_{(0,0)},~p_{(0,1)},~p_{(1,0)},~p_{(1,1)})^\top$ are strictly identifiable, despite that $Q$ does not contain a submatrix $I_2$.
}
\end{example}

If there exist both two-parameter items and multi-parameter items in the model, we have the following identifiability result, the part (a) of which directly results from Theorem  \ref{thm-order} and Proposition \ref{new}. 
Please see Section D in the Supplementary Material for   details.

\begin{corollary}\label{prop-mix-multi} 
Assume 
$Q =(Q_{disj}^\top,\allowbreak Q_{conj}^\top,\allowbreak Q_{mult}^\top)^\top$
where $Q_{disj}$, $Q_{conj}$ and $Q_{multi}$  correspond to the two-parameter disjunctive, two-parameter conjunctive, and multi-parameter items, respectively. 
\begin{itemize}\setlength\itemsep{1mm}
\item[(a)]
If  $\Gamma = (\Gamma^{disj}(Q_{disj},\ma)^\top,~ \allowbreak\Gamma^{conj}(Q_{conj},\ma)^\top,~ \allowbreak\Gamma^{conj}(Q_{mult},\ma)^\top)^\top$ satisfies conditions (C3) and (C4) in Theorem \ref{thm-order}; or conditions (C3*) and (C4*) in Proposition \ref{new}, then $(\TT,\pp)$ are strictly identifiable.
\item[(b)]
If    $\Gamma$ satisfies condition  (E2) in Section D of the Supplementary Material, then $(\TT,\pp)$ are generically identifiable.
\end{itemize}
\end{corollary}

\subsection{Restricted Latent Class Models with Categorical Responses}\label{sec-cat}
 We next study restricted latent class models with multiple levels of  responses per item, i.e., categorical responses, instead of binary responses considered in previous sections. These models have been considered in \cite{davier2008general}, \cite{ma2016seq} and \cite{chen2018poly}. 
We consider the setting in \cite{chen2018poly}. Suppose for each item $j$ out of the $J$ items in a diagnostic test, there are $L_j$ categories of responses.
For each item $j$ and each category of response $l\in\{0,\ldots,L_j-1\}$, there are a set of positive response parameters of the latent classes $\ttt^{(l)}_j=\{\theta^{(l)}_{j,\aaa}:\,\aaa\in\mathcal A\}$ with   $\ttt^{(0)}_{j}=\mathbf 1-\sum_{l>0}\ttt^{(l)}_{j}$. Further, for each item $j$, the $\bq$-vector $\bq_j$ constrains the vector $\ttt^{(l)}_j$ based on \eqref{eq-constraints} for each category $l\in\{1,\ldots,L_j-1\}$ independently, other than the basic level $l=0$. Namely, for any $j\in\ms$,
\begin{align*}
\max_{\aaa\in\mc_j} \theta^{(l)}_{j, \aaa} 
=&~ \min_{\aaa\in\mc_j} \theta^{(l)}_{j, \aaa} 
> \theta^{(l)}_{j,\aaa'},
\quad \forall l\in\{1,\ldots,L_j-1\}~\text{and}~ \forall \aaa'\notin\mc_j.
\end{align*}
We collect all the model parameters in $(\TT^{\text{cat}},\pp)$ with $\TT^{\text{cat}}=\{\ttt_j^{(l)}:\,j=1,\ldots,J;\,l=0,\ldots,L_j-1\}$. 
Then we have the following identifiability result.

\begin{proposition}\label{prop-poly}
For a given $Q$-matrix, consider the following cases.
\begin{itemize}\setlength\itemsep{1mm}
\item[(a)] 
If for any $j\in\ms$ and   $l\in\{1,\ldots,L_j\}$, item parameters  $\{\theta_{j,\aaa}^{l},\aaa\in\ma\}$ follow the two-parameter   assumption, and $Q$ satisfies   (C1*) and (C2*) in Corollary \ref{cor-dina-main}, then $(\TT^{\text{cat}},\pp)$ are $\pp$-partially identifiable. 
\item[(b)] 
If for any $j\in\ms$ and   $l\in\{1,\ldots,L_j\}$, item parameters  $\{\theta_{j,\aaa}^{l},\aaa\in\ma\}$ follow the multi-parameter assumption, and $Q$ satisfies conditions (C5) and (C6) in Theorem \ref{thm-gen-q}, then $(\TT^{\text{cat}},\pp)$ are generically identifiable. 
\end{itemize}
\end{proposition}

\subsection{Deep Restricted Boltzmann Machines.}\label{sec-rbm}
Restricted latent class models share great similarities with Restricted Boltzmann Machines (RBM) \citep{goodfellow2016deep}.
We use a simple example to illustrate how the RBM architecture can be used as a special restricted latent class model for cognitive diagnosis. The RBM on the right panel of Figure \ref{fig-rbm} consists of two latent layers $\aaa^{(1)}$ and $\aaa^{(2)}$ and one observed layer $\RR$. In a diagnostic test, the $\RR$ represents multivariate binary responses to test items, the first latent layer $\aaa^{(1)}$  represents the  fine-grained binary skill attributes measured by the items, while the second binary latent layer $\aaa^{(2)}$ helps to model the dependence among $\aaa^{(1)}$  and may be interpreted as more general skill domains. Denote the lengths of vectors $\RR$, $\aaa^{(1)}$ and $\aaa^{(2)}$ by $J$, $K_1$ and $K_2$.
Under RBM assumptions, the probability distribution of all the observed and latent variables is
\begin{align}\label{eq-rbm}
\mathbb P(\RR, \aaa^{(1)},\aaa^{(2)}) = \frac{1}{Z}\exp\Big( 
&- \RR^\top \bo W^Q \aaa^{(1)}  - (\aaa^{(1)})^\top \bo U \aaa^{(2)} \Big),
\end{align}
where $Z$ is the normalization constant, and $\bo W^Q $, $\bo U$ are parameter matrices, of size $J\times K_1$ and $K_1\times K_2$, respectively. We drop the bias terms in the above energy function without loss of generality \citep{goodfellow2016deep}.
We can impose a $Q$-matrix of size $J\times K_1$ to restrict the parameters $\bo W^Q$ in \eqref{eq-rbm}. Specifically, $Q$ specifies which entries of $\bo W^Q=(w_{j,k})$ are zero, i.e., $w_{j,k}= 0$ if $q_{j,k}=0$. The form of  $Q$ underlying the $\bo W^Q$ in Figure \ref{fig-rbm} is on the left panel of the figure.
\begin{figure}[h!]
\centering
\begin{minipage}{0.4\textwidth}
\normalsize
$$
Q=\begin{pmatrix}
1 & 0 & 1 & 0 \\
1 & 1 & 0 & 0 \\
0 & 1 & 1 & 0 \\
0 & 0 & 1 & 1 \\
0 & 1 & 0 & 1 \\
\end{pmatrix};
$$
\end{minipage}
\begin{minipage}{0.6\textwidth}
\centering
\includegraphics[width=0.9\textwidth]{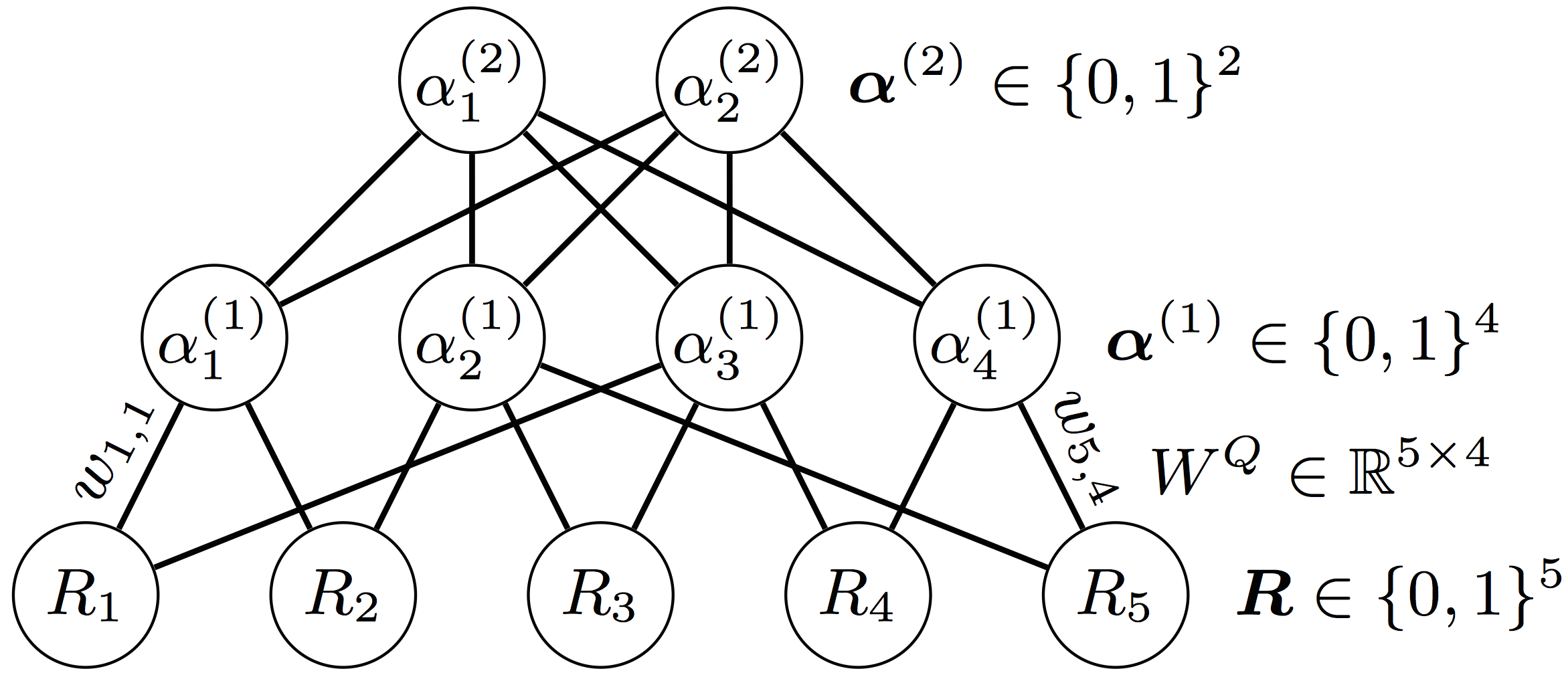}
\end{minipage}
\caption{(Deep) Restricted Boltzmann Machine}
\label{fig-rbm}
\end{figure}

We call $\bo W^Q$ the \textit{item parameters} of a RBM, since these parameters relate to the observed responses to items; and call a RBM with a $Q$-matrix  structure an \textit{item-parameter-restricted} RBM.
Then an item-parameter-restricted RBM can be viewed as a multi-parameter main-effect restricted latent class model, with $\aaa^{(1)}$ belonging to the latent class space $\{0,1\}^{K_1}$. 
The  next proposition establishes  identifiability of the item parameters $\bo W^Q$.

\begin{proposition}\label{prop-rbm}
For a given $Q$-matrix, consider the following cases.
\begin{itemize}
\item[(a)] If there is no sparsity structure in $\bo W^Q$ (i.e.,   $Q={\mathbf 1}_{J\times K}$), then as long as $J\geq 2K_1+1$, the item parameters $\bo W^Q$ are generically identifiable.
\item[(b)] If the $Q$-matrix satisfies the sufficient conditions for strict  or generic identifiability in Section \ref{sec-multi-para}, then $\bo W^Q$ are strictly or generically identifiable, respectively.
\end{itemize}
\end{proposition}
Proposition \ref{prop-rbm} establishes  identifiability of the item parameters $\bo W^Q$, which provides the theoretical guarantee  in the application of item calibration to assess the quality of the items. 
It would also be interesting to further investigate identifiability of other parameters besides the item parameters in a deep restricted Boltzmann machine, which we leave for future study.

\section{Discussion}\label{sec-disc}
This paper proposes a general framework of strict and partial identifiability of restricted latent class models.

We provide a flowchart in Figure \ref{fig-flowchart} to summarize our main theoretical results in Sections \ref{sec-two-para} and \ref{sec-multi-para}. The flowchart illustrates how to apply the new theory in cognitive diagnosis. Specifically, given the specification of the $Q$-matrix,  the latent class space $\ma\subseteq\{0,1\}^K$, and the diagnostic model assumptions, one can   construct the corresponding $J\times |\ma|$ $\Gamma$-matrix based on the $\mc_j$'s defined in \eqref{eq-dcj}. 
Then  in the case of a separable $\Gamma$-matrix, if the model is two-parameter, the  $\pp$-partial identifiability exactly reduces to strict identifiability  and one can   use results in Section \ref{sec-two-para} to establish strict identifiability; and if the model is multi-parameter, one can use results Theorem \ref{thm-order} and Proposition \ref{new} in Section \ref{sec-gen} for strict identifiability.
On the other hand, if the  $\Gamma$-matrix is inseparable, depending on whether the model is two-parameter or multi-parameter, one can use the results in Section \ref{sec-2p-q} or those in Section \ref{sec-multi-para}  to check whether the model is $\pp$-partially identifiable or generically identifiable, respectively.
Note that in the special case of  $\ma=\{0,1\}^K$,   the $\Gamma$-matrix with $2^K$ columns  is separable if and only if the $Q$-matrix contains an identity submatrix $I_K$, a key condition assumed in previous works \citep[e.g.,][]{xu2017,xu2018jasa}. 
Hence, this work not only largely relaxes these  existing conditions for strict identifiability by allowing more flexible attribute structures with an arbitrary $\ma$,   but also provides the first study on partial identifiability when the $Q$-matrix does not include an $I_K$ (the $\Gamma$-matrix is inseparable).
We give easily-checkable identifiability conditions to ensure estimability of the model parameters, and these conditions serve as practical guidelines for designing statistically valid diagnostic tests.

\begin{figure}[h!] 
\centering
\includegraphics[width=0.95\textwidth]{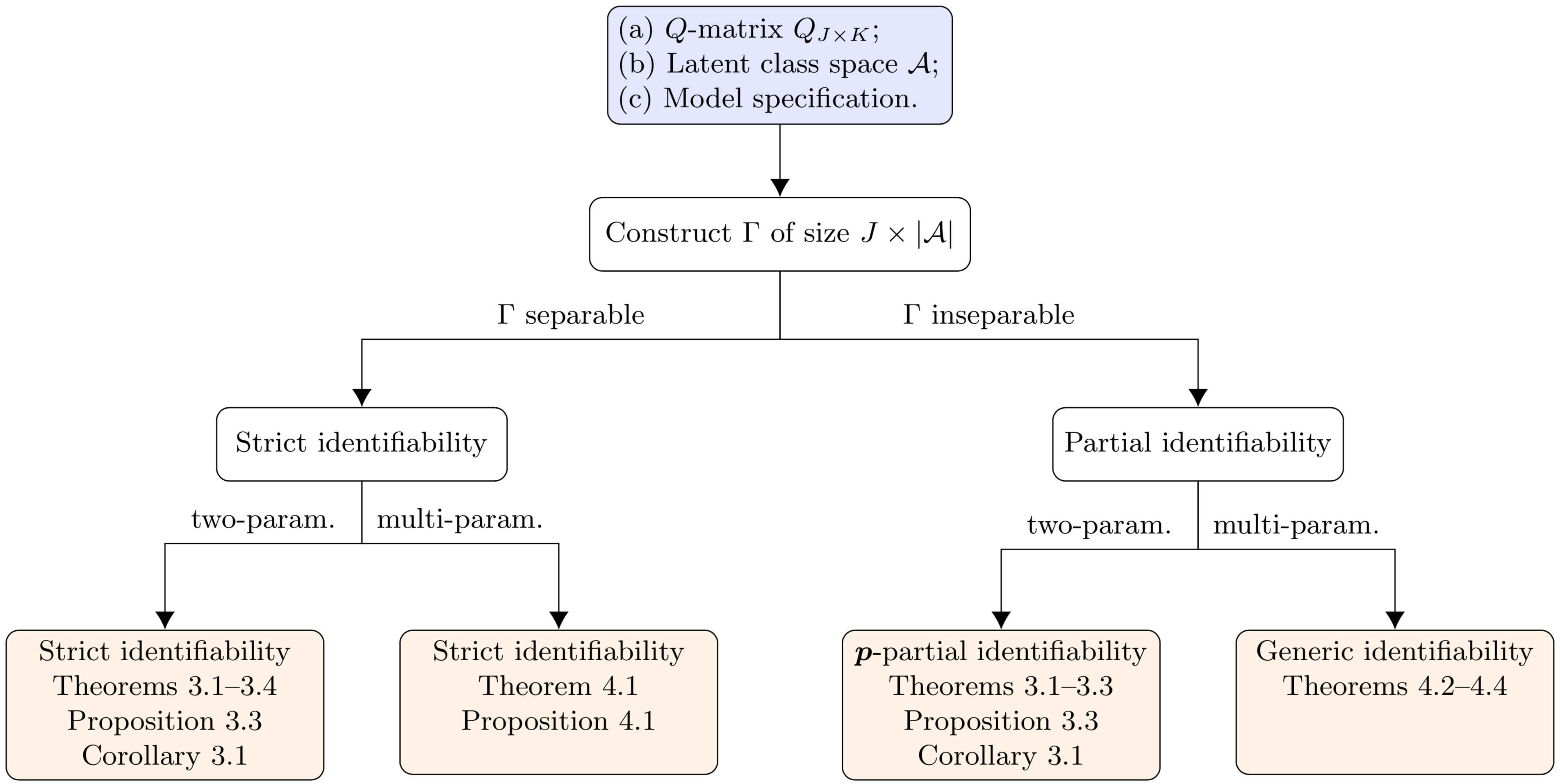}
\caption{Flowchart of the results in Sections \ref{sec-two-para} and \ref{sec-multi-para}}
\label{fig-flowchart}
\end{figure}

We point out that the strict identifiability results in Section \ref{sec-gen} (Theorem \ref{thm-order} and Proposition \ref{new})  apply to  the general family of restricted latent class models   satisfying constraints \eqref{eq-constraints}, including not only multi-parameter  but also two-parameter models;
on the other hand, since these results are established under the general constraints \eqref{eq-constraints},
their conditions  are stronger than those in Section \ref{sec-two-para} under two-parameter models.
In contrast, the generic identifiability results in  Sections \ref{sec-gen} and \ref{sec-gen-q} (Theorem \ref{thm-order-gen}--\ref{prop-vio-gen}) only apply  to multi-parameter models. This is because under generic identifiability, the nonidentifiable measure-zero subset of a multi-parameter model's parameter space (such as GDINA), could still contain the  parameter space corresponding to a two-parameter submodel (such as DINA), making these generic identifiability results  not applicable to   two-parameter models.
Nevertheless, generic identifiability is a general concept not just restricted to the multi-parameter models.
An interesting future direction to study is the generic identifiability of two-parameter models under the introduced $\pp$-partial identifiability framework; that is, one can study what conditions lead to the generic identifiability of $(\ttt^+,\ttt^-,\nnu)$.
We also point that a multi-parameter model can also be $\pp$-partially identifiable, as discussed in Remark \ref{rmk-mult-pid}.

For the  $\pp$-partial identifiability and generic identifiability results in Sections \ref{sec-two-para}--\ref{sec-extend},   we assume that the model specification for each item, the design matrix and    latent class space $\ma$   are available as prior knowledge.
{
In practice, there can be scenarios where not all of such information is available.
As pointed out by one reviewer, in applications of cognitive diagnostic modeling, both the advances in modeling capacity and computing flexibility, and the recent real-data examples provide ground for adopting a model with mixed type of items, which are determined in a data-driven way. 
To this end, our strict identifiability results in Section \ref{sec-gen} and those in Section \ref{sec-mixed} for mixed-items models can  be applied to assess identifiability \textit{a posteriori}.}
When deciding which model to use in practice, one can use the response data to determine the number of latent classes and determine which diagnostic model an item conforms to. For instance, 
one may employ the  popular information criteria such as AIC and BIC to perform model selection; or one may first fit a general cognitive diagnostic model, such as GDINA or GDM, then use the Wald test to determine which submodel an item follows \cite{dela2011}.  Alternatively, one may  use a penalized likelihood method \cite{xu2018jasa} or Bayesian method \cite{chen2018bayesian} to directly estimate the structure of the item parameters for each item; such structure informs the model specification of the item.
For the selected candidate models, we would recommend further applying our identifiability theory to assess their identifiability and validity.
The general theoretical framework developed in this paper would be a useful tool to develop the identifiability and estimability conditions for learning the item-level model structure and the population-level latent class space $\ma$.
{This is an  interesting and important  direction that we plan to pursue   in the future.}

\vspace{4mm}
\textbf{Acknowledgements.} The authors thank the editors, an   associate editor, and two   reviewers for their helpful and constructive comments.

\begin{supplement}\label{suppA}
  \sname{Supplement to ``Partial Identifiability of Restricted Latent Class Models''}
    \stitle{}
   \slink[doi]{XXX}
  \sdatatype{supp.pdf}
  \sdescription{The supplementary material contains the proofs of main results and some technical lemmas.}
\end{supplement}

\bibliographystyle{imsart-number}
\bibliography{ref}

\newpage

~\\

\begin{appendix}
\begin{center}
\textbf{SUPPLEMENT TO ``PARTIAL IDENTIFIABILITY OF RESTRICTED LATENT CLASS MODELS"}
\end{center}

\bigskip

	This supplementary material is organized as follows.   Section A   presents the $Q$-matrices associated with the two real datasets in Example \ref{exp-timss-8th} and \ref{exp-frac} and details of establishing model identifiability of them.   Section B and C   provide the proofs of the main theoretical results for the two-parameter and multi-parameter restricted latent class models in Sections \ref{sec-two-para} and \ref{sec-multi-para} of the main text, respectively. Section D gives the proofs of the results in Section \ref{sec-extend} in the main text.  Section E  gives the proofs of some technical lemmas. \\


\appendix
\section*{Section A: $Q$-matrices associated with real data}

\subsection{TIMSS Data $Q$-matrix and its identifiability.}
Table \ref{tab-tim-43} presents the full $43\times 12$  $Q$-matrix $Q_{43\times 12}$ for the TIMSS data, which is introduced in Example \ref{exp-timss-8th} of the main text. 
 The $Q$-matrix was constructed by mathematics educators and  researchers and its form was specified in \cite{choi2015cdm}. Please refer to \cite{choi2015cdm}  for more details about the   test items and fine-grained attributes. We next show how our theoretical results guarantee $\pp$-partial identifiability of two-parameter models and generic identifiability of multi-parameter models under this $Q_{43\times 12}$.
 
 
 \paragraph{$\pp$-partial identifiability.}
 We show that the $Q_{43\times 12}$ satisfies conditions (C1$^*$) and (C2$^*$).
 The $Q_{43\times 12}$ in Table \ref{tab-tim-43}
 contains 9 basis items $S_{basis}=\{4$,  8,  15, 16, 19, 24,  30,  34,  $38\}$ and the remaining $34$    non-basis items. We can check that each basis item is $S_{non}$-differentiable and conditions (C1$^*$) and (C2$^*$) hold. Thus Corollary \ref{cor-dina-main} implies $\pp$-partial identifiability of the two-parameter restricted latent class models, and also guarantees estimability of $(\ttt^+,\ttt^-,\nnu)$.
 
 \paragraph{Generic identifiability.}
We show that the $Q_{43\times 12}$ satisfies conditions (C5) and (C6). 
In particular,   let $S_1=\{1$,   3,    4,  5,   7,    8,  12, 13, 15, 17,  19, $38\}$ and $S_2=\{2$,   11,   16,  20, 22,  23,  24,  26,    30, 31, 33,   $34
\}$, then items in each of $S_1$ and $S_2$ can be arranged in a way such that the sub-$Q$-matrices $Q_1$ and $Q_2$ take the form of \eqref{eq-diag}, which implies   condition (C5). In addition,  each attribute is required by at least one item in $(S_1\cup S_2)^c$  and thus condition (C6) is also satisfied.
 Theorem \ref{thm-gen-q} then gives the generic identifiability of any  multi-parameter   model associated with this  $Q$-matrix.

\begin{table}[h!]
\centering
\caption{$Q$-matrix, TIMSS 2003 $8$th Grade Data}
\begin{tabular}{l|cccccccccccc}
\hline
Item  &  $\alpha_1$ & $\alpha_2$ & $\alpha_3$ & $\alpha_4$ & $\alpha_5$ & $\alpha_6$ & $\alpha_7$ & $\alpha_8$ & $\alpha_9$ & $\alpha_{10}$ & $\alpha_{11}$ & $\alpha_{12}$ \\
\hline
1 &  1 & 0 & 1 & 1 & 0 & 0& 0 & 0& 0 & 0& 0 & 0\\
2 &  1 & 0 & 0 & 0 & 0 & 0& 0 & 0& 0 & 0& 0 & 0\\
3 &  1 & 0 & 1 & 0 & 0 & 0& 0 & 0& 0 & 1& 0 & 0\\
4 &  0 & 1 & 1 & 0 & 0 & 0& 0 & 0& 0 & 1& 0 & 0\\
5 &  0 & 0 & 0 & 0 & 0 & 0& 1 & 0& 0 & 0& 1 & 0\\
6 &  0 & 0 & 0 & 0 & 0 & 0& 1 & 0& 0 & 1& 1 & 0\\
7 &  0 & 0 & 0 & 0 & 0 & 0& 1 & 0& 0 & 0& 1 & 0\\
8 &  0 & 1 & 0 & 1 & 0 & 0& 0 & 0& 0 & 0& 0 & 0\\
9 &  0 & 0 & 0 & 1 & 0 & 0& 1 & 0& 0 & 0& 0 & 0\\
10 &  0 & 0 & 0 & 1 & 1 & 1& 0 & 0& 0 & 0& 0 & 0\\
11 & 0 & 0 & 0 & 1 & 1 & 1 & 0 & 0 & 0 & 0 & 0 & 0\\
12 & 1 & 0 & 0 & 0 & 1 & 0 & 0 & 0 & 0 & 0 & 0 & 0\\
13 & 0 & 0 & 0 & 0 & 0 & 0 & 0 & 1 & 1 & 0 & 0 & 0\\
14 & 0 & 0 & 0 & 0 & 0 & 0 & 0 & 1 & 1 & 0 & 0 & 0\\
15 & 0 & 0 & 0 & 0 & 0 & 0 & 0 & 1 & 0 & 0 & 0 & 0\\
16 & 0 & 0 & 0 & 1 & 1 & 0 & 0 & 0 & 0 & 0 & 0 & 0\\
17 & 1 & 0 & 0 & 0 & 0 & 0 & 0 & 0 & 0 & 1 & 0 & 0\\
18 & 1 & 1 & 1 & 0 & 0 & 0 & 0 & 0 & 0 & 1 & 0 & 0\\
19 & 0 & 0 & 0 & 1 & 0 & 0 & 0 & 0 & 0 & 0 & 0 & 1\\
20 & 1 & 0 & 0 & 0 & 0 & 0 & 1 & 0 & 0 & 0 & 0 & 0\\
21 & 1 & 1 & 0 & 0 & 0 & 0 & 0 & 0 & 0 & 0 & 0 & 0\\
22 & 0 & 0 & 1 & 1 & 1 & 0 & 0 & 0 & 0 & 0 & 0 & 0\\
23 & 1 & 0 & 0 & 0 & 0 & 0 & 0 & 0 & 0 & 1 & 0 & 0\\
24 & 0 & 0 & 0 & 0 & 0 & 0 & 0 & 0 & 1 & 1 & 0 & 0\\
25 & 0 & 0 & 0 & 0 & 0 & 0 & 1 & 0 & 0 & 0 & 0 & 0\\
26 & 1 & 0 & 0 & 0 & 0 & 0 & 0 & 0 & 0 & 0 & 0 & 1\\
27 & 1 & 0 & 0 & 0 & 0 & 0 & 0 & 0 & 0 & 0 & 0 & 0\\
28 & 1 & 1 & 0 & 0 & 0 & 0 & 0 & 0 & 0 & 0 & 0 & 0\\
29 & 1 & 1 & 0 & 0 & 0 & 0 & 0 & 0 & 0 & 0 & 0 & 0\\
30 & 0 & 0 & 0 & 0 & 0 & 0 & 0 & 0 & 0 & 1 & 1 & 0\\
31 & 0 & 0 & 1 & 0 & 0 & 0 & 1 & 0 & 0 & 0 & 1 & 0\\
32 & 1 & 0 & 0 & 0 & 0 & 0 & 0 & 0 & 0 & 0 & 0 & 0\\
33 & 0 & 0 & 0 & 0 & 0 & 0 & 1 & 1 & 0 & 0 & 0 & 0\\
34 & 0 & 1 & 0 & 0 & 1 & 0 & 0 & 0 & 0 & 0 & 0 & 0\\
35 & 0 & 0 & 0 & 0 & 0 & 0 & 1 & 0 & 0 & 1 & 0 & 0\\
36 & 0 & 0 & 0 & 0 & 1 & 0 & 0 & 1 & 1 & 0 & 0 & 0\\
37 & 0 & 0 & 0 & 0 & 0 & 0 & 1 & 0 & 0 & 0 & 0 & 0\\
38 & 0 & 0 & 0 & 0 & 1 & 1 & 0 & 0 & 0 & 0 & 0 & 0\\
39 & 1 & 0 & 0 & 0 & 0 & 0& 0 & 0& 0 & 0& 0 & 0\\
40 & 1 & 1 & 1 & 0 & 0 & 0& 0 & 0& 0 & 0& 0 & 0\\
41 & 1 & 0 & 0 & 0 & 0 & 0& 0 & 0& 0 & 0& 0 & 0\\
42 & 1 & 1 & 0 & 1 & 0 & 0& 0 & 0& 0 & 0& 0 & 0\\
43 & 1 & 0 & 0 & 0 & 0 & 0& 0 & 0& 0 & 0& 0 & 1\\
\hline 
\end{tabular}
\label{tab-tim-43}
\end{table}

\subsection{Fraction Subtraction Data $Q$-matrix and its identifiability.}\label{sec-qmat-frac}
Table \ref{tab-qfrac} presents the full $20\times 8$ $Q$-matrix $Q_{20\times 8}$ for the fraction subtraction data, which is introduced in Example \ref{exp-frac} in the main text.
We next show how our theoretical results guarantee $\pp$-partial identifiability of two-parameter models and generic identifiability of multi-parameter models under this $Q_{20\times 8}$.

\begin{table}[h!]
\centering
\caption{$Q$-matrix, Fraction Data}
\begin{tabular}{ll|cccccccc}
\hline
Item ID &Content &  $\alpha_1$ & $\alpha_2$ & $\alpha_3$ & $\alpha_4$ & $\alpha_5$ & $\alpha_6$ & $\alpha_7$ & $\alpha_8$\\
\hline
1& $\frac{5}{3} - \frac{3}{4}$ & 0 &	0&	0&	1&	0&	1&	1&	0\\[1mm]
2& $\frac{3}{4} - \frac{3}{8}$ & 0 &   0&	0&	1&	0&	0&	1&	0\\[1mm]
3& $\frac{5}{6} -\frac{1}{9}$ & 0&	0&	0&	1&	0&	0&	1&	0\\[1mm]
4& $3 \frac{1}{2} - 2\frac{3}{2}$ & 0&  1&	1&	0&	1&	0&	1&	0\\[1mm]
5& $4\frac{3}{5} -3 \frac{4}{10}$ & 0&	1&	0& 	1&	0&	0&	1&	1\\[1mm]
6& $\frac{6}{7} -\frac{4}{7}$ & 0&	0&	0&	0&	0&	0&	1&	0\\[1mm]
7& $ 3 -2 \frac{1}{5}$  & 1&	1&	0&	0&	0&	0&	1&	0\\[1mm]
8& $\frac{2}{3} -\frac{2}{3}$ & 0&	0&	0&	0&	0&	0&	1&	0\\[1mm]
9&  $3 \frac{7}{8} - 2$ & 0&	1&	0&	0&	0&	0&	0&	0\\[1mm]
10& $4 \frac{4}{12} - 2 \frac{7}{12}$ & 0&	1&	0&	0&	1&	0&	1&	1\\[1mm]
11&  $4 \frac{1}{3} -2\frac{4}{3}$  & 0&	1&	0&	0&	1&	0&	1&	0\\[1mm]
12& $\frac{11}{8}- \frac{1}{8}$ & 0&	0&	0&	0&	0&	0&	1&	1\\[1mm]
13&  $3\frac{3}{8} - 2\frac{5}{6}$ & 0&	1&	0&	1&	1&	0&	1&	0\\[1mm]
14&  $3 \frac{4}{5} - 3\frac{2}{5}$ & 0&	1&	0&	0&	0&	0&	1&	0\\[1mm]
15&  $2 - \frac{1}{3}$ & 1&	0&	0&	0&	0&	0&	1&	0\\[1mm]
16&  $4 \frac{5}{7} - 1\frac{4}{7}$ & 0&	1&	0&	0&	0&	0&	1&	0\\[1mm]
17&  $7\frac{3}{5} - \frac{4}{5}$ & 0&	1&	0&	0&	1&	0&	1&	0\\[1mm]
18& $4\frac{1}{10} - 2\frac{8}{10}$ & 0&	1&	0&	0&	1&	1&	1&	0\\[1mm]
19& $4 - 1\frac{4}{3}$ & 1&	1&	1&	0&	1&	0&	1&	0\\[1mm]
20& $ 4 \frac{1}{3} -1\frac{ 5}{3}$ & 0&	1&	1&	0&	1&	0&	1&	0\\[1mm]
\hline
\end{tabular}
\label{tab-qfrac}
\end{table}

\paragraph{$\pp$-partial identifiability.}
We apply Theorem \ref{thm_cond} since attribute $k=6$ is only required by two items $\{1,18\}$ and condition (C1$^*$) is violated.
Specifically, we transform the original $Q$-matrix to the form of \eqref{eq2v} with $\vv_1 = (0,\, 0,\, 0,\, 1,\, 0,\, 1,\, 0)$, $\vv_2 = (0,\, 1,\, 0,\, 0,\, 1,\, 1,\, 0)$ and submatrix $Q'$   as specified in \eqref{eq-Qdetail}, by first exchanging the second and the eighteenth rows and  then exchanging the first and the sixth columns. The transformed $Q$-matrix falls into the case (a) of (B.1)   in Theorem \ref{thm_cond}, and it suffices to show that the   $Q'$-matrix in \eqref{eq-Qdetail} satisfy condition (C1$^*$) and (C2$^*$). We can check that (C1$^*$) holds for $Q'$ that attributes required by each   $\bq$-vector in $Q'$ are repeatedly measured by at least two disjoint sets of other items. In addition,  (C2$^*$) is satisfied because $Q'$ only has one   basis item   $S_{basis}(Q')=\{9\}$ and the   item $j=9$ is $S_{non}$-differentiable.
\begin{equation}\label{eq-Qdetail}
Q' = 
\begin{blockarray}{cccccccc}
& \alpha_1 & \alpha_2 & \alpha_3 & \alpha_4 & \alpha_5 & \alpha_7 & \alpha_8\\
\begin{block}{l(ccccccc)}
2& 0&   0&  0&	1&	0&	    1&	0\\
3& 0&	0&	0&	1&	0&		1&	0\\
4& 0&   1&	1&	0&	1&		1&	0\\
5& 0&	1&	0& 	1&	0&		1&	1\\
6& 0&	0&	0&	0&	0&		1&	0\\
7& 1&	1&	0&	0&	0&		1&	0\\
8& 0&	0&	0&	0&	0&		1&	0\\
9& 0&	1&	0&	0&	0&		0&	0\\
10& 0&	1&	0&	0&	1&		1&	1\\
11& 0&	1&	0&	0&	1&		1&	0\\
12& 0&	0&	0&	0&	0&		1&	1\\
13& 0&	1&	0&	1&	1&		1&	0\\
14& 0&	1&	0&	0&	0&		1&	0\\
15& 1&	0&	0&	0&	0&		1&	0\\
16& 0&	1&	0&	0&	0&		1&	0\\
17& 0&	1&	0&	0&	1&		1&	0\\
19& 1&	1&	1&	0&	1&		1&	0\\
20& 0&	1&	1&	0&	1&		1&	0\\
\end{block}
\end{blockarray}
\begin{array}{l}
\preceq \bq_3,~\bq_5 \\ 
\preceq \bq_2,~\bq_5 \\ 
\preceq \bq_9,~ \bq_{20}\\ 
\preceq \bq_3\vee\bq_{4}\vee\bq_{10},~\bq_{12}\vee\bq_{13} \\
\preceq \bq_2,~\bq_8; \\
\preceq \bq_{14}\vee\bq_{15},~\bq_{19} \\
\preceq \bq_2,~\bq_6 \\
\preceq \bq_4,~\bq_5; \\
\preceq \bq_{4}\vee\bq_{5},~ \bq_7\vee\bq_{12}\vee\bq_{13} \\
\preceq \bq_{4}\vee\bq_{5},~ \bq_7\vee\bq_{12}\vee\bq_{13} \\
\preceq \bq_{4}\vee\bq_{5},~ \bq_7\vee\bq_{12}\vee\bq_{13} \\
\preceq \bq_{4}\vee\bq_{5},~ \bq_{13} \\
\preceq \bq_4,~\bq_5 \\
\preceq \bq_{14}\vee\bq_{15},~\bq_{19} \\
\preceq \bq_4,~\bq_5 \\
\preceq \bq_{13},~\bq_{18} \\
\preceq \bq_{4}\vee\bq_{7},~ \bq_{15}\vee\bq_{19} \\
\preceq \bq_{4}\vee\bq_{5},~ \bq_{13} \\
\end{array}
\end{equation}
Theorem \ref{thm_cond} therefore gives the $\pp$-partial identifiability of the two-parameter models.

\paragraph{Generic identifiability.}
We apply Theorem \ref{prop-vio-gen} since attribute $6$ is required by only two items, item 6 and item 18. Rearranging the columns and rows of this $Q$-matrix to the form of \eqref{eq-gen-v1v2} with $\vv_1 = (0,\, 0,\, 0,\, 1,\, 0,\, 1,\, 0)$ and $\vv_2 = (0,\, 1,\, 0,\, 0,\, 1,\, 1,\, 0)$, we have $\vv_1\vee\vv_2\neq \mo_{K-1}$ and the sub-matrix $Q'$ part satisfies conditions (C5) and (C6), so Theorem \ref{prop-vio-gen} gives the generic identifiability of      multi-parameter $Q$-restricted latent class models.

\medskip
\section*{Section B: Proof of Main Results in Section 3}\label{sec-two-proof}

In this section we first introduce some technical quantities and their properties which will be useful in later proofs, then present the proofs of the main results in Section \ref{sec-two-para} for two-parameter restricted latent class models.

To facilitate the study of parameter identifiability of restricted latent class models, we consider a marginal probability matrix $T(\TT)$ of size $2^J\times m$ as follows, where $J=|\ms|$ denotes the number of items and $m=|\ma|$ denotes the number of classes. Rows of $T(\TT)$ are indexed by the $2^J$ possible response patterns $\rr = (r_1,\ldots,r_J)^\top\in\{0,1\}^J$ and columns of $T(\TT)$ are indexed by latent classes $\aaa\in\ma$, while the $(\rr, \aaa)$th entry of $T(\TT)$, denoted by $T_{\rr, \aaa}(\TT)$, represents the marginal probability that subjects in latent class $\aaa$ provide positive responses to the set of items $\{j: r_j=1\}$, namely 
\[
T_{\rr,\aaa}(\TT) = P(\boldsymbol R \succeq \rr\mid\TT, \aaa) = \prod_{j=1}^J \theta_{j,\aaa}^{r_j}.
\]
Denote the $\aaa$th column vector and the $\rr$th row vector of the $T$-matrix by $T_{\Cdot,\aaa}(\TT)$ and $T_{\rr,\Cdot}(\TT)$ respectively. Let $\ee_j$ denote the $J$-dimensional unit vector with the $j$th element being one and all the other elements being zero, then any response pattern $\rr$ can be written as a sum of some $\ee$-vectors, namely $\rr=\sum_{j:r_j=1}\ee_j$.
The $\rr$th element of the $2^J$-dimensional vector $T(\TT)\pp$ is
\[
\{T(\TT)\pp\}_{\rr}= T_{\rr,\Cdot}(\TT)\pp = \sum_{\aaa\in\ma} T_{\rr,\aaa}(\TT) p_{\aaa} = P(\boldsymbol R \succeq \rr \mid \TT).
\]
Based on the   $T$-matrix, we have the following   definition of identifiability  for model parameters ($\TT, \pp$), equivalent to definition  \eqref{eq-orig} in Section 2.3 of the main text. The equivalence of the two definitions comes from that two sets of model parameters lead to the same marginal distribution of  responses $\{P(\boldsymbol R \succeq \rr \mid \TT), \forall \rr\in\{0,1\}^J\}$ if and only if they  lead to the same distribution of the responses $\{P(\boldsymbol R = \rr \mid \TT), \forall \rr\in\{0,1\}^J\}$.
\begin{proposition}\label{equiv-def}
Under a restricted latent class model, the model parameters are identifiable if and only if for any ($\TT$, $\pp$) and  ($\bar\TT$, $\bar\pp$), 
\begin{equation}\label{eq1}
T(\TT)\pp = T(\bar\TT)\bar\pp
\end{equation}
implies $(\TT,\pp) = (\bar\TT,\bar\pp)$.
\end{proposition}

 Together with this equivalent definition, the following proposition, which was introduced in \cite{xu2017}, describes an important algebraic property of the $T$-matrix and  will be used in our proofs.
\begin{proposition}\label{prop-ltrans}
For any $\ttt^*=(\theta_1,\ldots,\theta_J)^\top\in\mathbb R^J$, there exists an invertible lower triangular matrix $D(\ttt^*)$ depending solely on $\ttt^*$, such that the diagonal elements of $D(\ttt^*)$ are all 1, 
and 
\[
T(\TT-\ttt^*\mathbf1^\top) = D(\ttt^*)T(\TT).
\]
\end{proposition}

\noindent 
Another useful property of the $T$-matrix  is given by the following lemma, whose proof is given in Section D.
\begin{lemma}\label{lem-rk}
Denote the $T$-matrix corresponding to a subset of items $S$ by $T(\TT_{S})$, where $\TT_S=(\theta_{j,\aaa},~j\in S,~\aaa\in\ma)$. 
If for an item set $S$, the $\Gamma$-matrix $\Gamma^S$ of size $|S|\times m$ is separable, then the corresponding $T$-matrix $T(\TT_{S})$ of size $2^{|S|}\times m$  has full column rank $m$.
\end{lemma}

 Equipped with the above   developments, now we are ready to prove the main results.


\begin{proof}[Proof of Proposition \ref{prop1} and Proposition \ref{prop1'}] 
When $\TT$ is known, by Proposition \eqref{equiv-def}, we only need to show that if   $T(\TT)\pp = T(\TT)\bar\pp$, then  $\pp=\bar\pp$. 
This directly follows from the result in Lemma \ref{lem-rk} that  when $\Gamma$ is separable,      the $T$-matrix $T(\TT)$ has full column rank $m$.

We next prove the necessity part of Proposition \ref{prop1} that the separability of the $\Gamma$-matrix   is    necessary  for  identifiability of  $\pp$.
Suppose $\Gamma$ is inseparable and consider the representatives $\aaa_{\ma_1},\ldots,\aaa_{\ma_C}$ from the $C$ equivalence classes, respectively.
It suffices to show that for any $ \pp\neq \bar\pp$, if 
$\nnu= \bar \nnu$, where $\bar \nnu = (\bar \nu_{[\aaa_{\ma_i}]}, i=1,\ldots,C)$ and $\bar \nu_{[\aaa_{\ma_i}]} = \sum_{\aaa:\,\aaa\in \ma_i}\bar p_{\aaa}$, 
then $T(\TT)\pp = T( \TT)\bar\pp.$
 Note that under the two-parameter restricted latent class models, any two equivalence latent classes $\aaa\stackrel{\Gamma}{\sim}\aaa'$    have identical item parameter vectors, i.e. $\ttt_{\Cdot,\aaa} = \ttt_{\Cdot,\aaa'}$. This further implies  $T_{\Cdot,\aaa}(\TT) = T_{\Cdot,\aaa'}(\TT)$ by the definition of the $T$-matrix.
Let $\Gamma^{eq}$ be the $J\times C$ submatrix of $\Gamma$ that consists of the   column vectors indexed by 
$\aaa_{\ma_i}$, $i=1,\ldots,C$, 
and $T^{eq}(\TT)$ be the corresponding $2^J\times C$ submatrix of $T(\TT)$. 
Then if 
$\nnu= \bar \nnu$,
\[\begin{aligned}
T(\TT)\bar\pp &= \sum_{i=1}^C \sum_{\aaa\in\ma_i} T_{\Cdot,\aaa}(\TT)\bar p_{\aaa} 
= \sum_{i=1}^C T_{\Cdot,\aaa_{\ma_i}}(\TT) \bar\nu_{\ma_i}\\
& =T^{eq}(\TT)\bar\nnu = T^{eq}(\TT) \nnu\\
&
= \sum_{i=1}^C T_{\Cdot,\aaa_{\ma_i}}(\TT) \nu_{\ma_i}= \sum_{i=1}^C \sum_{\aaa\in\ma_i} T_{\Cdot,\aaa}(\TT)  p_{\aaa}  = T(\TT)\pp,\\
\end{aligned}\]
This proves that given an inseparable $\Gamma$-matrix, $\pp$ is not identifiable. 

Lastly, we prove Proposition \ref{prop1'} that when the $\Gamma$-matrix is inseparable,   
the grouped proportion parameters  $\nnu$ is identifiable.   
By Proposition \ref{equiv-def}, we only need to show that if   $T(\TT)\pp = T(\TT)\bar\pp$, then $\nnu=\bar\nnu$. 
  From the   calculation in the previous paragraph, we know
  $T(\TT)\pp= T^{eq}(\TT) \nnu$ and  $T(\TT)\bar\pp= T^{eq}(\TT) \bar\nnu.$ 
  Since $\Gamma^{eq}$ is separable by its construction, Lemma \ref{lem-rk} gives that   $T^{eq}(\TT)$ has full column rank $C$. Therefore,   $T(\TT)\pp = T(\TT)\bar\pp$ implies $\nnu=\bar\nnu$ and   $\nnu$ is identifiable.
This completes the proof.
\end{proof}


\begin{proof}[Proof of Equation \eqref{eq-group-q} in Remark \ref{remark1}]
We introduce a notation first. For any set of items $S\subset\{1,\ldots,J\}$, we denote $\bq_S =\vee_{h\in S}\,\bq_h$.
To prove the first part of \eqref{eq-group-q}, it suffices to show that under the conjunctive DINA model, (i) for any $\aaa_1,\aaa_2\in\mathcal R^{Q,conj}$ and $\aaa_1\neq\aaa_2$, we have $\Gamma_{\Cdot,\aaa_1}\neq\Gamma_{\Cdot,\aaa_2}$; and (ii) for any $\aaa\in\{0,1\}^K$, there exists $\aaa'\in\mathcal R^{Q,conj}$ such that $\Gamma_{\Cdot,\aaa}=\Gamma_{\Cdot,\aaa'}$. 

For any $\aaa_1,\aaa_2\in\mathcal R^{Q,conj}$, without loss of generality, we can denote $\aaa_1 = \vee_{h\in S_1}\bq_h$ and $\aaa_2 = \vee_{h\in S_2}\bq_h$ where $S_1, S_2\subset\{1,\ldots,J\}$ are two different sets of items. Then by definition, in the vector $\Gamma_{\Cdot,\aaa_1}$, the entry $\Gamma_{j,\aaa_1}=1$ if and only if $j\in S_1$; and similarly in the vector $\Gamma_{\Cdot,\aaa_2}$, the entry $\Gamma_{j,\aaa_2}=1$ if and only if $j\in S_2$. Since $S_1\neq S_2$, we must have the two vectors different, i.e., $\Gamma_{\Cdot,\aaa_1}\neq \Gamma_{\Cdot,\aaa_2}$. This proves (i). Next, for any $\aaa\in\{0,1\}^K$,  
we collect the items that $\aaa$ is capable of in the set $S_{\aaa} = \{j\in\ms:\aaa\succeq\bq_j\}$, and just define 
$
\aaa' = \bq_{S_{\aaa}}. 
$
then clearly $\aaa'\in\mathcal R^{Q,conj}$. Additionally, the set of items that $\aaa'$ is capable of is also $S_{\aaa}$, so $\Gamma_{\Cdot, \aaa} = \Gamma_{\Cdot, \aaa'} $. This proves (ii). So the first part of \eqref{eq-group-q} holds.

To prove the second part of (9), it suffices to show that under the disjunctive DINO model, (iii) for any $\aaa_1,\aaa_2\in\mathcal R^{Q,comp}$ and $\aaa_1\neq\aaa_2$, we have $\Gamma^{disj}_{\Cdot,\aaa_1}\neq\Gamma^{disj}_{\Cdot,\aaa_2}$; and (iv) for any $\aaa\in\{0,1\}^K$, there exists $\aaa'\in\mathcal R^{Q,disj}$ such that $\Gamma^{disj}_{\Cdot,\aaa}=\Gamma^{disj}_{\Cdot,\aaa'}$.
First, for $\aaa_1,\aaa_2\in\mathcal R^{Q,disj}$ and $\aaa_1\neq\aaa_2$, they can be written as $\aaa_1 = \mathbf1_K^\top - \bq_{S_1}$ and $\aaa_2 = \mathbf1_K^\top - \bq_{S_2}$ where $S_1,S_2$ are two different item sets. Then 
\begin{align}\label{eq-disj}
\Gamma^{disj}_{j,\aaa_1} = &~ I(\aaa_1\nprec\bq_j) = I(\mathbf1_K^\top - \bq_{S_1}\nprec\bq_j)\\ \notag
=&~ I(\exists k~s.t.~q_{j,k}=1,~q_{S_1,k}=0) 
= I(j\not\in S_1),
\end{align}
and similarly $\Gamma^{disj}_{j,\aaa_2} = I(j\not\in S_2)$. Since $S_1\neq S_2$, we have the inequalities of the two column vectors $\Gamma^{disj}_{\Cdot,\aaa_1}\neq\Gamma^{disj}_{\Cdot,\aaa_2}$. This proves (iii). 
Next, for any $\aaa\in\{0,1\}^K$, we define $S^{\star}_{\aaa} = \{j\in\ms:\, \aaa\prec\bq_j\}$, which is the set of items $\aaa$ is \textit{not }capable of under the disjunctive model. Define 
$$
\aaa' = \mathbf1^\top_K - \bq_{S^\star_{\aaa}},
$$
then clearly $\aaa'\in \mathcal R^{Q,disj}$. Further, similar to the derivation in \eqref{eq-disj}, we have 
$$
\Gamma^{disj}_{j,\aaa'} = I(j\not\in S^\star_{\aaa}),
$$
which implies the set of items $\aaa'$ is \textit{not} capable of is also $S^\star_{\aaa}$. This means $\Gamma^{disj}_{\Cdot,\aaa}=\Gamma^{disj}_{\Cdot,\aaa'}$ and proves (iv).
\end{proof}
\color{black}

In the following proofs of the results for two-parameter restricted latent class models, for any latent class $\aaa$, we   use $[\aaa]$ to denote the $\Gamma$-induced equivalence class containing $\aaa$. 
Then by definition, with $j$ ranging in the set of all items and $[\aaa]$ ranging in the set of all equivalence classes, $(\theta_{j,[\aaa]})$ give all the item parameters of interest while $(\nu_{[\aaa]})$ give all the grouped proportion parameters of interest, under the framework of $\pp$-partial identifiability.
In the following, when there is no ambiguity, we write the item parameters as $\TT=(\theta_{j,[\aaa]})$; and write  $T^{eq}(\bar\TT)\bar\nnu = T^{eq}(\TT) \nnu$ as $T(\bar\TT)\bar\nnu = T(\TT) \nnu$, for  $\TT=(\theta_{j,[\aaa]})$, $\nnu=(\nu_{[\aaa]})$, and $\bar\TT=(\bar\theta_{j,[\aaa]})$, $\bar\nnu=(\bar\nu_{[\aaa]})$.

\begin{proof}[Proof of Theorem \ref{thm1}]
To show the $\pp$-partial identifiability, Proposition \ref{equiv-def} implies that we only need to show for any ($\TT$, $\nnu$) and  ($\bar\TT$, $\bar\nnu$), 
$ T(\TT)\nnu = T(\bar\TT)\bar\nnu$
implies $(\TT,\nnu) = (\bar\TT,\bar\nnu)$.
We prove this in two steps: in Step 1, we   show  the Repeated Measurement Condition (C1) ensures  identifiability of $(\ttt^+,\ttt^-_{non})$; in Step 2, we show the Sequentially Differentiable Condition (C2) additionally ensures  identifiability of the remaining  parameters $(\ttt^-_{basis}, \nnu)$, where $\ttt^-_{basis}=(\theta^-_j,~j\in S_{basis})$.
In both steps, we  frequently use the following lemma, whose proof is postponed to Section D. 
\begin{lemma}\label{lemmaS2}
	Under the two-parameter restricted latent class models,   Equation \eqref{eq1} implies that $\theta_j^+\neq \bar\theta_j^-$ and $\theta_j^-\neq \bar\theta_j^+$ for any item $j$.
\end{lemma}

\paragraph{\bf Step 1.}
To show  identifiability of $(\ttt^+,\ttt^-_{non})$ under (C1), we start with two identifiability results in the following two cases (a) and (b).
\begin{enumerate}
 \item[(a)] If for item $j$, there exist two disjoint sets of items $S_1$ and  $S_2$, both not containing $j$, such that 
\begin{equation}\label{casea}
\mc_j \supseteq \mc_{S_1},\quad \mc_j \supseteq \mc_{S_2},	
\end{equation}
then we have the identifiability of $\theta^+_j$, as proved in the following.

\smallskip
Define
\[\ttt^* = \sum_{h\in S_1}\theta^-_{h}\ee_{h} + \sum_{m\in S_2}\bar\theta^-_{m}\ee_{m};\]
then consider the two row vectors corresponding to response pattern $\rr^*=\sum_{h\in S_1\cup S_2}\ee_{h}$ in the transformed $T$-matrices $T(\TT-\ttt^*\mathbf1^\top) $ and $T(\bar\TT-\ttt^*\mathbf1^\top) $, respectively, and we have the following expressions:
\begin{equation}
\begin{aligned}\label{eq-thm1}
&T_{\rr^*,[\aaa]}(\TT-\ttt^*\mathbf1^\top) =
 T_{\rr^*,[\aaa]}(\ttt^+ - \ttt^*, \ttt^- - \ttt^*) \\ 
& =  \begin{cases}  
\prod_{h\in S_1}(\theta^+_h - \theta^-_h)\prod_{m\in S_2}(\theta^+_m - \bar\theta^-_m), &  \mbox{for } [\aaa]\in\mc_{S_1}\cap \mc_{S_2}; \\
0, & \text{otherwise}.
\end{cases}\\
& T_{\rr^*,[\aaa]}(\bar\TT-\ttt^*\mathbf1^\top)
=T_{\rr^*,[\aaa]}(\bar\ttt^+ - \ttt^*, \bar\ttt^- - \ttt^*) \\ 
& =  \begin{cases}  
\prod_{h\in S_1}(\bar\theta^+_h - \theta^-_h)\prod_{m\in S_2}(\bar\theta^+_m - \bar\theta^-_m), & \mbox{for } [\aaa]\in\mc_{S_1}\cap \mc_{S_2}; \\
0, & \text{otherwise}.
\end{cases}
\end{aligned}
\end{equation} 
Since $\theta^+_h\neq \bar\theta^-_h$ and $\bar\theta^+_h\neq \theta^-_h$ for all $h$ by Lemma \ref{lemmaS2}, we have
\begin{eqnarray*}
	&& T_{\rr^*,\Cdot}(\ttt^+-\ttt, \ttt^- -\ttt)\nnu \\
	&=& \Big( \sum_{[\aaa]\in\mc_{S_1}\cap \mc_{S_2}} \nu_{[\aaa]}\Big)
\prod_{h\in S_1}(\bar\theta^+_h - \theta^-_h)\prod_{m\in S_2}(\bar\theta^+_m - \bar\theta^-_m)\neq0,
\end{eqnarray*}
and similarly $T_{\rr^*,\Cdot}(\bar\ttt^+-\ttt, \bar\ttt^- -\ttt)\bar\nnu\neq0$.
Therefore, $\mc_{S_1}\subseteq\mc_j$, $\mc_{S_2}\subseteq\mc_j$ together with Equation (\ref{eq1}) indicates
\begin{eqnarray}\label{eq-thm1-denom1}
\theta^+_j
&=& 
\frac{T_{\rr^*+\ee_j,\Cdot}(\ttt^+-\ttt, \ttt^--\ttt)\nnu}{T_{\rr^*,\Cdot}(\ttt^+-\ttt, \ttt^- -\ttt)\nnu} \notag\\
&=& 
\frac{T_{\rr^*+\ee_j,\Cdot}(\bar\ttt^+-\ttt, \bar\ttt^--\ttt)\bar\nnu}{T_{\rr^*,\Cdot}( \bar\ttt^+-\ttt, \bar\ttt^- -\ttt)\bar\nnu}
~=~ \bar \theta^+_j.
\end{eqnarray}

\item[(b)] If for item $j$, there exist another item $h$ and an item set $S_2$ not containing $j$ or $h$, such that 
\begin{equation}\label{caseb}
\mc_{h}\supseteq\mc_j\supseteq\mc_{S_2},	
\end{equation}
then we have the identifiability of $\theta^-_j$, as proved in the following.

\smallskip
From the proof of (a) we can obtain the identifiability of $\theta^+_{h}$, i.e., 
$\theta^+_{h} = \bar\theta^+_{h}$.
Define $\ttt^* = \theta^+_{h} \ee_{h}$.
Consider the two row vectors corresponding to response pattern $\rr^*=\ee_h$ in the transformed $T$-matrices $T(\TT-\ttt^*\mathbf1^\top) $ and $T(\bar\TT-\ttt^*\mathbf1^\top) $, respectively, and we have
\begin{equation}\label{eq-thm11}
\begin{aligned}
 T_{\ee_{h},[\aaa]}(\ttt^+ - \ttt^*, \ttt^- - \ttt^*) 
=  \begin{cases}  
\theta^-_h - \theta^+_h, &  \mbox{for } [\aaa]\in \mc_{h}^c; \\
0, & \text{otherwise}.
\end{cases}\\
 T_{\ee_{h},[\aaa]}(\bar\ttt^+ - \ttt^*, \bar\ttt^- - \ttt^*)  
=  \begin{cases}  
\bar\theta^-_h - \theta^+_h, &  \mbox{for } [\aaa]\in \mc_{h}^c; \\
0, & \text{otherwise}.
\end{cases}
\end{aligned}
\end{equation}
Moreover, we have
$$T_{\ee_{h},\Cdot}(\ttt^+ - \ttt^*, \ttt^- - \ttt^*)\nnu=T_{\ee_{h},\Cdot}(\bar\ttt^+ - \ttt^*, \bar\ttt^- - \ttt^*)\bar\nnu\neq 0.$$
Since 
$\mc_h^c\subseteq \mc_j^c$, Equation (\ref{eq1}) indicates
\begin{eqnarray}\label{eq-thm1-denom2}
\theta^-_j &=& 
\frac{T_{\ee_{h}+\ee_j,\Cdot}(\ttt^+-\ttt, \ttt^--\ttt)\nnu}{T_{\ee_{h},\Cdot}(\ttt^+-\ttt, \ttt^- -\ttt)\nnu} \notag\\
&=& 
\frac{T_{\ee_{h}+\ee_j,\Cdot}(\bar\ttt^+-\ttt, \bar\ttt^--\ttt)\bar\nnu}{T_{\ee_{h},\Cdot}(\bar\ttt^+-\ttt, \bar\ttt^- -\ttt)\bar\nnu}
~=~ \bar \theta^-_j.
\end{eqnarray}
\end{enumerate}
With the above results in cases (a) and (b), we show that (C1) ensures the identifiability of $(\ttt^+,\ttt^-_{non})$. Specifically, if condition (C1) is satisfied, then for each item $j$, there exist two item sets $S_1$ and $S_2$ satisfying \eqref{casea}. Thus the   result  for case  (a)  implies that  
   the items parameters  $\ttt^+$     are identifiable.
Moreover, for any non-basis item $j$, by definition there must exist an item $h$ such that $\mc_{h}\supseteq\mc_j$; condition (C1) further guarantees that there exists another set $S_2$ not containing $j$ such that $\{h\}\cap S_2=\varnothing$ and $\mc_j\supseteq \mc_{S_2}$. Therefore,   \eqref{caseb} is satisfied  and the   result for case  (b) implies   
 that  $\theta_j^-$ is identifiable for all $j\in S_{non}$.

\paragraph{\bf Step 2.}
This step proves that when   (C2) additionally holds, the parameter $\theta_j^-$ of each basis item $j$ is identifiable.  
Following the definition of the sequentially expanding procedure in (C2),
we first prove that in each expanding step,  $\theta_j^- = \bar\theta_j^-$ for all $j\in S_{sep}$, namely, every item $j$ included into the separator set through the expanding procedure has its lower level parameter $\theta_j^-$ identifiable.
To show this, it suffices  to prove the result that if an item $j$ is set $S$-differentiable and $\theta_h^- = \bar\theta_h^-$, $\theta_h^+ = \bar\theta_h^+$ for any $h\in S$, then $\theta_j^- = \bar\theta_j^-$.

If $j$ is $S$-differentiable, by definition there exist two item sets $S^+_j$, $S^-_j\subseteq S$ that are not necessarily disjoint such that $\mc_{S^-_j}\backslash \mc_{S^+_j}\subseteq\mc_j^c$.
Define 
\[
\ttt^* = \sum_{h\in S_j^+\cup S_j^-} \theta^-_{h} \ee_{h},
\]
and define response patterns
\[
\rr^+ = \sum_{h\in S_j^+} \ee_{h},\quad \rr^- = \sum_{h\in S_j^-} \ee_{h}.
\]
Note that the nonzero entries of the row vectors $T_{\rr^+,\Cdot}(\ttt^+-\ttt^*,\ttt^--\ttt^*)$  and $T_{\rr^-,\Cdot}(\ttt^+-\ttt^*,\ttt^--\ttt^*)$ correspond to the capable classes of $S_j^+$ and $S_j^-$, respectively. Specifically, 
\[\begin{aligned}
T_{\rr^+,[\aaa]}(\ttt^+-\ttt^*,\ttt^--\ttt^*) &= 
\begin{cases}
\prod_{j\in S_j^+} (\theta_j^+ - \theta_j^-), & [\aaa]\in\mc_{S_j^+};\\
0, & [\aaa]\notin\mc_{S_j^+}.\\
\end{cases}\\
T_{\rr^-,[\aaa]}(\ttt^+-\ttt^*,\ttt^--\ttt^*) &= 
\begin{cases}
\prod_{j\in S_j^-} (\theta_j^+ - \theta_j^-), & [\aaa]\in\mc_{S_j^-};\\
0, & [\aaa]\notin\mc_{S_j^-}.\\
\end{cases}
\end{aligned}\]
We define a linear transformation of the above   vectors $T_{\rr^+,\Cdot}(\ttt^+-\ttt^*,\ttt^--\ttt^*)$  and $T_{\rr^-,\Cdot}(\ttt^+-\ttt^*,\ttt^--\ttt^*)$   as
\[
T_{(\rr^- + k\cdot\rr^+),\Cdot}(\ttt^+-\ttt^*,\ttt^--\ttt^*) : = T_{\rr^-,\Cdot}(\ttt^+-\ttt^*,\ttt^--\ttt^*) + k\cdot T_{\rr^+,\Cdot}(\ttt^+-\ttt^*,\ttt^--\ttt^*)
\] 
where 
\[
k = -\frac{\prod_{j\in S_j^-} (\theta_j^+ - \theta_j^-)}{\prod_{j\in S_j^+} (\theta_j^+ - \theta_j^-)}\neq 0.
\]
Since  the capable classes of $S_j^+$ must also be capable classes of $S_j^-$, 
  we have 
\[
T_{(\rr^- + k\cdot\rr^+),[\aaa]}(\ttt^+-\ttt^*,\ttt^--\ttt^*) = 
\begin{cases}
\prod_{j\in S_j^-} (\theta_j^+ - \theta_j^-), & [\aaa]\in\mc_{S^-_j}\setminus \mc_{S^+_j};\\
0, & \text{otherwise}.\\
\end{cases}
\]
Under the assumption that    
$\theta_h^- = \bar\theta_h^-$, $\theta_h^+ = \bar\theta_h^+$ for any $h\in S$, we also have
\[
T_{(\rr^- + k\cdot\rr^+),[\aaa]}(\bar\ttt^+-\ttt^*,\bar\ttt^--\ttt^*) = 
\begin{cases}
\prod_{j\in S_j^-} (\theta_j^+ - \theta_j^-), & [\aaa]\in\mc_{S^-_j}\setminus \mc_{S^+_j};\\
0, & \text{otherwise}.\\
\end{cases}
\]
Note that the condition  $\mc_{S^-_j}\setminus \mc_{S^+_j} \subseteq \mc_j^c$ implies for any $[\aaa]\in\mc_{S^-_j}\backslash \mc_{S^+_j}$, one must have $[\aaa]\in\mc_j^c$. Since 
$$T_{(\rr^- + k\cdot\rr^+),\Cdot}(\ttt^+-\ttt^*,\ttt^--\ttt^*)\nnu = T_{(\rr^- + k\cdot\rr^+),\Cdot}(\bar\ttt^+-\ttt^*,\bar\ttt^--\ttt^*)\bar\nnu \neq0,$$ 
Since $j\notin (S_j^-\cup S_j^+)$, Equation (\ref{eq1}) implies
\[\begin{aligned}
\theta_j^-
=& ~\frac{\{T_{ \ee_j,\Cdot}(\ttt^+-\ttt^*,\ttt^--\ttt^*)\odot         
T_{ (\rr^- + k\cdot\rr^+),\Cdot}(\ttt^+-\ttt^*,\ttt^--\ttt^*)\}\nnu}
{T_{(\rr^- + k\cdot\rr^+),\Cdot}(\ttt^+-\ttt^*,\ttt^--\ttt^*)\nnu}\\
=& ~\frac{\{T_{ \ee_j,\Cdot} (\bar\ttt^+-\ttt^*,\bar\ttt^--\ttt^*)
\odot
T_{(\rr^- + k\cdot\rr^+),\Cdot}(\bar\ttt^+-\ttt^*,\bar\ttt^--\ttt^*)\}\bar\nnu}
{T_{(\rr^- + k\cdot\rr^+),\Cdot}(\bar\ttt^+-\ttt^*,\bar\ttt^--\ttt^*)\bar\nnu}\\
= &~\bar\theta_j^-,
\end{aligned}\]
where $\odot$ denotes the element-wise product of two vectors.
This proves the claim that if item $j$ is set $S$-differentiable and $\theta_h^- = \bar\theta_h^-$, $\theta_h^+ = \bar\theta_h^+$ for any $h\in S$, then $\theta_j^- = \bar\theta_j^-$.
Together with the result in Step 1 and  the definition of the sequentially expanding procedure in (C2), we therefore have the identifiability of  $(\ttt^+, \ttt^-)$.

With   $(\ttt^+, \ttt^-) = (\bar\ttt^+, \bar\ttt^-)$, Equation \eqref{eq1} simplifies   to $ T(\ttt^+, \ttt^-) \pp = T(\ttt^+, \ttt^-)\bar\pp=0$. The last part of the proof of Propositions 1 and 4 then gives the identifiability of $\nnu.$ 
This completes the proof of the theorem.
\end{proof}

\smallskip
\begin{proof}[Proof of Proposition \ref{prop-after1}]
For ease of discussion, in this proof we  use $T(\ttt^+, \ttt^-\mid\Gamma(S))$ to denote the $T$-matrix   associated with any $S$-adjusted design matrix $\Gamma(S)$ and  item parameters $(\ttt^+, \ttt^-)$.

For any $S$-adjusted $\Gamma(S)$-matrix, we define another set of item parameters    $\tilde\ttt^+ =(\tilde\theta_1^+,\ldots,\tilde\theta_J^+) \mbox{ and } \tilde\ttt^- =(\tilde\theta_1^-,\ldots,\tilde\theta_J^-),$  where  
$\tilde\theta^-_j = \theta^+_j$, $\tilde\theta^+_j = \theta^-_j$ for all $j\in S$, and $\tilde\theta^-_j = \theta^-_j$, $\tilde\theta^+_j = \theta^+_j$ for all $j\notin S$. 
We first show that   the $T$-matrix $T(\tilde \ttt^+, \tilde \ttt^-\mid\Gamma(S))$   can be viewed as the $T$-matrix associated with the original $\Gamma$-matrix  with  item parameters  $(\ttt^+, \ttt^-)$, i.e.,  
\begin{equation}\label{feb}
T_{\rr,\aaa}(\tilde\ttt^+, \tilde\ttt^-\mid\Gamma(S)) = T_{\rr,\aaa}( \ttt^+,  \ttt^-\mid\Gamma).	
\end{equation}
To show this, 
note that for any response pattern $\rr\in\{0,1\}^J$, we have
$$
 T_{\rr,\aaa}(\ttt^+, \ttt^-\mid \Gamma) =
\prod_{j: r_j=1} \Big[\Gamma_{j,\aaa}\theta^+_j + (1-\Gamma_{j,\aaa})\theta_j^- \Big]$$ and 
\[\begin{aligned}
& T_{\rr,\aaa}(\tilde\ttt^+, \tilde\ttt^-\mid\Gamma(S))\\
=&  
\prod_{j: r_j=1} \Big[\{\Gamma(S)\}_{j,\aaa}\tilde\theta^+_j + (1-\{\Gamma(S)\}_{j,\aaa})\tilde\theta_j^- \Big] \\
=& \prod_{j\in S: r_j=1} \Big[(1-\Gamma_{j,\aaa})\tilde\theta^+_j + \Gamma_{j,\aaa}\tilde\theta_j^- \Big]\times  \prod_{j\notin S: r_j=1} \Big[\Gamma_{j,\aaa}\tilde\theta^+_j + (1-\Gamma_{j,\aaa})\tilde\theta_j^- \Big] \\
=& \prod_{j\in S: r_j=1} \Big[(1-\Gamma_{j,\aaa})\theta^-_j + \Gamma_{j,\aaa}\theta_j^+ \Big]\times \prod_{j\notin S: r_j=1} \Big[\Gamma_{j,\aaa}\theta^+_j + (1-\Gamma_{j,\aaa})\theta_j^- \Big] \\
=&
\prod_{j: r_j=1} \Big[\Gamma_{j,\aaa}\theta^+_j + (1-\Gamma_{j,\aaa})\theta_j^- \Big]\\
=& ~T_{\rr,\aaa}(\ttt^+, \ttt^-\mid \Gamma).
\end{aligned}\]
   
 With the result in  \eqref{feb}, to prove Proposition \ref{prop-after1},  it suffices to show that the identifiability  argument in Theorem \ref{thm1} still holds if the $\Gamma$-induced restrictions  of the item parameters,  $\theta_j^+ > \theta_j^- $ for all $j=1,\ldots,J$, are  replaced by the constraints that  $ \theta_j^+ < \theta_j^-$ for  any $j\in S$ and $ \theta_j^+ > \theta_j^-$ for any $j\notin S$.

We next prove this claim. 
If   $\theta_j^+<\theta_j^-$ for some items $j$, the conclusion  in Lemma \ref{lemmaS2} still holds. In particular, following the proof of Lemma \ref{lemmaS2}, if   $\theta_j^+<\theta_j^-$, then 
\[
\theta_j^+ = \sum_{\aaa\in\ma} \theta_j^+ p_{\aaa} \leq 
\sum_{\aaa\in\ma} \theta_{j,\aaa} p_{\aaa} = \sum_{\aaa\in\ma} \bar\theta_{j,\aaa}\bar p_{\aaa} \leq \sum_{\aaa\in\ma}\bar\theta_j^-\bar p_{\aaa} = \bar\theta_j^-,
\]
where among the two ``$\leq$" there is at least a strict ``$<$". This implies  $\theta_j^+ \neq \bar\theta_j^-$ for all $j=1,\ldots,J$, and a similar argument gives $\theta_j^- \neq \bar\theta_j^+$ for all $j=1,\ldots,J$.  
With these results, we can check that  all the needed inequalities in 
the proof of Theorem \ref{thm1} still hold and all the proof steps proceed with no changes.
This proves the conclusion of the proposition.
\end{proof}

\smallskip
Next we prove  the identifiability results for the two-parameter $Q$-restricted models. We say a $Q$-matrix of size $J\times K$ is \textit{complete} for the two-parameter model, if after some row permutation it contains an identity submatrix $\mathcal I_K$.
Under   the conjunctive model assumption, let 
\begin{equation}\label{RU}
	\mathcal R^Q=\mathcal R^{Q,conj}=\{\mz_K^\top\}\cup\{\aaa=\vee_{h\in S}\,\bq_h: \forall S \subset \mathcal S\} 
\end{equation}
 be defined as in Remark 1 of the main text.
 Since elements of $\mathcal R^Q$ are $K$-dimensional binary vectors, they can be viewed as attribute profiles and  $\mathcal R^Q \subseteq \{0,1\}^K$. When $Q$ is complete, clearly $\mathcal R^Q = \{0,1\}^K$. The row-union space $\mathcal R^Q$ has the following two properties.
First, every two attribute profiles in $\mathcal R^Q$ have different ideal response vectors, i.e.
\begin{equation}\label{eq-gamma1}
\forall \aaa_1,\aaa_2\in\mathcal R^Q,~~\aaa_1\neq\aaa_2,\quad \Gamma_{\Cdot,\,\aaa_1} = \Gamma_{\Cdot,\,\aaa_2}.
\end{equation}
Second, when $Q$ is incomplete, for any attribute profile $\aaa\in\{0,1\}^K$, there must exist some $\aaa'\in\mathcal R^Q$ that has the same ideal response vector as $\aaa$, i.e.
\begin{equation}\label{eq-gamma2}
\forall \aaa\in\{0,1\}^K,\quad \exists \aaa'\in\mathcal R^Q~~\text{such that}~~\aaa\succeq\aaa'~\text{and}~\Gamma_{\Cdot,\,\aaa} = \Gamma_{\Cdot,\,\aaa'}.
\end{equation}
 Based on the above two properties, when $\ma$ is saturated, $\mathcal R^{Q}$   is a complete set of representatives of the conjunctive   equivalence classes.
 Similarly, we can show $\mathcal R^{Q,comp}  =\{\mo_K^\top-\aaa: \aaa\in\mathcal R^{Q}\}$ gives a complete set of representatives of the compensatory   equivalence classes.
Therefore, this proves the claims in Remark 1 of the main text.
In the following proofs of Corollary \ref{cor-dina-main}, Theorem \ref{thm_cond}, Theorem \ref{thm-com-necc} and Theorem \ref{cor-suffnece} for the two-parameter $Q$-restricted models, when there is no ambiguity, we will exchangeably say an equivalence class $[\aaa]$ is induced by the $\Gamma$-matrix or is induced by the corresponding $Q$-matrix.
 



\smallskip

\begin{proof}[Proof of Corollary \ref{cor-dina-main}]
With definitions of non-basis and basis items introduced in \eqref{eq-defdina} and definition of $S$-differentiable item introduced in \eqref{eq-defdiff-dina}, conditions (C1) and (C2) exactly reduce to the new conditions (C1$^*$) and (C2$^*$) regarding the $Q$-matrix for the two-parameter conjunctive model, therefore by Theorem \ref{thm1}, (C1$^*$) and (C2$^*$) are sufficient for the $\pp$-partial identifiability of the conjunctive models.

On the other hand, for the two-parameter compensatory model, if the $Q$-matrix satisfies the new conditions (C1$^*$) and (C2$^*$), then we have that $\Gamma^{conj}$ satisfies the original conditions (C1) and (C2).
Given an arbitrary $Q$-matrix, by the definition of the conjunctive $\Gamma^{conj}$ and compensatory $\Gamma^{comp}$, for any item $j$ and any attribute profile $\aaa\in\{0,1\}^K$, we can obtain
\begin{equation}\label{eq-gg}
\Gamma^{comp}_{j,\,\aaa} = 1 - \Gamma^{conj}_{j,\,\mo-\aaa} = I(\alpha_k=1~\text{for some}~k~\text{s.t.}~q_{j,k}=1),
\end{equation}
where $\mo-\aaa=(1-\alpha_1,\ldots,1-\alpha_K)$.
This means the two matrices $\Gamma^{conj}$ and $\mo_{J\times C}-\Gamma^{comp}$ only differ by a column permutation. 
Noting that conditions (C1) and (C2) do not depend on the order of the column vectors, so if $\Gamma^{conj}$ satisfies (C1) and (C2), then $\mo_{J\times C}-\Gamma^{comp}$ also satisfies (C1) and (C2). Then Proposition \ref{prop-after1} implies the two parameter compensatory model with design matrix $\Gamma^{comp}$ is $\pp$-partially identifiable. 
\end{proof}

\smallskip
\begin{proof}[Proof of Theorem \ref{thm_cond}]
Without loss of generality, we focus on the proof of the conclusion for the two-parameter conjunctive model,  and all the arguments also hold for the two-parameter disjunctive model, following the similar argument in the proof of Proposition \ref{prop-after1}. 
In the following, we first present the proof of part (a), then that of part (b.2), and finally that of part (b.1).

\vspace{2mm}
\noindent
\textbf{Proof of  part (a)}.
 Without loss of generality, assume the $Q$-matrix takes the following form
\[
Q=\left(
\begin{array}{cc}
1 & \vv_1 \\
\mz & Q'
\end{array}
\right),
\]
where $Q'$ is a submatrix of size $(J-1)\times(K-1)$ and $\vv_1$ is a $(K-1)$-dimensional vector.
For any attribute profile $\aaa=(0,\aaa_{2:K})$, denote $\aaa+\ee_1 = (1,\aaa_{2:K})$; and for any $\aaa=(1,\aaa_{2:K})$, denote $\aaa-\ee_1 = (0,\aaa_{2:K})$.
Consider any valid set of parameters $(\ttt^+,\ttt^-,\nnu)$. To prove the conclusion in (A), we next construct another set of parameters $(\bar\ttt^+,\bar\ttt^-,\bar\nnu)\neq (\ttt^+,\ttt^-,\nnu)$ but $T(\ttt^+, \ttt^-)\nnu = T(\bar\ttt^+,\ttt^-)\bar{\nnu}$.
In particular, we set  $\bar\ttt^- = \ttt^-$, $\bar\theta^+_j=\theta^+_j$ for $j=2,\ldots,J$, and choose $\bar\theta^+_1$ close enough but not equal to $\theta^+_1$. 
Define
\[\begin{aligned}
\mathcal R_0 &= \{\aaa\in\mathcal R^Q: \alpha_1=0, \aaa\succeq(0,\vv_1)\},\\ 
\mathcal R_1 &= \{\aaa\in\mathcal R^Q: \alpha_1=1, \aaa\succeq(0,\vv_1)\},
\end{aligned}\]
then we can see that the two sets $\mathcal R_0$ and $\mathcal R_1$ are disjoint and their elements are paired in the sense that for any $\aaa\in\mathcal R_0$, one has $\aaa+\ee_1\in\mathcal R_1$ and for any $\aaa\in\mathcal R_1$, one has $\aaa-\ee_1\in\mathcal R_0$.
To construct the proportion parameters $\bar\nnu$, we set
\begin{equation}\label{eq-2p-cons}
\begin{cases}
\bar\nu_{[\aaa]} = \nu_{[\aaa]} + \Big(1-\frac{\theta^+_1 -\theta_1^- }{ \bar\theta^+_1-\theta_1^-} \Big)\nu_{[\aaa+\ee_1]}, & \forall\aaa\in\mathcal R_0; \\
\bar\nu_{[\aaa]} = \frac{\theta^+_1 -\theta_1^- }{ \bar\theta^+_1-\theta_1^-} {\nu}_{[\aaa]},  & \forall\aaa\in\mathcal R_1; \\
\bar\nu_{[\aaa]} =\nu_{[\aaa]}, & \forall\aaa\in\mathcal R^Q\setminus(\mathcal R_0\cup \mathcal R_1).
\end{cases}\end{equation}
For notational simplicity, denote $\mathcal R_c =\mathcal R^Q\setminus(\mathcal R_0\cup \mathcal R_1)$.
Next we show that under the two different sets of parameters $(\ttt^+,\ttt^-,\nnu)$ and $(\bar\ttt^+,\ttt^-,\bar\nnu)$, for any response pattern $\rr\in\{0,1\}^J$,
\begin{equation}\label{eq-2par}
T_{\rr,\Cdot}(\bar\ttt^+-\ttt^-,\boldsymbol{0})\bar{\nnu} = T_{\rr,\Cdot}(\ttt^+-\ttt^-,\boldsymbol{0})\bar{\nnu},
\end{equation}
which will complete the proof.
To this end, we consider two types of response patterns $\rr=(r_1,\ldots,r_J)$ respectively in the following: (a) $r_1=0$; and (b)  $r_1=1$.

\begin{itemize}
\item[(a)]
Firstly, for any $\rr\in\{0,1\}^J$ such that $r_1=0$, $T_{\rr,\Cdot}(\ttt^+-\ttt^-,\boldsymbol{0})=T_{\rr,\Cdot}(\bar\ttt^+-\ttt^-,\boldsymbol{0})$, so by our construction,
\begin{align*}
& T_{\rr,\Cdot}(\bar\ttt^+-\ttt^-,\boldsymbol{0})\bar{\nnu} = T_{\rr,\Cdot}(\ttt^+-\ttt^-,\boldsymbol{0})\bar{\nnu} \\
=& \sum_{\aaa\in\mathcal R^Q}T_{\rr,[\aaa]}(\ttt^+-\ttt^-,\boldsymbol{0})\bar\nu_{[\aaa]} \\
=& \sum_{\aaa\in\mathcal R_0}T_{\rr,[\aaa]}(\ttt^+-\ttt^-,\boldsymbol{0}) \bar\nu_{[\aaa]} 
+ \sum_{\aaa\in\mathcal R_1}T_{\rr,[\aaa]}(\ttt^+-\ttt^-,\boldsymbol{0})\bar\nu_{[\aaa]} \\
&+  \sum_{\aaa\in\mathcal R_c}T_{\rr,[\aaa]}(\ttt^+-\ttt^-,\boldsymbol{0})\bar\nu_{[\aaa]}\allowdisplaybreaks\\ 
=& \sum_{\aaa\in\mathcal R_0}T_{\rr,[\aaa]}(\ttt^+-\ttt^-,\boldsymbol{0}) \left(\nu_{[\aaa]} + \left(1-\frac{\theta^+_1 -\theta_1^- }{ \bar\theta^+_1-\theta_1^-} \right)\nu_{[\aaa+\ee_1]} \right) \\
&+ \sum_{\aaa\in\mathcal R_1}T_{\rr,[\aaa]}(\ttt^+-\ttt^-,\boldsymbol{0})\left(\frac{\theta^+_1 -\theta_1^- }{ \bar\theta^+_1-\theta_1^-}{\nu}_{[\aaa]}\right)  \\
&+  \sum_{\aaa\in\mathcal R_c}T_{\rr,[\aaa]}(\ttt^+-\ttt^-,\boldsymbol{0})\nu_{[\aaa]}\\
:=& I_0 + I_1 + I_c.
\end{align*}
Note that the elements in $\mathcal R_0$ and $\mathcal R_1$ are paired, and moreover, for any pair of attribute profiles $(\aaa,\aaa+\ee_1)$ where $\aaa\in\mathcal R_0$ and $\aaa+\ee_1\in\mathcal R_1$, we have 
\begin{equation}\label{eq-e1}
T_{\rr,[\aaa]}(\ttt^+-\ttt^-,\boldsymbol{0}) = T_{\rr,[\aaa+\ee_1]}(\ttt^+-\ttt^-,\boldsymbol{0}) = \prod_{j: r_j=1}(\theta_{j,[\aaa]} - \theta_j^-)\end{equation}
for any type-(a) response pattern $\rr$, namely $\rr\in\{0,1\}^J$ such that $r_1=0$. Equation \eqref{eq-e1} leads to
\[\begin{aligned}
I_1 = & \sum_{\aaa\in\mathcal R_0}T_{\rr,[\aaa+\ee_1]}(\ttt^+-\ttt^-,\boldsymbol{0})\left(\frac{\theta^+_1 -\theta_1^- }{ \bar\theta^+_1-\theta_1^-}{\nu}_{[\aaa+\ee_1]}\right)  \\
= & \sum_{\aaa\in\mathcal R_0}T_{\rr,[\aaa]}(\ttt^+-\ttt^-,\boldsymbol{0})\left(\frac{\theta^+_1 -\theta_1^- }{ \bar\theta^+_1-\theta_1^-} {\nu}_{[\aaa+\ee_1]}\right). \\
\end{aligned}\]
Therefore we have
\[
\begin{aligned}
&I_0+I_1 \\
=& \sum_{\aaa\in\mathcal R_0}T_{\rr,[\aaa]}(\ttt^+-\ttt^-,\boldsymbol{0}){\small\left(\nu_{[\aaa]} + \left(1-\frac{\theta^+_1 -\theta_1^- }{ \bar\theta^+_1-\theta_1^-}\right)\nu_{[\aaa+\ee_1]} + \frac{\theta^+_1 -\theta_1^- }{ \bar\theta^+_1-\theta_1^-} {\nu}_{[\aaa+\ee_1]} \right)} \\
=& \sum_{\aaa\in\mathcal R_0}T_{\rr,[\aaa]}(\ttt^+-\ttt^-,\boldsymbol{0})\Big(\nu_{[\aaa]} +  {\nu}_{[\aaa+\ee_1]} \Big) \\
=& \sum_{\aaa\in\mathcal R_0}T_{\rr,[\aaa]}(\ttt^+-\ttt^-,\boldsymbol{0})\nu_{[\aaa]} + \sum_{\aaa\in\mathcal R_1}T_{\rr,[\aaa]}(\ttt^+-\ttt^-,\boldsymbol{0})\nu_{[\aaa]},
\end{aligned}
\]
where the last equality also results from \eqref{eq-e1}. This further results in
\[
I_0+I_1+I_c = \sum_{\aaa\in\mathcal R^Q}T_{\rr,[\aaa]}(\ttt^+-\ttt^-,\boldsymbol{0})\nu_{[\aaa]} = T_{\rr,\Cdot}(\ttt^+-\ttt^-,\boldsymbol{0})\nnu.
\]
This proves that for any $\rr$ such that $r_1=0$, Equation \eqref{eq-2par} holds.

\smallskip
\item[(b)]
Secondly, consider the type-(b) response pattern, namely those $\rr=(1,r_2,\allowbreak\ldots,r_J)$. For such $\rr$, denote $\rr-\ee_1 = (0,r_2,\ldots,r_J)$, then 
\[
T_{\rr,[\aaa]}(\bar\ttt^+-\ttt^-,\boldsymbol{0}) =
\begin{cases}
(\bar\theta_1^+ - \theta_1^-)\cdot T_{\rr-\ee_1,[\aaa]}(\ttt^+-\ttt^-,\boldsymbol{0}), & \aaa\succeq(1,\vv_1);\\
0, & \aaa\succeq(1,\vv_1),
\end{cases}
\]
which indicates  $T_{\rr,[\aaa]}(\bar\ttt^+-\ttt^-,\boldsymbol{0})=0$ for all $\aaa\in\mathcal R_0\cup \mathcal R_c$. This is because for $\aaa\in\mcr_0$, $\alpha_1=0\ngeq 1$; and for $\aaa\in\mcr_c$, $(\alpha_2,\ldots,\alpha_K)\nsucceq\vv_1$ by our definitions.
Therefore,
\[\begin{aligned}
&T_{\rr,\Cdot}(\bar\ttt^+-\ttt^-,\boldsymbol{0})\bar{\nnu} 
= \sum_{\aaa\in\mathcal R^Q} T_{\rr,[\aaa]}(\bar\ttt^+-\ttt^-,\boldsymbol{0})\bar{\nu}_{[\aaa]} \\
=& \sum_{\aaa\in\mathcal R^Q\atop\,\aaa\succeq(1,\vv_1)} T_{\rr-\ee_1,[\aaa]}(\ttt^+-\ttt^-,\boldsymbol{0}) (\bar\theta_1^+ - \theta_1^-) \bar{\nu}_{[\aaa]} \\
=& \sum_{\aaa\in\mathcal R_1}T_{\rr-\ee_1,[\aaa]}(\ttt^+-\ttt^-,\boldsymbol{0}) (\bar\theta_1^+ - \theta_1^-)\bar{\nu}_{[\aaa]}\\
=& \sum_{\aaa\in\mathcal R_1}T_{\rr -\ee_1,[\aaa]}(\ttt^+-\ttt^-,\boldsymbol{0})(\theta^+_1-\theta_1^-)\nu_{[\aaa]}\\
=& T_{\rr,\Cdot}(\ttt^+-\ttt^-,\boldsymbol{0})\nnu,
\end{aligned}\]
where our previous construction $(\bar\theta^+_1-\theta_1^-)\bar{\nu}_{[\aaa]}=(\theta^+_1-\theta_1^-){\nu}_{[\aaa]}$ for $\aaa\in\mathcal R_1$ defined in \eqref{eq-2p-cons} is used to obtain the last but second equality. This proves that for any $\rr$ such that $r_1=1$, Equation \eqref{eq-2par} holds.
\end{itemize}

Now that we have proved Equation \eqref{eq-2par} holds for any $\rr\in\{0,1\}^J$, we have found two different sets of parameters $(\ttt^+,\nnu)\neq(\bar{\ttt}^+,\bar{\nnu})$ that give $T(\ttt^+, \ttt^-)\nnu = T(\bar\ttt^+,\ttt^-)\bar{\nnu}$. This shows the non-identifiability of the parameters $(\ttt^+,\ttt^-,\nnu)$, and concludes the proof of part (A).

\vspace{2mm}
\noindent
\textbf{Proof of Part (b.2)}.
Equation \eqref{eq1} is equivalent to
\begin{equation}\label{eq-tpr}
 T_{\rr,\Cdot}(\TT) \nnu =  T_{\rr,\Cdot}(\bar\TT) \bar\nnu\quad\text{for all}\quad \rr=(r_1,\ldots,r_J)^\top\in\{0,1\}^J.
\end{equation}
The detailed form of \eqref{eq-tpr} can be written as follows, for any $\rr\in\{0,1\}^J$,
\begin{equation}\label{eq-tprr}
\sum_{\aaa\in\mathcal R^{Q}}\prod_{r_j=1}\theta_{j,\,[\aaa]}\cdot \nu_{[\aaa]} = 
\sum_{\aaa\in\mathcal R^{Q}}\prod_{r_j=1}\bar\theta_{j,\,[\aaa]}\cdot\bar \nu_{[\aaa]},
\end{equation}
where $\mathcal R^Q$ denotes the row-union space of the $Q$-matrix $Q$ as in \eqref{RU}. 
For any attribute profile $\aaa\in\{0,1\}^K$,   $[\aaa]$ denotes the equivalence class containing $\aaa$ that is induced by $Q$.
Let $\aaa_{2:K}$ denote the vector containing last $K-1$ elements of $\aaa$, so $\aaa$ can be written as $\aaa=(\alpha_1,\aaa_{2:K})$ and $[\alpha_1,\aaa_{2:K}]$ represents the equivalence class $\aaa$ belongs to.
 Recall that we use $\RR=(R_1,\ldots,R_J)$ to denote a random response vector ranging in $\{0,1\}^J$, and use $\mathbf A=(A_1,\ldots,A_K)$ to denote a random attribute profile ranging in the latent class space $\ma\subseteq\{0,1\}^K$. Denote $\mathbf A_{2:K}:=(A_2,\ldots,A_K)$. 

Under the assumptions of part (B), the $Q$-matrix takes the following form
\begin{equation*}
Q=
\left(\begin{array}{cc}
1 & \vv_1^\top \\
1 & \vv_2^\top \\
\mathbf 0 & Q'
\end{array}\right).
\end{equation*}
 For any two different equivalence classes $[0,\aaa_{2:K}]$ and $[1,\aaa_{2:K}]$ where $\aaa_{2:K}\in\{0,1\}^{K-1}$, their corresponding item parameters to any item $j>2$ are the same, i.e., for any $j>2$ and any $\aaa_{2:K}\in\{0,1\}^{K-1}$,
\begin{eqnarray}\label{eq-pp}
\mathbb P(R_j=1\mid \ba=(1,\aaa_{2:K}))
&=&\mathbb P(R_j=1\mid \ba=(0,\aaa_{2:K}))\\
&=&\theta_{j,[0,\aaa_{2:K}]}.\notag
\end{eqnarray}
Therefore for any response pattern in the form $\rr=(0,0,r_3,\ldots,r_J)$, \eqref{eq-tprr} for such $\rr$ can be equivalently written as 
\begin{align}\label{eq-qprime00}
 &\sum_{\aaa_{2:K}\in\mathcal R^{Q'}}\prod_{j>2\atop \,r_j=1}\theta_{j,\,[0,\aaa_{2:K}]}\cdot (\nu_{[0,\aaa_{2:K}]}+\nu_{[1,\aaa_{2:K}]})\\
  = 
&\sum_{\aaa_{2:K}\in\mathcal R^{Q'}}\prod_{j>2\atop\,r_j=1}\bar\theta_{j,\,[0,\aaa_{2:K}]}\cdot(\bar\nu_{[0,\aaa_{2:K}]}+\bar\nu_{[1,\aaa_{2:K}]}),\notag
\end{align}
where   $\mathcal R^{Q'} $ is the row-union space of  $Q'$, i.e.,  
\[\begin{aligned}
\mathcal R^{Q'} &= \{\mz_{K-1}^\top\}\cup\{\aaa=\vee_{h\in S}\,\bq'_h: \forall S \subseteq \{3,\ldots,J\}\}. 
\end{aligned}\]
\eqref{eq-qprime00} involves $2^{J-2}$ equations with $(r_3,\ldots,r_j)$ freely ranging in $\{0,1\}^{J-2}$, which indicates that $\theta_{j,\,[0,\aaa_{2:K}]}$ and $(\nu_{[0,\aaa_{2:K}]}+\nu_{[1,\aaa_{2:K}]})$ can be viewed as item parameter and proportion parameter associated with the model under the $(J-2)\times(K-1)$ sub-matrix $Q'$.
Since the sub-matrix $Q'$ satisfies conditions (C1$^*$) and (C2$^*$), Theorem \ref{thm1} and the set of equations \eqref{eq-qprime00} lead to
\[
\forall j\geq 3,\quad
\theta_{j,\,[0,\aaa_{2:K}]} = \bar\theta_{j,\,[0,\aaa_{2:K}]},\quad \nu_{[0,\aaa_{2:K}]}+\nu_{[1,\aaa_{2:K}]} = \bar \nu_{[0,\aaa_{2:K}]}+\bar \nu_{[1,\aaa_{2:K}]}.
\]
This implies for any item $j\geq 3$, the item parameters $\theta^+_j$ and $\theta^-_j$ associated with the original $Q$-matrix are identifiable.

Now consider an arbitrary response pattern $\rr=(r_1,r_2,r_3,\ldots,\allowbreak r_J)$. 
We claim that \eqref{eq-tprr} for $\rr$ can be equivalently written as
\begin{align}\label{eq-qprime-n0}
&\sum_{\aaa_{2:K}\in\mathcal R^{Q'}}\prod_{j>2\atop\,r_j=1}\theta_{j,\,[0,\aaa_{2:K}]}\cdot \mathbb P(R_1\geq r_1,\,R_2\geq r_2,\, \ba_{2:K}=\aaa_{2:K})\\
=
&\sum_{\aaa_{2:K}\in\mathcal R^{Q'}}\prod_{j>2\atop\,r_j=1}\bar\theta_{j,\,[0,\aaa_{2:K}]}\cdot\overline {\mathbb P}(R_1\geq r_1,\,R_2\geq r_2,\,\ba_{2:K}=\aaa_{2:K}),\notag
\end{align}
where $\mathbb P(R_1\geq r_1,\,R_2\geq r_2,\, \ba_{2:K}=\aaa_{2:K})$ represents the probability of $\{R_1\geq r_1,\,R_2\geq r_2\}$ and the attribute profile $\ba$ has its last $K-1$ entries being $\aaa_{2:K}$ under the set of model parameters $(\ttt^+,\ttt^-,\nnu)$, while $\overline {\mathbb P}(R_1\geq r_1,\,R_2\geq r_2,\, \ba_{2:K}=\aaa_{2:K})$ represents that under model parameters $(\bar\ttt^+,\bar\ttt^-,\bar\nnu)$.
The reason \eqref{eq-tprr} can be equivalently written as \eqref{eq-qprime-n0} is that, given any $\aaa_{2:K}\in\mathcal R^{Q'}$ and any item $j\in\{3,\ldots,J\}$, the positive response probability of $[\alpha_1,\aaa_{2:K}]$ to item $j$ only depends on $\aaa^*$ part, regardless of the value of $\alpha_1$,  as shown in \eqref{eq-pp}. 
Therefore the terms in $T(\TT)_{\rr,\Cdot}\nnu$ can be grouped in such a way that it becomes the summation over all the $\aaa_{2:K}\in\mathcal R^{Q'}$, exactly as presented in Equation \eqref{eq-qprime-n0}.

A key observation is that, taking $(r_1,r_2)$ to be $(0,1)$, $(1,0)$, $(1,1)$ in \eqref{eq-tpr} respectively, we obtain another three sets of equations expressed in the form of \eqref{eq-qprime-n0}, which are exactly in the same form as \eqref{eq-qprime00} by just replacing $\nu_{[0,\aaa_{2:K}]}$ by $\mathbb P(R_1\geq r_1,\,R_2\geq r_2, \ba_{2:K}=\aaa_{2:K})$. Actually, taking $(r_1,r_2)=(0,0)$ gives $\mathbb P(R_1\geq 0,\,R_2\geq 0, \ba_{2:K}=\aaa_{2:K})=\nu_{[0,\aaa_{2:K}]}$.
By Theorem \ref{thm1}, this key observation results in that, for any $(r_1,r_2)\in\{0,1\}^2$ and any $\aaa_{2:K}\in\mathcal R^{Q'}$,
\begin{equation}\label{eq-f}
\begin{aligned}
& \mathbb P(R_1\geq r_1,\,R_2\geq r_2,\, \ba_{2:K}=\aaa_{2:K})\\
&= ~
\overline {\mathbb P}(R_1\geq r_1,\,R_2\geq r_2,\, \ba_{2:K}=\aaa_{2:K}). \\
\end{aligned}
\end{equation}
We will rely on  \eqref{eq-qprime-n0} and the above equality \eqref{eq-f} to proceed with the proof.
Now consider two types of combinations of row vectors of $Q'$, categorized based on their relationships with $\vv_1$ and $\vv_2$. In the following proof, write $R_1\geq r_1,R_2\geq r_2$ succinctly as $\RR_{1:2}\succeq\rr_{1:2}$. 
We consider the following cases (a$^*$) and (b$^*$).
\begin{enumerate}
\item[(a$^*$)] In this case, there exists two row vectors $\vv_0$ and $\vv'_0$ of $Q'$ s.t. $\vv_0\succeq \vv_1$, $\vv_0\nsucceq \vv_2$, and $\vv'_0\nsucceq \vv_1$, $\vv'_0\succeq \vv_2$.
\newline
Consider $\ba_{2:K}=\vv_0$, then $\vv_0\succeq \vv_1$, $\vv_0\nsucceq \vv_2$ imply that
\begin{eqnarray*}
	&&\mathbb P(\RR_{1:2}\succeq \rr_{1:2},\,\ba_{2:K}=\vv_0) \\
	&&=
\begin{cases}
\nu_{[0,\vv_0]} + \nu_{[1,\vv_0]}, & (r_1,r_2)=(0,0);\\
\theta_1^-\cdot\nu_{[0,\vv_0]} + \theta^+_1\cdot\nu_{[1,\vv_0]}, & (r_1,r_2)=(1,0);\\
\theta_2^-\cdot(\nu_{[0,\vv_0]} + \nu_{[1,\vv_0]}), & (r_1,r_2)=(0,1);\\
\theta_2^-\cdot(\theta_1^-\cdot\nu_{[0,\vv_0]} + \theta^+_1\cdot\nu_{[1,\vv_0]}), &(r_1,r_2)=(1,1).\\
\end{cases}
\end{eqnarray*}
Note that $\overline {\mathbb P}(\RR_{1:2}\succeq \rr_{1:2},\ba_{2:K}=\aaa_{2:K})$ takes the similar form as $\mathbb P(\RR_{1:2}\succeq \rr_{1:2},\ba_{2:K}=\aaa_{2:K})$, so in order to ensure \eqref{eq-f} the following equations must hold
\begin{eqnarray} \label{cond_c1}
 & & \quad   \begin{cases}
\nu_{[0,\vv_0]} + \nu_{[1,\vv_0]} = \bar\nu_{[0,\vv_0]} + \bar\nu_{[1,\vv_0]};\\
\theta^-_1\cdot\nu_{[0,\vv_0]} + \theta^+_1\cdot\nu_{[1,\vv_0]} = \bar \theta^-_1\cdot\bar\nu_{[0,\vv_0]} + \bar \theta^+_1\cdot\bar\nu_{[1,\vv_0]};\\
\theta^-_2\cdot(\nu_{[0,\vv_0]} + \nu_{[1,\vv_0]}) = \bar \theta^-_2\cdot(\bar\nu_{[0,\vv_0]} + \bar\nu_{[1,\vv_0]});\\
\theta^-_2\cdot(\theta^-_1\nu_{[0,\vv_0]} + \theta^+_1\nu_{[1,\vv_0]}) = \bar \theta^-_2\cdot(\bar \theta^-_1\bar\nu_{[0,\vv_0]} + \bar \theta^+_1\bar\nu_{[1,\vv_0]}).\\
\end{cases}
\end{eqnarray}
Taking the ratio of the third and the first equation above gives $\theta^-_2 = \bar \theta^-_2$. Similarly, $\vv'_0\nsucceq \vv_1$, $\vv'_0\succeq \vv_2$ also imply $\theta^-_1 = \bar \theta^-_1$. Plugging $\theta^-_1 = \bar \theta^-_1$ back to the second equation in (\ref{cond_c1}) gives $\theta^+_1 = \bar \theta^+_1$, and similarly $\theta^+_2 = \bar \theta^+_2$.

\item[(b$^*$)] In case (b$^*$), there exist two row vectors $\vv_0$, $\vv'_0$ of $Q'$ such that $\vv_0\nsucceq \vv_1$, $\vv_0\nsucceq \vv_2$, and $\vv'_0\succeq \vv_1$, $\vv'_0\succeq \vv_2$.
\newline
Consider $\ba_{2:K}=\vv_0$, then $\vv_0\nsucceq \vv_1$, $\vv_0\nsucceq \vv_2$ imply that the attribute profiles $(1,\vv_0)$, $(0,\vv_0)$ both belong to the same equivalence class $[1,\vv_0]$ induced by $Q$, and hence
\[\mathbb P(\RR_{1:2}\succeq \rr_{1:2},\,\ba_{2:K}=\vv_0)=
\begin{cases}
\nu_{[0,\vv_0]}, & (r_1,r_2)=(0,0);\\
\theta^-_1\cdot\nu_{[0,\vv_0]}, & (r_1,r_2)=(1,0);\\
\theta^-_2\cdot\nu_{[0,\vv_0]}, & (r_1,r_2)=(0,1);\\
\theta^-_1 \theta^-_2\cdot\nu_{[0,\vv_0]}, &(r_1,r_2)=(1,1).\\
\end{cases}
\]
With $f_{(r_1,r_2),\,\vv_0}$ taking the above form, \eqref{eq-f} implies $\theta^-_1=\bar \theta^-_1$ and $\theta^-_2=\bar \theta^-_2$. Then consider $\ba_{2:K}=\vv'_0$, then $\vv'_0\succeq \vv_1$ and $\vv'_0\succeq \vv_2$ imply that
\begin{align*}
&\mathbb	P(\RR_{1:2}\succeq \rr_{1:2},\,\ba_{2:K}=\vv'_0)\\
&=
\begin{cases}
\nu_{[0,\vv_0]} + \nu_{[1,\vv_0]}, & (r_1,r_2)=(0,0);\\
\theta^-_1\cdot\nu_{[0,\vv_0]} + \theta^+_1\cdot\nu_{[1,\vv_0]}, & (r_1,r_2)=(1,0);\\
\theta^-_2\cdot\nu_{[0,\vv_0]} + \theta^+_2\cdot\nu_{[1,\vv_0]}, & (r_1,r_2)=(0,1);\\
\theta^-_1 \theta^-_2\cdot\nu_{[0,\vv_0]} + \theta^+_1 \theta^+_2\cdot \nu_{[1,\vv_0]}, &(r_1,r_2)=(1,1).\\
\end{cases}
\end{align*}
With the above form of $\mathbb P(\RR_{1:2}\succeq \rr_{1:2},\,\ba_{2:K}=\vv'_0)$, \eqref{eq-f} gives that
\[\begin{cases}
\nu_{[0,\vv_0]} + \nu_{[1,\vv_0]} = \bar\nu_{[0,\vv_0]} + \bar\nu_{[1,\vv_0]};\\
\theta^-_1\cdot\nu_{[0,\vv_0]} + \theta^+_1\cdot\nu_{[1,\vv_0]} = \theta^-_1\cdot\bar\nu_{[0,\vv_0]} + \bar \theta^+_1\cdot\bar\nu_{[1,\vv_0]};\\
\theta^-_2\cdot\nu_{[0,\vv_0]} + \theta^+_2\cdot\nu_{[1,\vv_0]} = \theta^-_2\cdot\bar\nu_{[0,\vv_0]} + \bar \theta^+_2\cdot\bar\nu_{[1,\vv_0]};\\
\theta^-_1\theta^-_2\cdot\nu_{[0,\vv_0]} + \theta^+_1\theta^+_2\cdot\nu_{[1,\vv_0]} = \theta^-_1\theta^-_2\cdot\nu_{[0,\vv_0]} + \bar \theta^+_1\bar \theta^+_2\cdot\nu_{[1,\vv_0]}.\\
\end{cases}\]
where $\theta^-_1=\bar \theta^-_1$ and $\theta^-_2=\bar \theta^-_2$ are used. Solving the above equations gives $\theta^+_1=\bar \theta^+_1$ and $\theta^+_2=\bar \theta^+_2$.
\end{enumerate}
Based on the above discussion, if $Q'$ contains either of the type-(a$^*$) or type-(b$^*$) combinations of row vectors $\vv_0$ and $\vv'_0$, then we have $\theta^-_1 = \bar \theta^-_1$, $\theta^-_2 = \bar \theta^-_2$, $\theta^+_1 = \bar \theta^+_1$ and $\theta^+_2 = \bar \theta^+_2$,  and hence by Proposition \ref{prop1}, the grouped proportion parameters $\nnu$ are identifiable. 

Note that the arguments in (a$^*$) and (b$^*$) above do not depend on the assumption that $\vv_0$ or $\vv'_0$ are single row vectors of $Q'$. Actually, if there exist two disjoint sets of items $S_1, S_2\subseteq \{3,\ldots,J\}$ such that 
\[
\vv_0=\vee_{h\in S_1\,} \bq'_{h},\quad \vv'_0=\vee_{h\in S_2\,} \bq'_{h},
\]
and the pair $(\vv_0,\vv'_0)$ satisfy either the type-(a$^*$) or the type-(b$^*$) constraint (namely {Either} $\vv_0\succeq \vv_1$, $\vv_0\nsucceq \vv_2$ and $\vv'_0\nsucceq \vv_1$, $\vv'_0\succeq \vv_2$;  {Or} $\vv_0\nsucceq \vv_1$, $\vv_0\nsucceq \vv_2$ and $\vv'_0\succeq \vv_1$, $\vv'_0\succeq \vv_2$), 
then the arguments in (a*), (b*) still hold, and the conclusion of partial identifiability follows. Next we show such pair $(\vv_0,\vv'_0)$ must exist.
The item set  $\{3,\ldots,J\}$ can be decomposed as $\{3,\ldots,J\}:=S_{00}\cup S_{10}\cup S_{02}\cup S_{12}$ where
\[\begin{aligned}
S_{00}&=\{3\leq j\leq J: \bq_{j}'\nsucceq\vv_1,~ \bq'_{j}\nsucceq\vv_2 \},\\
S_{10}&=\{3\leq j\leq J: \bq_{j}'\succeq\vv_1,~ \bq'_{j}\nsucceq\vv_2 \},\\
S_{02}&=\{3\leq j\leq J: \bq_{j}'\nsucceq\vv_1,~ \bq'_{j}\succeq\vv_2 \},\\
S_{12}&=
\{3,\ldots, J\}\setminus(S_{00}\cup S_{10}\cup S_{02}).
\end{aligned}\]
The assumption that $Q'$ satisfies condition (C$1^*$), 
implies that there exists $\vv'_0\in\mathcal R^{Q'}$ such that $\vv'_0\succeq \vv_1$, $\vv'_0\succeq \vv_2$. So if for $i=1,2$, (a) is satisfied, then the type-(b$^*$) combinations of row vectors exist in $Q'$. While if (a) is not satisfied and (b) is satisfied, then we claim that $S_{10}\neq\varnothing$ and $S_{02}\neq\varnothing$. This is because if $S_{10}=\varnothing$, then together with the fact that $S_{00}=\varnothing$ implied by the failure of (a), we will have $\{3,\ldots,J\}=S_{02}\cup S_{12}$. But this means for any item $j\geq 3$, $\bq_{j}'\succeq\vv_2$, contradictory to the assumption of case (b). So $S_{10}\neq\varnothing$ must hold, and similarly $S_{02}\neq\varnothing$ must hold. This ensures the type-(b$^*$) combinations of row vectors exist in $Q'$. In either scenarios, $Q'$ contains at least one of type-(a$^*$) or type-(b$^*$) combinations of row vectors, so we obtain the identifiability of all the item parameters. Applying Proposition \ref{prop1} gives the identifiability of the grouped proportion parameters $\nnu$, which completes the proof of part (B.2).

\vspace{3mm}

\noindent
\textbf{Proof of Part (b.1)}. Under the assumptions in part (B.1), the $Q$-matrix takes the following form
\begin{equation}\label{eq-q0v}
Q=
\left(\begin{array}{cc}
1 & \mz^\top \\
1 & \vv^\top \\
\mathbf 0 & Q'
\end{array}\right).
\end{equation}
Since there exists a single-attribute item with $\bq$-vector being $(1,\mz^\top)$,  for any $\aaa_{2:K}\in\mathcal R^{Q'}$ we have $[0,\aaa_{2:K}]\neq[1,\aaa_{2:K}]$, where the equivalence class notation $[\Cdot]$ represents that induced by the $J\times K$ $Q$-matrix $Q$.
Then following the similar arguments as in the proof of part (B.2), Equation \eqref{eq-tpr} hold as long as the following set of equations hold
\begin{align}\label{eq-ctex}
\begin{cases}
\nu_{[0,\aaa_{2:K}]} + \nu_{[1,\aaa_{2:K}]} = \bar\nu_{[0,\aaa_{2:K}]} + \bar\nu_{[1,\aaa_{2:K}]},& \forall \aaa_{2:K}\in\mathcal R^{Q'};\\
\theta^-_1\cdot\nu_{[0,\aaa_{2:K}]} + \theta^+_1\cdot\nu_{[1,\aaa_{2:K}]} \\
\qquad\qquad~  = \theta^-_1\cdot\bar\nu_{[0,\aaa_{2:K}]} + \bar \theta^+_1\cdot\bar\nu_{[1,\aaa_{2:K}]},& \forall \aaa_{2:K}\in\mathcal R^{Q'};\\
\theta^-_2\cdot\nu_{[0,\aaa_{2:K}]} + \theta^+_2\cdot\nu_{[1,\aaa_{2:K}]}  & \\
\qquad\qquad~  = \theta^-_2\cdot\bar\nu_{[0,\aaa_{2:K}]} + \bar \theta^+_2\cdot\bar\nu_{[1,\aaa_{2:K}]},& \forall \aaa_{2:K}\succeq \vv,~ \aaa_{2:K}\in\mathcal R^{Q'} ;\\
\theta^-_1\theta^-_2\cdot\nu_{[0,\aaa_{2:K}]} + \theta^+_1\theta^+_2\cdot\nu_{[1,\aaa_{2:K}]}  & \\
\qquad\qquad~  = \theta^-_1\theta^-_2\cdot\nu_{[0,\aaa_{2:K}]} + \bar \theta^+_1\bar \theta^+_2\cdot\nu_{[1,\aaa_{2:K}]},& \forall \aaa_{2:K}\succeq \vv,~ \aaa_{2:K}\in\mathcal R^{Q'} .\\
\end{cases}
\end{align}
Now consider a set of parameters $(\ttt^+,\ttt^-,\nnu)$ such that $\nu_{[0,\aaa_{2:K}]} = \rho\cdot\nu_{[1,\aaa_{2:K}]}$ for any $\aaa_{2:K}\in\mathcal R^{Q'}$, where $\rho$ is a positive constant.
Setting $\theta^+_1 = \bar \theta^+_1$, $\theta^-_2 = \bar \theta^-_2$, $\theta^+_j=\bar \theta^+_j$ and $\theta^-_j = \bar \theta^-_j$ for $j=3,\ldots,J$ and freely choosing any valid $\bar \theta^-_1$ which is not equal to $\theta^-_1$, we construct the remaining parameters $(\bar \theta^+_2, \bar\nnu)$ as follows. Let
\[
\bar \theta^+_2 = \frac{(\theta^+_1-\bar \theta^-_1)(\theta^+_2- \theta^-_2)}{(\theta^+_1-\bar \theta^-_1)+\rho(\theta^-_1-\bar \theta^-_1)} +\bar \theta^-_2, \]
and for any $\aaa_{2:K}\in\mathcal R^{Q'}$ let
\[
\begin{aligned}
\bar\nu_{[1,\aaa_{2:K}]} &= \frac{(\theta^+_1-\bar \theta^-_1) +\rho(\theta^-_1-\bar \theta^-_1)}{\theta^+_1-\bar \theta^-_1} \nu_{[1,\aaa_{2:K}]},\\
\bar\nu_{[0,\aaa_{2:K}]} &= \nu_{[0,\aaa_{2:K}]} + \nu_{[1,\aaa_{2:K}]} -\bar\nu_{[1,\aaa_{2:K}]},
\end{aligned}
\]
then by direct calculations one can check \eqref{eq-ctex} hold.
Therefore, we have found another set of parameters $(\bar\ttt^+,\bar\ttt^-,\bar\nnu)$ such that $(\bar\ttt^+,\bar\ttt^-,\bar\nnu) \neq (\ttt^+,\ttt^-,\nnu)$ and $T(\ttt^+,\ttt^-)\nnu = T(\bar\ttt^+,\bar\ttt^-)\bar\nnu$, which shows the non-identifiability of the model parameters  under the $Q$ in the form of \eqref{eq-q0v}. This completes the proof of part (B.1).
\end{proof}

\smallskip
\begin{proof}[Proof of Theorem \ref{thm-com-necc}]
Without loss of generality, we again focus on the proof of the conclusion for the two-parameter conjunctive models since all the arguments also hold for the compensatory models, following the similar argument in the proof of Proposition \ref{prop-after1}. 
Suppose condition (C$1^*$) holds.
Without loss of generality, suppose condition (C2$^{**}$) does not hold for some basis item $j$, and suppose that the first $K_1$ entries of the row vector $\bq_j$ in the $Q$-matrix corresponding to this basis item are 1's and the remaining $K-K_1$ entries of $\bq$ are 0's, i.e. 
\[
\bq_j = (\underbrace{1,\,\ldots,\, 1,}_\text{columns $1,\ldots,K_1$} 0,\,\ldots, \,0).
\]
Denote $S_{-j} = \{1,\ldots,J\}\setminus \{j\}$. Since $j$ is a basis item, any item in $S_{-j}$ requires some attribute not required by $j$, i.e.
\[
\forall h\in S_{-j},\quad q_{h,k} = 1\text{ for some } k\in\{K_1+1,\ldots,K\}.
\]
We claim, the assumption that (C2$^{**}$) does not hold for item $h$, implies that row vectors of items in $S_{-j}$ can be arranged in a way $\{\uu_1,\ldots,\uu_{J-1}\}$ such that for any $2\leq i\leq J-1$, $\uu_i$ requires at least one more attribute in $\{K_1+1,\ldots,K\}$ that is not required by $\cup_{1\leq s\leq i-1}\{\uu_s\}$. This claim is true since otherwise for some $h\in S_{-j}$ and $S_0\subseteq S_{-j}\backslash\{h\}$, the difference of attributes required by $\{h\}$ and $S_{0}$ are only among $\{1,\ldots,K_1\}$, then taking $S^-_j = S_{0}$ and $S^+_j = S_{0}\cup \{h\}$ makes (C2$^{**}$) hold for item $j$. 
In other words, for some $1\leq k_1<k_2<\ldots<k_{J-1}\leq K-K_1$ we have that
\[
\ww_1 = \vv_{k_1},\quad \ww_2 = \vv_{k_2},\quad \ldots,\quad \ww_{J-1} = \vv_{k_{J-1}},
\]
where $\vv_1,\vv_2,\ldots,\vv_{m}$ takes the form as follows
\begin{equation}\label{eq-necc-form}
\begin{aligned}
\begin{array}{cccc|ccccc}
\bq_j: & 1 & \cdots & 1 & 0 & 0 & \cdots & \cdots & 0\\
\hline
\vv_1: & * & \cdots & * & 1 & 0 & \cdots & \cdots & 0  \\
\vv_2: & * & \cdots & * & * & 1 & \cdots & \cdots & 0 \\
\vdots & \vdots & \vdots & \vdots & \vdots & \vdots & \vdots & \vdots & \vdots \\
\vv_{K-K_1-1}: & * & \cdots & * & * & * & * & 1 & 0 \\
\vv_{K-K_1}: & * & \cdots & * & * & * & * & * & 1  \\
\end{array},
\end{aligned}
\end{equation}
Now we are ready to construct two different sets of parameters $(\ttt^+,\ttt^-,\nnu)$ $\neq$ $(\bar\ttt^+,\bar\ttt^-,\bar\nnu)$ that give (\ref{eq1}), i.e.
\[T(\ttt^+,\ttt^-)\nnu = T(\bar\ttt^+,\bar\ttt^-)\bar\nnu.\]
Given $(\ttt^+,\ttt^-,\nnu)$, condition (C$1^*$) guarantees $\ttt^+=\bar\ttt^+$ and $\theta^-_j=\bar \theta^-_j$ for $j\in S_{non}$. 
Equation (\ref{eq1}) holds if for another set of parameters $(\bar\ttt^+,\bar\ttt^-,\bar\nnu)$, the following equations hold for any $\ww_i$ such that $\ww_i\stackrel{\Gamma}{\nsim}\ww_i\vee\bq_j$
\begin{equation}\label{eq-count}
\begin{cases}
\nu_{[\ww_i]} + \nu_{[\bq_j\vee\ww_i]} = \bar \nu_{[\ww_i]} + \bar \nu_{[\bq_j\vee\ww_i]}; \\
\theta^-_j\cdot \nu_{[\ww_i]} + \theta^+_j\cdot \nu_{[\bq_j\vee\ww_i]} = \bar \theta^-_j\cdot \nu_{[\ww_i]} + \theta^+_j\cdot \nu_{[\bq_j\vee\ww_i]}, \\
\end{cases}
\end{equation}
with any other parameter not specified in \eqref{eq-count} equal to its counterpart in the original set of parameters $(\ttt^+,\ttt^-,\nnu)$.
Denote the cardinality of the set 
$\mathcal W = \{\ww_i:\ww_i\stackrel{\Gamma}{\nsim}\ww_i\vee\bq_j\}$ by $|\mathcal W|$. The set $\mathcal W$ is nonempty since $\mz\stackrel{\Gamma}{\nsim}\bq_j$ and $\mz\in\mathcal W$, where $\Gamma$ is the $\Gamma$-matrix corresponding to the saturated latent class space $\ma=\{0,1\}^K$.
Note that (\ref{eq-count}) involve $2|\mathcal W|+1$ free parameters $\{\bar \theta^-_j\}\cup\{\bar \nu_{[\ww_i]},~ \bar \nu_{[\ww_i\vee\bq]}:\ww_i\in\mathcal W\}$ while only contain $2|\mathcal W|$ equations, so there are infinitely many solutions to (\ref{eq-count}). This proves the non-identifiability of the model parameters.
\end{proof}

\smallskip
\begin{proof}[Proof of Theorem \ref{cor-suffnece}]
First prove the claim that conditions (C$1^*$) and (C$2^*$) are equivalent to conditions (C1$'$) and (C2$'$) under the assumption that the $Q$-matrix is complete and $p_{\aaa}>0$ for any $\aaa\in\{0,1\}^K$. Theorem 1 in \cite{id-dina} established that if $Q$ is complete and $p_{\aaa}>0$ for any $\aaa\in\{0,1\}^K$, then conditions (C1$'$) and (C2$'$) combined is sufficient and necessary for the identifiability of the DINA model parameters $(\ttt^+,\ttt^-,\pp)$. 
Since (C$1^*$) and (C$2^*$) are sufficient conditions for identifiability, they must imply the necessary conditions (C1$'$) and (C2$'$). In the following we prove the other direction, i.e., conditions (C1$'$) and (C2$'$) imply conditions (C$1^*$) and (C$2^*$).

When $Q$ is complete, if condition (C1$'$) holds that attribute $k$ is required by at least three items in the $Q$-matrix, then for each unit vector $\ee_k$ as the $\bq$-vector, there must exist two other items $j_k^1$ and $j_k^2$ that also measure attribute $k$.
Let $S_k^i=\{j_k^i\}$, $i=1,2$, then $S_k^1$ and $S_k^2$ are the two disjoint item sets that satisfy condition (C$1^*$) that $\ee_k = \bq_k \preceq \vee_{h\in S_k^i} \bq_h = \bq_{j_k^i}$ for $i = 1$ and $2$. This shows (C1$'$) implies (C$1^*$).

Assume without loss of generality that $Q$ takes the form
\begin{equation}\label{eq-complete-q}
Q = \left(\begin{array}{c}
\mathcal I_K \\
Q'
\end{array}\right).
\end{equation}
If condition (C2$'$) is satisfied, we next explicitly construct a  procedure that sequentially expands the separator set $S_{sep}$ until $S_{sep}=\ms$ finally, which by Theorem \ref{thm1} would establish identifiability of all the model parameters. The existence of such sequential procedure would ensure the Sequentially Differentiable Condition (C$2^*$) holds. Theorem \ref{thm1} has already established that condition (C$1^*$) suffices for the identifiability of all the slipping parameters, and that of the guessing parameters of the non-basis items. Specifically for the complete $Q$-matrix in the form of \eqref{eq-complete-q}, this conclusion implies $\theta^+_j = \bar \theta^+_j$ for all $j=1,\ldots,J$ and $\theta^-_j = \bar \theta^-_j$ for all $j=K+1,\ldots,J$, because any item $j>K$ must be a non-basis item in the sense that there always exists some item $k\in\{1,\ldots,K\}$ such that $\bq_k=\ee_k\preceq\bq_j$. It remains to show the guessing parameters of the first $K$ items are identifiable, i.e., $\theta^-_k = \bar \theta^-_k$ for $k=1,\ldots,K$.  
For any binary vectors $\boa=(a_1,\ldots,a_L)$, $\bob=(b_1,\ldots,b_L)$ of the same length, we say $\boa$ is lexicographically smaller than $\bob$, denoted by $\boa\prec_{\text{lex}}\bob$, if either $a_1<b_1$; or there exists some $2\leq i\leq l$ such that $a_i<b_i$ and $a_j=b_j$ for all $j<i$.
Now that the $K$ column vectors of $Q'$ are mutually distinct,
there is a unique permutation $(m_1,m_2,\ldots,m_K)$ of $(1,2,\ldots,K)$ such that 
$Q'_{\Cdot,m_1}\prec_{\text{lex}}Q'_{\Cdot,m_2}\prec_{\text{lex}}\ldots\prec_{\text{lex}}Q'_{\Cdot,m_K}$. For any $1\leq i<j\leq K$, since $Q'_{\Cdot,m_i}\prec_{\text{\text{lex}}}Q'_{\Cdot,m_j}$, we must have $Q'_{\Cdot,m_i}\nsucceq Q'_{\Cdot,m_j}$. This fact will be useful in the following proof.

We start with the initial separator set $S_{sep}:=S_{0} = \{K+1,\ldots,J\}$. Note that at this starting stage $S_{sep}\subseteq S_{non}$. 
We next argue that item $m_1$ is $S_{0}$-differentiable, and further, $m_{i}$ is $\left(S_0\cup \{m_1,\ldots,m_{i-1}\}\right)$-differentiable for all $i=2,\ldots,K$. Noting that $Q_{\Cdot,m_1}$ is of the smallest lexicographic order among all the column vectors of the submatrix $Q'$, define 
\[\begin{aligned}
S_{m_1}^- &= \{j\in S_0: q_{j, m_1}=0\},
\end{aligned}\]
then $\vee_{h\in S_0}\bq_{h}$ equals the all-one vector under condition (C1$'$) while $\vee_{h\in S_{m_1}^-}\bq_{h}$ equals the vector that is zero in the $m_1$th entry and one otherwise, i.e., 
\[\begin{aligned}
\vee_{h\in S_0}\bq_{h} &= (1,\ldots,1),\\
\vee_{h\in S_{m_1}^-}\bq_{h}&=(1,\ldots,1,\underbrace{0}_\text{column $m_1$},1,\ldots,1),
\end{aligned}\] 
so $\vee_{h\in S_0}\bq_{h} - \vee_{h\in S_{m_1}^-}\bq_{h} = \ee_{m_1}^\top = \bq_{m_1}$. By definition of $S$-differentiable, this means item $m_1$ is $S_0$-differentiable. Then expand the separator set by including item $m_1$ in it, i.e. let $S_{sep}:=S_0\cup \{m_1\}$. Now further define $S_{m_2}^- = \{j\in S_0: Q_{j, m_2}=0\}\cup\{m_1\}$, then $S_{m_2}^-\subseteq S_{sep}$. Similarly it is easy to check 
\[\begin{aligned}
\vee_{h\in S_{sep}}\bq_{h} &= (1,\ldots,1),\\
\vee_{h\in S_{m_2}^-}\bq_{h}&=(1,\ldots,1,\underbrace{0}_\text{column $m_2$},1,\ldots,1),
\end{aligned}\]
and this implies item $m_2$ is $S_{sep}$-differentiable. The similar  argument would give that $m_{i}$ is $\left(S_0\cup \{m_1,\ldots,m_{i-1}\}\right)$-differentiable for all $i=2,\ldots,K$, so the sequential expanding procedure ends up with $S_{sep} = \{1,\ldots,J\}=\ms$. Note that we start with an initial separator set $S_0$ that is a subset of $S_{non}$ and in each expanding step we included exactly one more item into $S_{sep}$ even if we might have included more (all the items that are $S_{sep}$-differentiable could be included, which can be more than one), the fact that in our procedure $S_{sep}$  finally equals $\ms$ actually proves a stronger conclusion than the existence of a sequential procedure described in condition (C$2^*$), so the Sequentially Differentiable Condition (C$2^*$) holds.
By now we have shown conditions (C1$'$) and (C2$'$) also imply conditions (C$1^*$) and (C$2^*$). 

Since (C1$'$) and (C2$'$) combined is necessary, (C$1^*$) and (C$2^*$) combined is also necessary.
This completes the proof of the theorem that (C$1^*$) and (C2$^*$) are sufficient and necessary for strict identifiability of the two-parameter model when the $Q$-matrix is complete and $p_{\aaa}>0$ for all $\aaa\in\{0,1\}^K$.
\end{proof}

\medskip
\section*{Section C: Proof of Main Results in Section 4}\label{sec-multi-proof}

We introduce a useful lemma before proving Theorem  \ref{thm-order} and Theorem \ref{thm-gen-q}, the results of strict identifiability of multi-parameter restricted latent class models. 
The proof of the following lemma is given in Section D. For notational simplicity, we denote $\theta_{j, \mo}:=
\max_{\aaa: \Gamma_{j,\aaa} = 1} \theta_{j, \aaa} 
= \min_{\aaa: \Gamma_{j,\aaa} = 1} \theta_{j, \aaa}$ in the following discussion.

\begin{lemma}\label{lem-neq}
For an arbitrary restricted latent class model satisfying constraints \eqref{eq-constraints},
if Equation (\ref{eq1}) holds, then for any $j\in S_1\cup S_2$ and any $\aaa$ such that $\Gamma_{j,\aaa}=0$,
\[\theta_{\ee_j,\aaa} \neq \bar\theta_{\ee_j,\mo},\quad
\theta_{\ee_j,\mo} \neq \bar\theta_{\ee_j,\aaa}.\]
\end{lemma}
To prove Theorem  \ref{thm-order}, we also need the following lemma, whose proof is given in Section D.
\begin{lemma}\label{lem-order-a} 
Under the assumptions of Theorem \ref{thm-order}, for any $\aaa$ there exists vectors $\uu_{\aaa}$ and $\vv_{\aaa}$ such that 
\begin{equation}\label{eq-order}
\begin{aligned}
\{\vv_{\aaa}^\top \cdot T(\TT_{S_2})\}_{\aaa} \neq 0;\qquad 
\{\vv_{\aaa}^\top \cdot T(\TT_{S_2})\}_{\aaa'} &= 0,\quad \forall\aaa'\precneqq_{S_1}\aaa.\\
\{\uu_{\aaa}^\top \cdot T(\bar\TT_{S_1})\}_{\aaa} \neq 0;\qquad 
\{\uu_{\aaa}^\top \cdot T(\bar\TT_{S_1})\}_{\aaa'} &= 0,\quad \forall\aaa'\precneqq_{S_2}\aaa.\\ 
\end{aligned}
\end{equation}
\end{lemma}

\begin{proof}[Proof of Theorem \ref{thm-order}]
Equipped with  Lemmas \ref{lem-neq} and \ref{lem-order-a},   we  prove Theorem \ref{thm-order} in the following  three steps. Without loss of generality, assume $S_1=\{1,\ldots,M_1\}$ and $S_2=\{M_1+1,\ldots,M_1+M_2\}$, namely item set $S_1$ contains the first $M_1$ items and item set $S_2$ contains the next $M_2$ items.
\vspace{4mm}
\newline
\textit{Step 1:}
$\theta_{\ee_j,\aaa_0} = \bar \theta_{\ee_j,\aaa_0}$ for $j>M_1+M_2$.
\vspace{2mm}
\newline
\textit{Step 2:}
$\theta_{\ee_j,\aaa} = \bar \theta_{\ee_j,\aaa}$ for $j>M_1+M_2$ and any $\aaa$.
\vspace{2mm}
\newline
\textit{Step 3:}
$\theta_{\ee_j,\aaa} = \bar \theta_{\ee_j,\aaa}$ and $p_{\aaa}=\bar p_{\aaa}$ for $1\leq j\leq M_1+M_2$ and any $\aaa$.
\vspace{4mm}

\noindent
Now we start the proof of the result step by step. 
\vspace{4mm}

\noindent
\textbf{Step 1}.
Define $\ttt^*\in\mathbb R^J$ to be 
\[
\ttt^* = (\bar\theta_{\ee_1,\mo},\ldots,\bar\theta_{\ee_{M_1},\mo},
          \theta_{\ee_{M_1+1},\mo},\ldots,\theta_{\ee_{M_1+M_2},\mo},
          \mz_{J-M_1-M_2})^\top,
\]
and consider the row vector of the transformed $T$-matrix $T(\TT-\ttt^*\mathbf1^\top)$ corresponding to $\rr=\sum_{k=1}^{M_1+M_2}\ee_k$ is
\[\begin{aligned}
&T_{\sum_{k=1}^{M_1+M_2}\ee_k,\cdot}(\TT-\ttt^*\mathbf1^\top) 
= \bigodot_{k=1}^{M_1+M_2} T_{\ee_k,\cdot}(\TT-\ttt^*\mathbf1^\top)  \\
= &\left(\prod_{k=1}^{M_1} (\theta_{\ee_k,\aaa_0}-\bar\theta_{\ee_k,\mo}) 
\prod_{k=1}^{M_2}(\theta_{\ee_{M_1+k},\aaa_0}-\theta_{\ee_{M_1+k},\mo}), 
\mathbf0^\top_{M}\right),
\end{aligned}\]
where the last $M$ elements of this row vector are all zero.
By Lemma \ref{lem-neq}, the first element is nonzero, i.e.,
\[
\prod_{k=1}^{M_1} (\theta_{\ee_k,\aaa_0}-\bar\theta_{\ee_k,\mo}) 
\prod_{k=1}^{M_2} (\theta_{\ee_{M_1+k},\aaa_0}-\theta_{\ee_{M_1+k},\mo})
\neq 0.
\]
Then similarly for parameters $(\bar\TT,\bar\pp)$ we have
\[\begin{aligned}
&T_{\sum_{k=1}^{M_1+M_2}\ee_k}(\bar\TT-\ttt^*\mathbf1^\top) \\
& = \biggr(\prod_{k=1}^{M_1} (\bar\theta_{\ee_k,\aaa_0}-\bar\theta_{\ee_k,\mo}) 
\prod_{k=1}^{M_2}(\bar\theta_{\ee_{M_1+k},\aaa_0}-\theta_{\ee_{M_1+k},\mo}), 
\mathbf0^\top_{M}\biggr) 
\end{aligned}\]
and 
\[
\prod_{k=1}^{M_1} (\bar\theta_{\ee_k,\aaa_0}-\bar\theta_{\ee_k,\mo}) 
\prod_{k=1}^{M_2}(\bar\theta_{\ee_{M_1+k},\aaa_0}-\theta_{\ee_{M_1+k},\mo})
\neq 0.
\]

Now consider $\theta_{\ee_j,\aaa_0}$ for any $j>M_1+M_2$. The row vectors of $T(\TT-\ttt^*\mathbf1^\top)$ and $T(\bar\TT-\ttt^*\mathbf1^\top)$ corresponding to the response pattern $\rr=\sum_{k=1}^{M_1+M_2}\ee_k + \ee_j$ are
\[\begin{aligned}
&T_{\sum_{k=1}^{2M}\ee_k+\ee_j,\cdot}(\TT-\ttt^*\mathbf1^\top) \\
&= \left(\theta_{\ee_j,\aaa_0} \prod_{k=1}^{M_1} (\theta_{\ee_k,\aaa_0}-\bar\theta_{\ee_k,\mo}) 
\prod_{k=1}^{M_2}(\theta_{\ee_{M_1+k},\aaa_0}-\theta_{\ee_{M_1+k},\mo}),
\mathbf0^\top_{M}\right),
\end{aligned}\]
and
\[\begin{aligned}
&T_{\sum_{k=1}^{2M}\ee_k+\ee_j,\cdot}(\bar\TT-\ttt^*\mathbf1^\top) \\
&= \left(\bar\theta_{\ee_j,\aaa_0} \prod_{k=1}^{M_1} (\bar\theta_{\ee_k,\aaa_0}-\bar\theta_{\ee_k,\mo}) 
\prod_{k=1}^{M_2}(\bar\theta_{\ee_{M_1+k},\aaa_0}-\theta_{\ee_{M_1+k},\mo}),
\mathbf0^\top_{M}\right),
\end{aligned}\]
respectively. Note Equation (\ref{eq1}) implies that
\begin{eqnarray*}
	\theta_{\ee_j,\aaa_0} &=&
\frac{T_{\sum_{k=1}^{M_1+M_2}\ee_k+\ee_j,\cdot}(\TT-\ttt^*\mathbf1^\top)\pp}{T_{\sum_{k=1}^{M_1+M_2}\ee_k,\cdot}(\TT-\ttt^*\mathbf1^\top)\pp}\\
&=& \frac{T_{\sum_{k=1}^{M_1+M_2}\ee_k+\ee_j,\cdot}(\bar\TT-\ttt^*\mathbf1^\top)\bar\pp}{T_{\sum_{k=1}^{M_1+M_2}\ee_k,\cdot}(\bar\TT-\ttt^*\mathbf1^\top)\bar\pp}
~= ~\bar\theta_{\ee_j,\aaa_0}.
\end{eqnarray*}
 \vspace{2mm}

\noindent
\textbf{Step 2}.
First consider any $j\in (S_1\cup S_2)^c$. For any $\aaa$, define
\[\ttt_{\aaa} = 
\sum_{h\in S_1: \Gamma_{h,\aaa}=0} \theta_{\ee_h,\mo}\ee_h +
\sum_{h\in S_2: \Gamma_{h,\aaa}=0} \bar\theta_{\ee_h,\mo}\ee_h,\]
and consider the row vector corresponding to response pattern $\rr = \sum_{h\in S_1} \ee_h$ in the transformed $T$-matrix, then we have
\[
\begin{aligned}
T_{\sum_{h\in S_1}\ee_h,\aaa'}(\TT-\ttt_{\aaa}\mo^\top) &\neq 0 \quad\text{iff}\quad \aaa'\preceq_{S_1}\aaa, \\
T_{\sum_{h\in S_2}\ee_h,\aaa'}(\bar\TT-\ttt_{\aaa}\mo^\top) & \neq 0 \quad\text{iff}\quad \aaa'\preceq_{S_2}\aaa.
\end{aligned}
\]
We only prove the first inequality above and the second is just similar. Note
\begin{equation}\label{eq-t7-3}
T_{\sum_{h\in S_1}\ee_h,\aaa'}(\TT-\ttt_{\aaa}\mo^\top) = 
\prod_{h\in S_1: \Gamma_{h,\aaa}=0} (\theta_{\ee_h,\aaa'} - \theta_{\ee_h,\mo}),
\end{equation}
and if $\aaa'\npreceq\aaa$, then there exists some $h$ such that $\Gamma_{h,\aaa'}=1$, $\Gamma_{h,\aaa}=0$ and hence $\theta_{\ee_h,\aaa'}-\theta_{\ee_h,\mo}=0$, which makes the product in (\ref{eq-t7-3}) equal to 0; while if $\aaa'\preceq\aaa$, then for all $h\in S_1$ such that $\Gamma_{h,\aaa}=0$, we have $\Gamma_{h,\aaa'}\leq \Gamma_{h,\aaa}=0$ and hence $\theta_{\ee_h,\aaa'}-\theta_{\ee_h,\mo}\neq 0$, so the product in (\ref{eq-t7-3}) is nonzero.

Then we use the properties of $\uu_{\aaa}$ and $\vv_{\aaa}$ to continue with the proof. First note that the existence of $\uu_{\aaa}$ and $\vv_{\aaa}$ satisfying (\ref{eq-order}) only rely on the full-column-rank property of $T(\TT_{S_i})$ and $T(\bar\TT_{S_i})$, so for some full-rank linear transformation matrix $A$ there still exists some $\uu_{\aaa}$ and $\vv_{\aaa}$ such that
\[\begin{aligned}
\vv_{\aaa}^\top \cdot A\cdot T(\bar\TT_{S_2}) &= (\mz,\underbrace{1}_\text{column $\aaa$},\mz),\\
\uu_{\aaa}^\top \cdot A\cdot T(\TT_{S_1}) &= (\mz,\underbrace{1}_\text{column $\aaa$},\mz),
\end{aligned}\]
and
\begin{eqnarray}\label{eq-order2}
&~~~~ \{\vv_{\aaa}^\top \cdot A\cdot T(\TT_{S_2})\}_{\aaa} \neq 0;\quad 
\{\vv_{\aaa}^\top \cdot A\cdot T(\TT_{S_2})\}_{\aaa'}  = 0, ~ \forall\aaa'\precneqq_{S_1}\aaa; \\
&~~~~ \{\uu_{\aaa}^\top \cdot A\cdot T(\bar\TT_{S_1})\}_{\aaa} \neq 0;\quad 
\{\uu_{\aaa}^\top \cdot A\cdot T(\bar\TT_{S_1})\}_{\aaa'} = 0,~ \forall\aaa'\precneqq_{S_2}\aaa.\notag 
 \end{eqnarray}
Now note that $T_{\sum_{h\in S_2}\ee_h,\cdot}(\TT-\ttt_{\aaa}\mo^\top)$ can just be expressed as $D(\ttt_{\aaa})\cdot T(\TT_{S_2})$ indicated by Proposition \ref{prop-ltrans}, so we have
\begin{eqnarray}\label{eq-t7-4}
 &&\{\uu_{\aaa}^\top \cdot T_{\sum_{h\in S_1}\ee_h,\Cdot}(\TT-\ttt_{\aaa}\mo^\top)\}
\odot \{\vv_{\aaa}^\top \cdot T_{\sum_{h\in S_2}\ee_h,\Cdot}(\TT-\ttt_{\aaa}\mo^\top)\} \\
&&= (\mz,\underbrace{x_{\aaa}}_\text{column $\aaa$},\mz),\quad\mbox{ with } x_{\aaa}\neq 0,\notag
\end{eqnarray}
\begin{eqnarray}\label{eq-t7-5}
 &&\{\uu_{\aaa}^\top \cdot T_{\sum_{h\in S_1}\ee_h,\Cdot}(\bar\TT-\ttt_{\aaa}\mo^\top)\}
\odot \{\vv_{\aaa}^\top \cdot T_{\sum_{h\in S_2}\ee_h,\Cdot}(\bar\TT-\ttt_{\aaa}\mo^\top)\}\\
&&= (\mz,\underbrace{\bar y_{\aaa}}_\text{column $\aaa$},\mz),\quad\mbox{ with } \bar y_{\aaa}\neq 0.\notag
 \end{eqnarray}
 Note that the left hand sides of equations \eqref{eq-t7-4} and \eqref{eq-t7-5} are both row transformations of the $T$-matrix, namely there exists a matrix $M_1$ such that 
\[\eqref{eq-t7-4}=M_1\cdot T(\TT),\quad\eqref{eq-t7-5}=M_1\cdot T(\bar\TT),\]
so by Equation \eqref{eq1}, we have $\eqref{eq-t7-4}\cdot\pp= \eqref{eq-t7-5}\cdot\bar\pp \neq 0$. 
Now consider any item $j\in (S_1\cup S_2)^c$, since \eqref{eq-t7-4} and \eqref{eq-t7-5} involve rows of the $T$-matrices only with respect to items included in $S_1\cup S_2$, Equation \eqref{eq1} further implies 
$\{T_{\ee_j,\aaa}(\TT)\odot (\ref{eq-t7-4})\}\cdot \pp=\{T_{\ee_j,\aaa}(\bar\TT)\odot (\ref{eq-t7-5})\}\cdot \bar\pp$, therefore we have the equality
 \begin{align}\label{aaaa}
\frac{\{T_{\ee_j,\aaa}(\TT)\odot (\ref{eq-t7-4})\}\cdot \pp}{(\ref{eq-t7-4})\cdot \pp}
= \frac{\{T_{\ee_j,\aaa}(\bar\TT)\odot (\ref{eq-t7-5})\}\cdot \bar\pp}{(\ref{eq-t7-5})\cdot \bar\pp}.
\end{align}
Note that the left and right hand sides of the above equation can be written as
\[\begin{aligned}
\mbox{LHS of \eqref{aaaa} } &= \frac{\theta_{\ee_j,\aaa}\cdot (\ref{eq-t7-4})\cdot \bar\pp}{(\ref{eq-t7-4})\cdot \bar\pp} = \theta_{j,\aaa},\\
\mbox{RHS of \eqref{aaaa} } &= \frac{\bar\theta_{\ee_j,\aaa}\cdot (\ref{eq-t7-5})\cdot \bar\pp}{(\ref{eq-t7-5})\cdot \bar\pp} =\bar \theta_{j,\aaa},
\end{aligned}\]
so $\theta_{j,\aaa}=\bar\theta_{j,\aaa}$. 
\vspace{2mm}

\noindent
\textbf{Step 3}.
{\it First we prove $\theta_{\ee_j,\mo} = \bar\theta_{\ee_j,\mo}$ for any} $j\in S_1\cup S_2$. 
Given $\aaa$, define
\[
\ttt^* = \sum_{h\in S_1:\Gamma_{h,\aaa}=0}\theta_{\ee_h,\mo}\ee_h.
\]
Note that if for some $\aaa$, $\Gamma_{h,\aaa}=1$ for all $h\in S_1$, then $\ttt^*$ is defined to be the zero vector.
With $\ttt^*$, the row vector corresponding to $\rr^*=\sum_{h\in S_1:\Gamma_{h,\aaa}=0}\ee_h$ in the transformed $T$-matrix takes the following form
\[\begin{aligned}
& T_{\rr^*,\cdot}(\TT-\ttt^*\mo^\top) \\
& = 
\Big(\prod_{h\in S_1:\Gamma_{h,\aaa}=0}(\theta_{\ee_h,\aaa_0} - \theta_{\ee_h,\mo}),*,\ldots,*,
\prod_{h\in S_1:\Gamma_{h,\aaa}=0}(\theta_{\ee_h,\aaa} - \theta_{\ee_h,\mo}),0,\ldots,0\Big),
\end{aligned}\]
and satisfies that 
\[T_{\rr^*,\aaa}(\TT-\ttt^*\mo^\top)\neq 0;\quad
T_{\rr^*,\aaa'}(\TT-\ttt^*\mo^\top)= 0,\quad\forall \aaa'\npreceq_{S_1}\aaa.\]
From previous constructions we have
\[
\vv^\top_{\aaa}\cdot T(\bar\TT_{S_2}) = (\mz,\underbrace{1}_\text{column $\aaa$},\mz)^\top,
\]
and denote the value in column $\aaa$ of $\vv^\top_{\aaa}\cdot T(\TT_{S_2})$ by $b_{\vv,\aaa}$.
Consider any $j\in S_1\cup S_2$ such that $\Gamma_{j,\aaa}=1$, then obviously $\ee_j$ is not included in the sum in the previously defined response pattern $\rr^*$, because $\rr^*$ only contains those items that $\aaa$ is not capable of. So we have
\begin{equation}\label{eq-s3-1}
\begin{aligned}
& T_{\rr^*,\cdot}(\TT-\ttt^*\mo^\top)
\odot
\{\vv^\top_{\aaa}\cdot T(\TT_{S_2})\} \\
& =\Big( \mz^\top, 
\underbrace{b_{\vv,\aaa}\cdot\prod_{h\in S_1:\Gamma_{h,\aaa_k}=0}(\theta_{\ee_h,\aaa} - \theta_{\ee_h,\mo})}_\text{column $\aaa$},
 \mz^\top \Big),
\end{aligned}\end{equation}
\begin{equation}\label{eq-s3-2}
\begin{aligned}
& T_{\rr^*+\ee_j,\cdot}(\TT-\ttt^*\mo^\top)
\odot
\{\vv^\top_{\aaa}\cdot T(\TT_{S_2})\} \\
& =\Big( \mz^\top, 
\underbrace{\theta_{\ee_j,\mo}\cdot b_{\vv,\aaa}\cdot\prod_{h\in S_1:\Gamma_{h,\aaa}=0}(\theta_{\ee_h,\aaa} - \theta_{\ee_h,\mo})}_\text{column $\aaa$},
 \mz^\top \Big).
\end{aligned}\end{equation}
Similarly for $(\bar\TT,\bar\pp)$ we have
\begin{equation}\label{eq-s3-3}
\begin{aligned}
& T_{\rr^*,\cdot}(\bar\TT-\ttt^*\mo^\top)
\odot
\{\vv^\top_{\aaa}\cdot T(\bar\TT_{S_2})\} \\
& =\Big( \mz^\top, 
\underbrace{\prod_{h\in S_1:\Gamma_{h,\aaa}=0}(\bar\theta_{\ee_h,\aaa} - \theta_{\ee_h,\mo})}_\text{column $\aaa$},
 \mz^\top \Big),
\end{aligned}\end{equation}
\begin{equation}\label{eq-s3-4}
\begin{aligned}
& T_{\rr^*+\ee_j,\cdot}(\bar\TT-\ttt^*\mo^\top)
\odot
\{\vv^\top_{\aaa}\cdot T(\bar\TT_{S_2})\} \\
&=\Big( \mz^\top, \underbrace{\bar\theta_{\ee_j,\mo}\cdot\prod_{h\in S_1:\Gamma_{h,\aaa}=0}
(\bar\theta_{\ee_h,\aaa} - \theta_{\ee_h,\mo})}_\text{column $\aaa$}, \mz^\top \Big).
\end{aligned}\end{equation}
Equation (\ref{eq1}) implies $(\ref{eq-s3-1})\cdot\pp = (\ref{eq-s3-3})\cdot\bar\pp$, and since $(\ref{eq-s3-3})\cdot\bar\pp\neq 0$, we must also have $(\ref{eq-s3-1})\cdot\bar\pp\neq 0$, which indicates $b_{\vv,\aaa}\neq 0$.
The above four equations along with (\ref{eq1}) give that
\[
\theta_{\ee_j,\mo} = \theta_{\ee_j,\aaa} =
\frac{(\ref{eq-s3-2})\cdot\pp}{(\ref{eq-s3-1})\cdot\pp} = \frac{(\ref{eq-s3-4})\cdot\bar\pp}{(\ref{eq-s3-3})\cdot\bar\pp}
= \bar\theta_{\ee_j,\aaa} = \bar\theta_{\ee_j,\mo},\quad \forall j\in S_2.
\]
Note that the above equality $\theta_{\ee_{j},\mo} = \bar\theta_{\ee_{j},\mo}$ holds for any $\aaa$ and any item $j$ such that $\Gamma_{j,\aaa}=1$. 
Therefore we have shown $\theta_{\ee_{j},\mo} = \bar\theta_{\ee_{j},\mo}$ holds for any $j\in S_1\cup S_2$.
Similarly we also have $\theta_{\ee_{j},\aaa_0} = \bar\theta_{\ee_{j},\aaa_0}$. In summary,
\[\begin{aligned}
\theta_{\ee_{j},\aaa_0} = \bar\theta_{\ee_{j},\aaa_0},\quad
\theta_{\ee_{j},\mo} = \bar\theta_{\ee_{j},\mo},\quad 
\forall j\in S_1 \cup S_2. \\
\end{aligned}\]
For $\aaa=\aaa_0$ define
\[
\ttt^{*} = \sum_{h\in S_1} \theta_{\ee_h,\mo}\ee_h,
\]
then $T_{\sum_{h\in S_1}\ee_h}(\TT-\ttt^{*}\mo^\top)\pp
= T_{\sum_{h\in S_1}\ee_h}(\bar\TT-\ttt^{*}\mo^\top)\bar\pp$ gives
\[
\prod_{h\in S_1}(\theta_{\ee_h,\aaa_0} - \theta_{\ee_h,\mo})p_{\aaa_0} =
\prod_{h\in S_1}(\theta_{\ee_h,\aaa_0} - \theta_{\ee_h,\mo})\bar p_{\aaa_0},
\]
so $p_{\aaa_0}=\bar p_{\aaa_0}$.

\smallskip
 {\it Next we show $\theta_{\ee_j,\aaa}=\bar\theta_{\ee_j,\aaa}$ for any $\aaa$ and $j\in S_1\cup S_2$, where $\Gamma_{j,\aaa}=0$.}
We use the induction method to show that for any $\aaa\in\mc$,
\begin{equation}\label{eq-theta-p}
\forall j\in S_1\cup S_2,\quad
\theta_{j,\aaa} = \bar\theta_{j,\aaa},\quad p_{\aaa} = \bar p_{\aaa}.
\end{equation}
Firstly, we prove \eqref{eq-theta-p} hold for $\aaa=\aaa_1$, where $\aaa_1$ denotes the latent class with the smallest lexicographical order among $\mc\setminus\{\aaa_0\}$.
For $\aaa=\aaa_1$, define
\begin{equation}\label{eq-t-s3}
\ttt^* = \sum_{h\in S_1:\Gamma_{h,\aaa_1}=0}\theta_{\ee_h,\mo}\ee_h
+ \sum_{h\in S_1:\Gamma_{h,\aaa_1}=1}\theta_{\ee_h,\aaa_0}\ee_h,
\end{equation}
then the row vectors of $\rr^*=\sum_{h\in S_1}\ee_h$ in the transformed $T$-matrices only contain one nonzero element corresponding to column $\aaa_1$ as follows
\begin{align}\label{eq-new1}
&~ T_{\rr^*,\cdot}(\TT-\ttt^*\mo^\top)\\ \notag
=&~ \biggr(\mz^\top,
\prod_{h\in S_1:\Gamma_{h,\aaa_1}=0}(\theta_{\ee_h,\aaa_1} - \theta_{\ee_h,\mo})
\prod_{h\in S_1:\Gamma_{h,\aaa_1}=1}(\theta_{\ee_h,\aaa_1} - \theta_{\ee_h,\aaa_0}),
\mz^\top\biggr),
\end{align}
\begin{align}\label{eq-new2}
&~T_{\rr^*,\cdot}(\bar\TT-\ttt^*\mo^\top)\\ \notag
=&~ \biggr(\mz^\top,
\prod_{h\in S_1:\Gamma_{h,\aaa_1}=0}(\bar\theta_{\ee_h,\aaa_1} - \theta_{\ee_h,\mo})
\prod_{h\in S_1:\Gamma_{h,\aaa_1}=1}(\bar\theta_{\ee_h,\aaa_1} - \theta_{\ee_h,\aaa_0}),
\mz^\top\biggr),
\end{align}
and this is because for any other latent class $\aaa'\neq \aaa_1$, the $\aaa'$ is capable of at least one item in $S_1$ that $\aaa_1$ is not capable of. 
Now consider the row vector corresponding to response pattern $\rr+\ee_j$ for $j\in S_2$ in the transformed $T$-matrices, and we have
 \begin{align*}
& T_{\rr^*+\ee_j,\cdot}(\TT-\ttt^*\mo^\top)  \\ 
&= \Big(\mz^\top,
 \theta_{\ee_j,\aaa}\cdot
\prod_{h\in S_1:\Gamma_{h,\aaa_1}=0}(\theta_{\ee_h,\aaa_1} - \theta_{\ee_h,\mo})
\prod_{h\in S_1:\Gamma_{h,\aaa_1}=1}(\theta_{\ee_h,\aaa_1} - \theta_{\ee_h,\aaa_0}),
\mz^\top\Big),
\end{align*} 
and
 \begin{align*}
& T_{\rr^*+\ee_j,\cdot}(\bar\TT-\ttt^*\mo^\top)  \\
&= \Big(\mz^\top, \bar\theta_{\ee_j,\aaa_1}\cdot
\prod_{h\in S_1:\Gamma_{h,\aaa_1}=0}(\bar\theta_{\ee_h,\aaa_1} - \theta_{\ee_h,\mo})
\prod_{h\in S_1:\Gamma_{h,\aaa_1}=1}(\bar\theta_{\ee_h,\aaa_1} - \theta_{\ee_h,\aaa_0}),
\mz^\top\Big).
\end{align*} 
The above four equations along with Equation (\ref{eq1}) indicate for $j\in S_2$ we have
\[
\theta_{\ee_j,\aaa_1} = \bar\theta_{\ee_j,\aaa_1}.
\]
Similarly for $j\in S_1$ we also have $\theta_{\ee_{j},\aaa_1} = \bar\theta_{\ee_{j},\aaa_1}$. 
Plugging $\theta_{\ee_j,\aaa_1} = \bar\theta_{\ee_j,\aaa_1}$ into the equation $\eqref{eq-new1}\pp = \eqref{eq-new2}\bar \pp$ gives
$$
p_{\aaa_1} = \bar p_{\aaa_1}.
$$
So now we have shown \eqref{eq-theta-p} holds for $\aaa=\aaa_1$.

Then as the induction assumption,  suppose for any given $\aaa\in\mc$, we have
\[
\forall \aaa' \text{ s.t. } \aaa'\preceq_{S_1}\aaa,
\quad \forall j\in S_1\cup S_2,\quad \theta_{\ee_j,\aaa'} = \bar \theta_{\ee_j,\aaa'}
, \quad p_{\aaa'} = \bar p_{\aaa'}.
\]
Recall that $\aaa'\preceq_{S_1}\aaa$ if and only if $\aaa'\preceq_{S_2}\aaa$.
Define $\ttt^*$ as 
\[
\ttt^* = \sum_{h\in S_1:\Gamma_{h,\aaa}=0}\theta_{\ee_h,\mo}\ee_h
+ \sum_{h\in S_1:\Gamma_{h,\aaa}=1}\theta_{\ee_h,\aaa_0}\ee_h,
\]
then for $\rr^* := \sum_{h\in S_1}\ee_h$ we have 
\begin{align}\label{eq-induc-nn1}
T_{\rr^*,\cdot}(\TT-\ttt^*\mo^\top)& \pp
 = \sum_{\aaa'\preceq_{S_1}\aaa}t_{\rr^*,\aaa'}\cdot p_{\aaa'}\notag \\
 +& \prod_{h\in S_1:\Gamma_{h,\aaa}=0}(\theta_{\ee_h,\aaa} - \theta_{\ee_h,\mo})
\prod_{h\in S_1:\Gamma_{h,\aaa}=1}(\theta_{\ee_h,\aaa} - \theta_{\ee_h,\aaa_0})\cdot p_{\aaa},
 \end{align}

\begin{align}\label{eq-induc-nn2}
 T_{\rr^*,\cdot}(\bar\TT-\ttt^*\mo^\top)& \bar\pp
= \sum_{\aaa'\preceq_{S_1}\aaa}\bar t_{\rr^*,\aaa'}\cdot \bar p_{\aaa'} \notag\\
+& \prod_{h\in S_1:\Gamma_{h,\aaa}=0}(\bar\theta_{\ee_h,\aaa} - \theta_{\ee_h,\mo})
\prod_{h\in S_1:\Gamma_{h,\aaa}=1}(\bar\theta_{\ee_h,\aaa} - \theta_{\ee_h,\aaa_0})\cdot \bar p_{\aaa},
\end{align} 
where the notations $t_{\rr^*,\aaa'}$ and $\bar t_{\rr^*,\aaa'}$ are defined as
\[\begin{aligned}
t_{\rr^*,\aaa'} &= \prod_{h\in S_1:\Gamma_{h,\aaa}=0}(\theta_{\ee_h,\aaa'} - \theta_{\ee_h,\mo})
\prod_{h\in S_1:\Gamma_{h,\aaa}=1}(\theta_{\ee_h,\aaa} - \theta_{\ee_h,\aaa_0}), \\
\bar t_{\rr^*,\aaa'} &= \prod_{h\in S_1:\Gamma_{h,\aaa}=0}(\bar\theta_{\ee_h,\aaa'} - \theta_{\ee_h,\mo})
\prod_{h\in S_1:\Gamma_{h,\aaa}=1}(\bar\theta_{\ee_h,\aaa} - \theta_{\ee_h,\aaa_0}).
\end{aligned}\]
Note that by induction assumption we have $\theta_{\ee_h,\aaa} = \bar\theta_{\ee_h,\aaa}$ for any $\aaa'$ such that $\aaa'\preceq_{S_1}\aaa$. This implies $t_{\rr^*,\aaa'} = \bar t_{\rr^*,\aaa'}$ and further implies
\[
\sum_{\aaa'\preceq_{S_1}\aaa}t_{\rr^*,\aaa'}\cdot p_{\aaa'}
=
\sum_{\aaa'\preceq_{S_1}\aaa}\bar t_{\rr^*,\aaa'}\cdot\bar p_{\aaa'}.
\]
So (\ref{eq-induc-nn1}) = (\ref{eq-induc-nn2}) gives
\begin{equation}\label{eq-induc-nn3}
\begin{aligned}
&\prod_{h\in S_1:\Gamma_{h,\aaa}=0}(\theta_{\ee_h,\aaa} - \theta_{\ee_h,\mo})
\prod_{h\in S_1:\Gamma_{h,\aaa}=1}(\theta_{\ee_h,\aaa} - \theta_{\ee_h,\aaa_0})\cdot p_{\aaa} \\
=&
\prod_{h\in S_1:\Gamma_{h,\aaa}=0}(\bar\theta_{\ee_h,\aaa} - \theta_{\ee_h,\mo})
\prod_{h\in S_1:\Gamma_{h,\aaa}=1}(\bar\theta_{\ee_h,\aaa} - \theta_{\ee_h,\aaa_0})\cdot \bar p_{\aaa},
\end{aligned}\end{equation}
and the two terms on both hand sides of the above equation are nonzero.
Now consider any $j\notin S_1$ and similarly $T_{\rr^*+\ee_j,\cdot}(\TT-\ttt^*\mo^\top) \pp = T_{\rr^*+\ee_j,\cdot}(\bar\TT-\ttt^*\mo^\top)\bar \pp$ yields

\begin{equation}\label{eq-induc-nn4}
\begin{aligned}
&\theta_{\ee_j,\aaa}\cdot\prod_{h\in S_1:\Gamma_{h,\aaa}=0}(\theta_{\ee_h,\aaa} - \theta_{\ee_h,\mo})
\prod_{h\in S_1:\Gamma_{h,\aaa}=1}(\theta_{\ee_h,\aaa} - \theta_{\ee_h,\aaa_0})\cdot p_{\aaa} \\
=& ~\bar\theta_{\ee_j,\aaa}\cdot
\prod_{h\in S_1:\Gamma_{h,\aaa}=0}(\bar\theta_{\ee_h,\aaa} - \theta_{\ee_h,\mo})
\prod_{h\in S_1:\Gamma_{h,\aaa}=1}(\bar\theta_{\ee_h,\aaa} - \theta_{\ee_h,\aaa_0})\cdot \bar p_{\aaa}.
\end{aligned}\end{equation}
Taking the ratio of the above two equations \eqref{eq-induc-nn4} and \eqref{eq-induc-nn3} gives
\[
\theta_{\ee_j,\aaa}=\bar\theta_{\ee_j,\aaa},\quad \forall j\notin S_1.
\]
Redefining $\rr^* := \sum_{h\in S_2}\ee_h$ similarly as above we have $\theta_{\ee_j,\aaa}=\bar\theta_{\ee_j,\aaa}$ for any $j\in S_1$.
Plug $\theta_{\ee_j,\aaa}=\bar\theta_{\ee_j,\aaa}$ for all $j\in S_1$ into (\ref{eq-induc-nn3}), then we have 
$p_{\aaa} = \bar p_{\aaa}$.
Now we have shown \eqref{eq-theta-p} hold for this particular $\aaa$. 
Then the induction argument gives
\[
\forall \aaa\in\mc,\quad \forall j\in S_1\cup S_2,\quad
\theta_{\ee_j,\aaa} = \bar\theta_{\ee_j,\aaa},\quad p_{\aaa} = \bar p_{\aaa}.
\]
Combined with the results in Step 1 and 2, all the model parameters $(\TT,\pp)$ are identifiable and the proof of Theorem \ref{thm-order} is complete.
\end{proof}

\smallskip
\begin{proof}[Proof of Proposition \ref{new}]
Without loss of generality, assume $S_1=\{1,\ldots,M_1\}$ and $S_2=\{M_1+1,\ldots,M_1+M_2\}$. Recall that $\mb_{S_1} = \mb_{S_2}$ under condition (C3*).
The outline of the proof is as follows.
\vspace{4mm}
\newline
\textit{Step 1:}
$\theta_{\ee_j,\aaa_0} = \bar \theta_{\ee_j,\aaa_0}$ for $j>M_1+M_2$.
\newline
\textit{Step 2:}
$\theta_{\ee_j,\aaa} = \bar \theta_{\ee_j,\aaa}$ for $j>M_1+M_2$ and $\aaa\in\mb_{S_1}$.
\newline
\textit{Step 3:}
$\theta_{\ee_j,\aaa} = \bar \theta_{\ee_j,\aaa}$ and $p_{\aaa}=\bar p_{\aaa}$ for $1\leq j\leq M_1+M_2$, $\aaa=\aaa_0$ or $\aaa\in\mb_{S_1}$.
\newline
\textit{Step 4:}
$\theta_{\ee_j,\aaa} = \bar \theta_{\ee_j,\aaa}$ and $p_{\aaa} = \bar p_{\aaa}$ for $1\leq j\leq J$ and for all $\aaa$.

\vspace{4mm}
\noindent
Next we start the proof of the theorem.

\noindent
\textbf{Step 1}. The proof is exactly the same as Step 1 of Theorem \ref{thm-order}.
\vspace{2mm}

\noindent
\textbf{Step 2}. 
First consider basis latent classes $\aaa$ under both $S_1$ and $S_2$. For $\aaa\in\mathcal B_{S_1}$, define
\[\begin{aligned}
\ttt^* =& \sum_{j\in S_1:\Gamma_{j,\aaa}=1}\bar\theta_{\ee_j, \aaa_0}\ee_j
+ \sum_{j\in S_1:\Gamma_{j,\aaa}=0}\bar\theta_{\ee_j, \mo}\ee_j \\
&+ \sum_{j\in S_2:\Gamma_{j,\aaa}=1}\theta_{\ee_j, \aaa_0}\ee_j
+ \sum_{j\in S_2:\Gamma_{j,\aaa}=0}\theta_{\ee_j, \mo}\ee_j,
\end{aligned}\]
then the row vectors $\rr^*=\sum_{j=1}^{M_1+M_2}\ee_j$ in the transformed $T$-matrices only contain one potentially nonzero element, corresponding to $\aaa$, as follows
\begin{align}\label{eq-s2-1}
&T_{\rr^*,\cdot}(\TT-\ttt^*\mo^\top)
\notag\\
= &~ \Big(\mz^\top, 
 \prod_{j\in S_1:\Gamma_{j,\aaa}=1} (\theta_{\ee_j,\aaa} - \bar\theta_{\ee_j,\aaa_0})
\prod_{j\in S_1:\Gamma_{j,\aaa}=0} (\theta_{\ee_j,\aaa} - \bar\theta_{\ee_j,\mo})  \notag\\
&  \quad\quad\quad \times\prod_{j\in S_2:\Gamma_{j,\aaa}=1} (\theta_{\ee_j,\aaa} - \theta_{\ee_j,\aaa_0})
\prod_{j\in S_2:\Gamma_{j,\aaa}=0} (\theta_{\ee_j,\aaa} - \theta_{\ee_j,\mo}),~
 \mz^\top\Big),
\end{align} 
and
\begin{align}\label{eq-s2-2}
& T_{\rr^*,\cdot}(\TT-\ttt^*\mo^\top)
\notag\\
= &~ \Big(\mz^\top, 
  \prod_{j\in S_1:\Gamma_{j,\aaa}=1} (\bar\theta_{\ee_j,\aaa} - \bar\theta_{\ee_j,\aaa_0})
\prod_{j\in S_1:\Gamma_{j,\aaa}=0} (\bar\theta_{\ee_j,\aaa} - \bar\theta_{\ee_j,\mo})  \notag \\
&  \quad\quad\quad \times \prod_{j\in S_2:\Gamma_{j,\aaa}=1} (\bar\theta_{\ee_j,\aaa} - \theta_{\ee_j,\aaa_0})
\prod_{j\in S_2:\Gamma_{j,\aaa}=0} (\bar\theta_{\ee_j,\aaa} - \theta_{\ee_j,\mo}),
 	~\mz^\top\Big).
\end{align} 
Lemma \ref{lem-neq} implies the product elements in (\ref{eq-s2-1}) and (\ref{eq-s2-2}) are both nonzero. Then consider any $j>M_1+M_2$, the row vector corresponding to the response pattern $\rr^*+\ee_j$ in the transformed $T$-matrices are
\begin{align}\label{eq-s2-3}
 & T_{\rr^*+\ee_j,\cdot}(\TT-\ttt^*\mo^\top) 
 \notag\\
= &~ \Big(\mz^\top, 
\theta_{\ee_j,\aaa} \cdot \prod_{h\in S_1:\Gamma_{h,\aaa}=1} (\theta_{\ee_h,\aaa} - \bar\theta_{\ee_h,\aaa_0})
\prod_{h\in S_1:\Gamma_{h,\aaa}=0} (\theta_{\ee_h,\aaa} - \bar\theta_{\ee_h,\mo})  \notag\\
&    \quad \times \prod_{h\in S_2:\Gamma_{h,\aaa}=1} (\theta_{\ee_h,\aaa} - \theta_{\ee_h,\aaa_0})
\prod_{h\in S_2:\Gamma_{h,\aaa}=0} (\theta_{\ee_h,\aaa} - \theta_{\ee_h,\mo}),
 ~\mz^\top\Big),
\end{align} 
and
\begin{align}\label{eq-s2-4}
 & T_{\rr^*+\ee_j,\cdot}(\TT-\ttt^*\mo^\top)  \notag\\
=& \Big(\mz^\top, 
\bar\theta_{\ee_j,\aaa}\cdot \prod_{h\in S_1:\Gamma_{h,\aaa}=1} (\bar\theta_{\ee_h,\aaa} - \bar\theta_{\ee_h,\aaa_0})
\prod_{h\in S_1:\Gamma_{h,\aaa}=0} (\bar\theta_{\ee_h,\aaa} - \bar\theta_{\ee_h,\mo})  \notag\\
&   \quad \times \prod_{h\in S_2:\Gamma_{h,\aaa}=1} (\bar\theta_{\ee_h,\aaa} - \theta_{\ee_h,\aaa_0})
\prod_{h\in S_2:\Gamma_{h,\aaa}=0} (\bar\theta_{\ee_h,\aaa} - \theta_{\ee_h,\mo}),
 ~\mz^\top\Big).
\end{align} 
So we have
\[
\theta_{\ee_j,\aaa} = \frac{(\ref{eq-s2-3})\cdot\pp}{(\ref{eq-s2-1})\cdot\pp} = 
\frac{(\ref{eq-s2-4})\cdot\bar\pp}{(\ref{eq-s2-2})\cdot\bar\pp} = \bar \theta_{\ee_j,\aaa},
\quad \forall\aaa\in \mb_{S_1},\quad \forall j > M_1+M_2.
\]

\smallskip
\noindent
\textbf{Step 3}.
 {\it We first prove $\theta_{\ee_j,\mo} = \bar\theta_{\ee_j,\mo}$ for any $j\in S_1\cup S_2$.} 
Given $\aaa\in\mb_{S_1}$, define
\[
\ttt^* = \sum_{h\in S_1:\Gamma_{h,\aaa}=0}\theta_{\ee_h,\mo}\ee_h,
\]
then the row vector corresponding to $\rr^*=\sum_{h\in S_1:\Gamma_{h,\aaa}=0}\ee_h$ in the transformed $T$-matrix takes the following form
\[\begin{aligned}
& T_{\rr^*,\cdot}(\TT-\ttt^*\mo^\top) \\
& = 
\Big(\prod_{h\in S_1:\Gamma_{h,\aaa}=0}(\theta_{\ee_h,\aaa_0} - \theta_{\ee_h,\mo}),*,\ldots,*,
\prod_{h\in S_1:\Gamma_{h,\aaa}=0}(\theta_{\ee_h,\aaa} - \theta_{\ee_h,\mo}),0,\ldots,0\Big).
\end{aligned}\]
Condition (C4*) implies that $(\theta_{j,\aaa},j\in(S_1\cup S_2)^c)\neq (\theta_{j,\aaa_0},j\in(S_1\cup S_2)^c)$ for any basis latent class $\aaa\in\mb_{S_1}$.
So there exist a $C$-dimensional vector $\mm$ such that the element in $\mm^\top\cdot T(\TT_{(M_1+M_2+1):J})$ corresponding to $\aaa_0$ is 0 and the element corresponding to $\aaa$ is 1, i.e.,
\[\begin{aligned}
\mm^\top\cdot T(\TT_{(M_1+M_2+1):J}) &= (0,*,\ldots,*,\underbrace{1}_\text{column $\aaa$},*,\ldots,*),\\
\end{aligned}\]
and based on the conclusions of Step 2, we also have
\[\begin{aligned}
\mm^\top\cdot T(\bar\TT_{(M_1+M_2+1):J}) &= (0,*,\ldots,*,\underbrace{1}_\text{column $\aaa$},*,\ldots,*).\\
\end{aligned}\]
By Lemma \ref{lem-rk} $T(\bar\TT_{S_2})$ has full column rank $C$, hence there exists a vector $\vv$ such that
\[
\vv^\top\cdot T(\bar\TT_{S_2}) = (\mz,\underbrace{1}_\text{column $\aaa$},\mz)^\top,
\]
and denote the value in column $\aaa$ of $\vv^\top\cdot T(\TT_{S_2})$ by $b_{\vv,\aaa}$.
Consider any $j\in S_1\cup S_2$ such that $\Gamma_{j,\aaa}=1$, then obviously $\ee_j$ is not included in the sum in the previously defined response pattern $\rr^*$, because $\rr^*$ only contains those items that $\aaa$ is not capable of. So we have
\begin{equation}\label{eq-s3-1}
\begin{aligned}
& T_{\rr^*,\cdot}(\TT-\ttt^*\mo^\top)
\odot
\{\mm^\top\cdot T(\TT_{(M_1+M_2+1):J})\} 
\odot
\{\vv^\top\cdot T(\TT_{S_2})\} \\
=&~\Big( \mz^\top, 
\underbrace{b_{\vv,\aaa}\cdot\prod_{h\in S_1:\Gamma_{h,\aaa_k}=0}(\theta_{\ee_h,\aaa} - \theta_{\ee_h,\mo})}_\text{column $\aaa$},
 \mz^\top \Big),
\end{aligned}\end{equation}
\begin{equation}\label{eq-s3-2}
\begin{aligned}
& T_{\rr^*+\ee_j,\cdot}(\TT-\ttt^*\mo^\top)
\odot
\{\mm^\top\cdot T(\TT_{(M_1+M_2+1):J})\} 
\odot
\{\vv^\top\cdot T(\TT_{S_2})\} \\
=&~\Big( \mz^\top, 
\underbrace{\theta_{\ee_j,\mo}\cdot b_{\vv,\aaa}\cdot\prod_{h\in S_1:\Gamma_{h,\aaa}=0}(\theta_{\ee_h,\aaa} - \theta_{\ee_h,\mo})}_\text{column $\aaa$},
 \mz^\top \Big).
\end{aligned}\end{equation}
Similarly for $(\bar\TT,\bar\pp)$ we have
\begin{equation}\label{eq-s3-3}
\begin{aligned}
& T_{\rr^*,\cdot}(\bar\TT-\ttt^*\mo^\top)
\odot
\{\mm^\top\cdot T(\bar\TT_{(M_1+M_2+1):J})\} 
\odot
\{\vv^\top\cdot T(\bar\TT_{S_2})\} \\
=&~\Big( \mz^\top, 
\underbrace{\prod_{h\in S_1:\Gamma_{h,\aaa}=0}(\bar\theta_{\ee_h,\aaa} - \theta_{\ee_h,\mo})}_\text{column $\aaa$},
 \mz^\top \Big),
\end{aligned}\end{equation}
\begin{equation}\label{eq-s3-4}
\begin{aligned}
& T_{\rr^*+\ee_j,\cdot}(\bar\TT-\ttt^*\mo^\top)
\odot
\{\mm^\top\cdot T(\bar\TT_{(M_1+M_2+1):J})\} 
\odot
\{\vv^\top\cdot T(\bar\TT_{S_2})\} \\
=&~\Big( \mz^\top, \underbrace{\bar\theta_{\ee_j,\mo}\cdot\prod_{h\in S_1:\Gamma_{h,\aaa}=0}
(\bar\theta_{\ee_h,\aaa} - \theta_{\ee_h,\mo})}_\text{column $\aaa$}, \mz^\top \Big).
\end{aligned}\end{equation}
Equation (\ref{eq1}) implies $(\ref{eq-s3-1})\cdot\pp = (\ref{eq-s3-3})\cdot\bar\pp$, and since $(\ref{eq-s3-3})\cdot\bar\pp\neq 0$, we must also have $(\ref{eq-s3-1})\cdot\bar\pp\neq 0$, which indicates $b_{\vv,\aaa}\neq 0$.
The above four equations along with (\ref{eq1}) give that
\[
\theta_{\ee_j,\mo} = \theta_{\ee_j,\aaa} =
\frac{(\ref{eq-s3-2})\cdot\pp}{(\ref{eq-s3-1})\cdot\pp} = \frac{(\ref{eq-s3-4})\cdot\bar\pp}{(\ref{eq-s3-3})\cdot\bar\pp}
= \bar\theta_{\ee_j,\aaa} = \bar\theta_{\ee_j,\mo},\quad \forall j\in S_2.
\]
Note that the above equality $\theta_{\ee_{j},\mo} = \bar\theta_{\ee_{j},\mo}$ holds for any $\aaa$ and any item $j$ such that $\Gamma_{j,\aaa}=1$. 
Therefore we have shown $\theta_{\ee_{j},\mo} = \bar\theta_{\ee_{j},\mo}$ holds for any $j\in S_1\cup S_2$.
Similarly we also have $\theta_{\ee_{j},\aaa_0} = \bar\theta_{\ee_{j},\aaa_0}$. In summary,
\[\begin{aligned}
\theta_{\ee_{j},\aaa_0} = \bar\theta_{\ee_{j},\aaa_0},\quad
\theta_{\ee_{j},\mo} = \bar\theta_{\ee_{j},\mo},\quad 
\forall j\in S_1 \cup S_2. \\
\end{aligned}\]
For $\aaa=\aaa_0$ define
\[
\ttt^{*} = \sum_{h\in S_1} \theta_{\ee_h,\mo}\ee_h,
\]
then $T_{\sum_{h\in S_1}\ee_h}(\TT-\ttt^{*}\mo^\top)\pp
= T_{\sum_{h\in S_1}\ee_h}(\bar\TT-\ttt^{*}\mo^\top)\bar\pp$ gives
\[
\prod_{h\in S_1}(\theta_{\ee_h,\aaa_0} - \theta_{\ee_h,\mo})p_{\aaa_0} =
\prod_{h\in S_1}(\theta_{\ee_h,\aaa_0} - \theta_{\ee_h,\mo})\bar p_{\aaa_0},
\]
so we also have
$p_{\aaa_0}=\bar p_{\aaa_0}.$

\smallskip
 {\it Next we show $\theta_{\ee_j,\aaa}=\bar\theta_{\ee_j,\aaa}$ for any $\aaa\in\mb_{S_1}$ and $j\in S_1\cup S_2$, where $\Gamma_{j,\aaa}=0$.}
Given $\aaa$, define
\begin{equation}\label{eq-t-s3}
\ttt^* = \sum_{h\in S_1:\Gamma_{h,\aaa}=0}\theta_{\ee_h,\mo}\ee_h
+ \sum_{h\in S_1:\Gamma_{h,\aaa}=1}\theta_{\ee_h,\aaa_0}\ee_h,
\end{equation}
then the row vectors of $\rr^*=\sum_{h\in S_1}\ee_h$ in the transformed $T$-matrices only contain one nonzero element corresponding to column $\aaa$ as follows
\[
T_{\rr^*,\cdot}(\TT-\ttt^*\mo^\top) = \biggr(\mz^\top,
\prod_{h\in S_1:\Gamma_{h,\aaa}=0}(\theta_{\ee_h,\aaa} - \theta_{\ee_h,\mo})
\prod_{h\in S_1:\Gamma_{h,\aaa}=1}(\theta_{\ee_h,\aaa} - \theta_{\ee_h,\aaa_0}),
\mz^\top\biggr),
\]
\[
T_{\rr^*,\cdot}(\bar\TT-\ttt^*\mo^\top) = \biggr(\mz^\top,
\prod_{h\in S_1:\Gamma_{h,\aaa}=0}(\bar\theta_{\ee_h,\aaa} - \theta_{\ee_h,\mo})
\prod_{h\in S_1:\Gamma_{h,\aaa}=1}(\bar\theta_{\ee_h,\aaa} - \theta_{\ee_h,\aaa_0}),
\mz^\top\biggr).
\]
Now consider the row vectors of $\rr+\ee_j$ for $j\in S_2$ in the transformed $T$-matrices, we have
 \begin{align*}
& T_{\rr^*+\ee_j,\cdot}(\TT-\ttt^*\mo^\top)\\
&  = \Big(\mz^\top,~ \theta_{\ee_j,\aaa}\cdot
\prod_{h\in S_1:\Gamma_{h,\aaa}=0}(\theta_{\ee_h,\aaa} - \theta_{\ee_h,\mo})
\prod_{h\in S_1:\Gamma_{h,\aaa}=1}(\theta_{\ee_h,\aaa} - \theta_{\ee_h,\aaa_0}),
~\mz^\top\Big),
\end{align*} 
and
 \begin{align*}
& T_{\rr^*+\ee_j,\cdot}(\bar\TT-\ttt^*\mo^\top) \\
&= \Big(\mz^\top, ~ \bar\theta_{\ee_j,\aaa}\cdot
\prod_{h\in S_1:\Gamma_{h,\aaa}=0}(\bar\theta_{\ee_h,\aaa} - \theta_{\ee_h,\mo})
\prod_{h\in S_1:\Gamma_{h,\aaa}=1}(\bar\theta_{\ee_h,\aaa} - \theta_{\ee_h,\aaa_0}),
~\mz^\top\Big).
\end{align*} 
The above four equations along with Equation (\ref{eq1}) indicate for $j\in S_2$ we have
\[
\theta_{\ee_j,\aaa} = \bar\theta_{\ee_j,\aaa}.
\]
Similarly for $\aaa\in\mb_{S_1}$, $j\in S_1$ we also have $\theta_{\ee_{j},\aaa} = \bar\theta_{\ee_{j},\aaa}$. In summary, we have
\[
\theta_{\ee_j,\aaa} = \bar\theta_{\ee_j,\aaa},\quad \forall \aaa\in\mathcal B_{S_1},\quad \forall j\in S_1\cup S_2.
\]
Now for $\aaa\in\mb_{S_1}$ define
\[
\ttt^{*} = \sum_{h\in S_1:\Gamma_{h,\aaa}=0} \theta_{\ee_h,\mo}\ee_h,
\]
then $T_{\sum_{h\in S_1}\ee_h}(\TT-\ttt^{*}\mo^\top)\pp
= T_{\sum_{h\in S_1}\ee_h}(\bar\TT-\ttt^{*}\mo^\top)\bar\pp$ gives
\[\begin{aligned}
&\prod_{h\in S_1}(\theta_{\ee_h,\aaa_0} - \theta_{\ee_h,\mo})p_{\aaa_0} + \prod_{h\in S_1}(\theta_{\ee_h,\aaa} - \theta_{\ee_h,\mo})p_{\aaa}\\
 =
&\prod_{h\in S_1}(\theta_{\ee_h,\aaa_0} - \theta_{\ee_h,\mo})p_{\aaa_0} + \prod_{h\in S_1}(\theta_{\ee_h,\aaa} - \theta_{\ee_h,\mo})\bar p_{\aaa},
\end{aligned}\]
which implies $p_{\aaa}=\bar p_{\aaa}$. This completes the proof of Step 3.


\smallskip
\noindent
\textbf{Step 4}.
We use the induction method to prove the conclusions for those $\aaa\notin \mb_{S_1}$. In previous steps we already established 
\[
p_{\aaa_0} = \bar p_{\aaa_0},\quad \theta_{\ee_j,\aaa_0} = \bar \theta_{\ee_j,\aaa_0},\quad\forall j\in\{1,\ldots,J\},
\]
and
\[
p_{\aaa} = \bar p_{\aaa},\quad \theta_{\ee_j,\aaa} = \bar \theta_{\ee_j,\aaa},\quad \forall\aaa\in\mb_{S_1},\quad \forall j\in\{1,\ldots,J\}.
\]
So as the induction assumption, suppose for any given $\aaa\notin\mb_{S_1}$, we have
\[
p_{\aaa'} = \bar p_{\aaa'}, \quad \theta_{\ee_j,\aaa'} = \bar \theta_{\ee_j,\aaa'},\quad
\forall \aaa' \text{ s.t. } \aaa'\preceq_{S_1}\aaa
\quad \forall j\in\{1,\ldots,J\}.
\]
Recall that $\aaa'\preceq_{S_1}\aaa$ if and only if $\aaa'\preceq_{S_2}\aaa$.
Define $\ttt^*$ as that in (\ref{eq-t-s3})
\[
\ttt^* = \sum_{h\in S_1:\Gamma_{h,\aaa}=0}\theta_{\ee_h,\mo}\ee_h
+ \sum_{h\in S_1:\Gamma_{h,\aaa}=1}\theta_{\ee_h,\aaa_0}\ee_h,
\]
then the row vector corresponding to $\rr^*=\sum_{h\in S_1}\ee_h$ in the transformed $T$-matrix takes the form
\begin{align}\label{eq-induc1}
& T_{\rr^*+\ee_j,\cdot}(\TT-\ttt^*\mo^\top) \pp =  \sum_{\aaa'\preceq_{S_1}\aaa}t_{\rr^*,\aaa'}\cdot p_{\aaa'} \\
&\quad + \prod_{h\in S_1:\Gamma_{h,\aaa}=0}(\theta_{\ee_h,\aaa} - \theta_{\ee_h,\mo})
\prod_{h\in S_1:\Gamma_{h,\aaa}=1}(\theta_{\ee_h,\aaa} - \theta_{\ee_h,\aaa_0})\cdot p_{\aaa},\notag
\end{align} 
\begin{align}\label{eq-induc2}
&  T_{\rr^*+\ee_j,\cdot}(\bar\TT-\ttt^*\mo^\top) \bar\pp
=  \sum_{\aaa'\preceq_{S_1}\aaa}\bar t_{\rr^*,\aaa'}\cdot \bar p_{\aaa'} \\
&\quad+ \prod_{h\in S_1:\Gamma_{h,\aaa}=0}(\bar\theta_{\ee_h,\aaa} - \theta_{\ee_h,\mo})
\prod_{h\in S_1:\Gamma_{h,\aaa}=1}(\bar\theta_{\ee_h,\aaa} - \theta_{\ee_h,\aaa_0})\cdot \bar p_{\aaa},\notag
\end{align} 
where the notations $t_{\rr^*,\aaa'}$ and $\bar t_{\rr^*,\aaa'}$ are defined as
\[\begin{aligned}
t_{\rr^*,\aaa'} &= \prod_{h\in S_1:\Gamma_{h,\aaa}=0}(\theta_{\ee_h,\aaa'} - \theta_{\ee_h,\mo})
\prod_{h\in S_1:\Gamma_{h,\aaa}=1}(\theta_{\ee_h,\aaa} - \theta_{\ee_h,\aaa_0}), \\
\bar t_{\rr^*,\aaa'} &= \prod_{h\in S_1:\Gamma_{h,\aaa}=0}(\bar\theta_{\ee_h,\aaa'} - \theta_{\ee_h,\mo})
\prod_{h\in S_1:\Gamma_{h,\aaa}=1}(\bar\theta_{\ee_h,\aaa} - \theta_{\ee_h,\aaa_0}).
\end{aligned}\]
Note that by induction assumption we have $\theta_{\ee_h,\aaa} = \bar\theta_{\ee_h,\aaa}$ for any $\aaa'$ such that $\aaa'\preceq_{S_1}\aaa$. This implies $t_{\rr^*,\aaa'} = \bar t_{\rr^*,\aaa'}$ and further implies
\[
\sum_{\aaa'\preceq_{S_1}\aaa}t_{\rr^*,\aaa'}\cdot p_{\aaa'}
=
\sum_{\aaa'\preceq_{S_1}\aaa}\bar t_{\rr^*,\aaa'}\cdot\bar p_{\aaa'}.
\]
So (\ref{eq-induc1}) = (\ref{eq-induc2}) gives
\begin{equation}\label{eq-induc3}
\begin{aligned}
&\prod_{h\in S_1:\Gamma_{h,\aaa}=0}(\theta_{\ee_h,\aaa} - \theta_{\ee_h,\mo})
\prod_{h\in S_1:\Gamma_{h,\aaa}=1}(\theta_{\ee_h,\aaa} - \theta_{\ee_h,\aaa_0})\cdot p_{\aaa} \\
=&
\prod_{h\in S_1:\Gamma_{h,\aaa}=0}(\bar\theta_{\ee_h,\aaa} - \theta_{\ee_h,\mo})
\prod_{h\in S_1:\Gamma_{h,\aaa}=1}(\bar\theta_{\ee_h,\aaa} - \theta_{\ee_h,\aaa_0})\cdot \bar p_{\aaa}.
\end{aligned}\end{equation}
Consider any $j\notin S_1$ and similarly we have
\begin{equation}\label{eq-induc4}
\begin{aligned}
&\theta_{\ee_j,\aaa}\cdot\prod_{h\in S_1:\Gamma_{h,\aaa}=0}(\theta_{\ee_h,\aaa} - \theta_{\ee_h,\mo})
\prod_{h\in S_1:\Gamma_{h,\aaa}=1}(\theta_{\ee_h,\aaa} - \theta_{\ee_h,\aaa_0})\cdot p_{\aaa} \\
=& \bar\theta_{\ee_j,\aaa}\cdot
\prod_{h\in S_1:\Gamma_{h,\aaa}=0}(\bar\theta_{\ee_h,\aaa} - \theta_{\ee_h,\mo})
\prod_{h\in S_1:\Gamma_{h,\aaa}=1}(\bar\theta_{\ee_h,\aaa} - \theta_{\ee_h,\aaa_0})\cdot \bar p_{\aaa}.
\end{aligned}\end{equation}
Taking the ratio of the above two equations gives
\[
\theta_{\ee_j,\aaa}=\bar\theta_{\ee_j,\aaa},\quad \forall j\notin S_1.
\]
Similarly we have $\theta_{\ee_j,\aaa}=\bar\theta_{\ee_j,\aaa}$ for any $j\in S_1$.
Plug in $\theta_{\ee_j,\aaa}=\bar\theta_{\ee_j,\aaa}$ for all $j\in S_1$ into (\ref{eq-induc3}), then we have 
\[
p_{\aaa} = \bar p_{\aaa}.
\]
This completes the proof of Proposition \ref{new}.
\end{proof}

\smallskip
\begin{proof}[Proof of Theorem \ref{thm-order-gen}]
Without loss of generality, we show the generic identifiability statement holds on the  parameter space $\mathcal T$
\[
\mathcal T = 
\Big\{(\TT,\pp):\, \forall~j, \max_{\aaa:\Gamma_{j,\aaa}=1}\theta_{j,\aaa} = \min_{\aaa:\Gamma_{j,\aaa}=1} \theta_{j,\aaa}>\theta_{j,\aaa'} \geq \theta_{j,\aaa_0},~\forall~ \Gamma_{j,\aaa'}=0 \Big\}.
\] 
On $\mathcal T$, altering some entries of zero to one in the $\Gamma$-matrix is equivalently imposing more affine constraints on the parameters and force them to be in a subset 
$\mathcal T^*$ of $\mathcal T$. 
Since Condition (C3) holds for model parameters belonging to the space $\mathcal T^*$, the proof of Theorem \ref{thm-order} gives that the matrix $T(\TT_{S_i})$ has full column rank $C$ for $i=1,2$ for $(\TT_{S_i},\pp)\in\mathcal T^*$.
Note that saying the $2^{|S_i|}\times C$ matrix $T(\TT_{S_i})$ has full column rank is equivalently saying the map sending $T(\TT_{S_i})$ to all its $\binom{2^{|S_i|}}{C}$ possible $C\times C$ minors $A_1^i,A_2^i,\ldots, A^i_{2^{|S_i|}}$ yields at least one nonzero minor, where $A^i_1,A^i_2,\ldots, A^i_{2^{|S_i|}}$ are all polynomials of the item parameters $\TT_{S_i}$. Define 
\[\mathcal V = \bigcup_{i=1,2} \Big\{\bigcap_{l=1}^{2^{|S_i|}} \{(\TT,\pp)\in \mathcal T: A^i_l(\TT_{S_i}) = 0\}\Big\},
\]
then $\mathcal V$ is a algebraic variety defined by polynomials of the model parameters. Moreover, $\mathcal V$ is a proper subvariety of $\mathcal T$, since the fact $T(\TT_{S_i})$ has full column rank $C$ for $i=1,2$ for one particular set of  $(\TT,\pp)\in\mathcal T^*$ ensures that there exists one particular set of model parameters that give nonzero values when plugged into the polynomials defining $\mathcal V$, which indicates that the polynomials defining $\mathcal V$ are not all zero polynomials.

This implies for generic choices of $(\TT,\pp)$ in the space $\mathcal T$,  $T(\TT_{S_i})$ has full column rank for $i=1,2$. Together with the assumption that (C4) holds for $\Gamma$, 
we obtain  generic identifiability of the model parameters. This completes the proof of Theorem \ref{thm-order-gen}.
\end{proof}

\smallskip
\begin{proof}[Proof of Theorem \ref{thm-gen-q}]
In the following, we say some statement ``generically" holds, if the subset of the parameter space where the statement does not hold is of Lebesgue measure zero. 
Without loss of generality, assume $Q$ takes the form
\[
Q = 
\left(\begin{array}{c} 
Q_1 \\
Q_2 \\
Q'
\end{array} \right),
\]
where under the assumptions of Theorem \ref{thm-gen-q}, $Q_1$ and $Q_2$ are $K\times K$ square matrices with diagonal elements all equal to 1. With a slight abuse of notation, for a $J_i\times K$ submatrix $Q_i$ of $Q$, let $T(Q_i,\TT_{Q_i})$ denote the $2^{J_i}\times 2^K$ $T$-matrix.
We consider the saturated model where all the main effect and interaction effect terms are included in modeling the item parameters, namely the positive response probability for attribute profile $\aaa$ and item $j$ takes the form
\begin{align}\label{eq-tja}
\theta_{j,\aaa} 
=& f\Big(\beta_{j,0} + \sum_{k=1}^K\beta_{j,k}q_{jk}\alpha_k  + \sum_{k'=k+1}^K\sum_{k=1}^{K-1}\beta_{j,kk'}(q_{jk}\alpha_k)(q_{jk'}\alpha_{k'}) \\
&+ \cdots + \beta_{j,12\cdots K}\prod_k (q_{jk}\alpha_k)\Big),\notag
\end{align}
where $f(\Cdot)$ is the link function, which can be the identify link, log link, or the logistic link. Note that taking those $\beta$-coefficients of the interaction terms to be zero, one is left with a main-effect model. Since the following arguments only rely on the main effect coefficients, the conclusion of the theorem applies to any multi-parameter restricted latent class model.

First prove that under condition (C5), the $T$-matrices $T(Q_1,\TT_{Q_1})$ and $T(Q_2,\TT_{Q_2})$ corresponding to $Q_1$ and $Q_2$ are both generically of full rank $2^K$. To show  generic identifiability, it suffices to find one specific set of item parameters $\TT$ satisfying the constraints imposed by the $Q$-matrix that make the $T$-matrices $T(Q_1,\TT_{Q_1})$ and $T(Q_2,\TT_{Q_2})$  have full rank. In the following we focus on $T(Q_1,\TT_{Q_1})$ only. For $k=1,\ldots,K$, set the $k$'th main effect parameter of the $k$'th item to be 1, i.e., set $\beta_{k,k}=1$, and all the other main effect and interaction effect parameters to be zero, then the $T$-matrix $T(Q_1,\TT_{Q_1})$ now becomes exactly the same as the $T$-matrix $T(\mathcal I_K, \widetilde \TT_{\mathcal I_K})$ under the identity $Q$-matrix $\mathcal I_K$ with the item parameters being 
\[
\tilde\theta_{\ee_k,\mz} = \beta_{k,0}\quad \text{and}\quad \tilde\theta_{\ee_k,\ee_k} = \tilde\theta_{\ee_k,\mo} = \beta_{k,0}+\beta_{k,k}~~ \text{for}~~ k\in\{1,\ldots,K\}.
\]
Moreover, defining $\tilde\ttt^* = (\tilde\theta_{\ee_1,\mo},\ldots,\tilde\theta_{\ee_K,\mo})^T$ and following a similar argument as in the proof of Lemma \ref{lem-rk}, we have that $T({\mathcal I_K}, \widetilde \TT - \tilde\ttt^*\mo^\top)$ takes an botomn-left triangular form with nonzero diagonal entries, thus Proposition \ref{prop-ltrans} gives that $T(\mathcal I_K, \widetilde \TT_{\mathcal I_K})$ is full-rank. Therefore $T(Q_1,\TT_{Q_1})$ is generically full-rank. Similarly $T(Q_2,\TT_{Q_2})$ is also generically full-rank.

We next show that if condition (C6) additionally holds, then any two different columns indexed by attribute profiles $\aaa$ and $\aaa'$ of $T(Q', \TT_{Q'})$ are generically distinct. For distinct $\aaa$, $\aaa'\in\{0,1\}^K$, they at least differ in one attribute $k$. Without loss of generality, assume $\alpha_k=1>0=\alpha'_k$. Condition (C6) ensures that there exists some item $j>2K$ such that $q_{j,k}=1$.  Under the model considered here with $\theta_{j,\aaa}$ in the form of \eqref{eq-tja}, this implies $\theta_{j,\aaa}\neq\theta_{j,\aaa'}$ generically.  

Next we introduce a result of uniqueness of three-way tensor decomposition to facilitate our proof.
Following \cite{kruskal1977three}, the Kruskal rank of a matrix is the the largest number $I$ such that every $I$ columns of the matrix are independent. For a matrix $M$, let $\text{rank}_K(M)$ denote its Kruskal rank.
From \cite{kruskal1977three}, \cite{rhodes2010concise} has the following result. 
\begin{lemma}[Rhodes, 2010]\label{lem-kruskal}
For a matrix $M_i$, denote the $j$th column of it by $\mm_j^i$.
Given matrices $M_i$ of size $s_i\times c$, let the matrix triple product $[M_1, M_2, M_3]$ be an $s_1\times s_2\times s_3$ tensor defined as a sum of $r$ rank-1 tensors by
\[
[M_1, M_2, M_3]=\sum_{j=1}^c \mm_j^1\otimes\mm_j^2\otimes\mm_j^3.
\]
Suppose $\text{rank}_{K}(M_1)=\text{rank}_{K}(M_2)=c$ and $\text{rank}_{K}(M_3)\allowbreak \geq 2$; 
$N_1$, $N_2$, $N_3$ are matrices with $c$ columns and $[M_1, M_2, M_3]=[N_1,N_2,N_3]$. Then there exists some permutation matrix $P$ and invertible diagonal matrices $D_i$ with $D_1D_2D_3 = \mathcal I_c$ such that $N_i = M_iD_iP$.
\end{lemma}

Now we consider three $T$-matrices, $T(Q_1,\allowbreak \TT_{Q_1})$, $T(Q_2,\allowbreak \TT_{Q_2})$ and $T(Q', \allowbreak\TT_{Q'})$, which are of size $2^K\times 2^K$, $2^K\times 2^K$ and $2^{J-2K}\times 2^K$. The rows of the three matrices are indexed by possible item combinations in the three item sets $\{1,\ldots,K\}$, $\{K+1,\ldots,2K\}$ and $\{2K+1,\ldots,J\}$ respectively. We use $\text{Diag}(\pp)$ to denote a diagonal matrix with the diagonal entries being elements of $\pp$, then it is not hard to see that $T(\TT)\pp$ is given by the matrix triple product $[T(Q_1, \TT_{Q_1}),\allowbreak T(Q_2, \TT_{Q_2}),\allowbreak T(Q', \allowbreak\TT_{Q'})\cdot\text{Diag}(\pp)]$, namely the matrix triple product of the three matrices exactly characterizes the distribution of the response vector $\RR$.
Clearly if a matrix has full column rank, then its Kruskal rank equals its rank, thus our previous arguments already established that
$\text{rank}_{K} \{T(Q_1, \TT_{Q_1})\} = \text{rank}_{K} \{T(Q_2, \TT_{Q_2})\} =2^K$ and $\text{rank}_K \{T(Q', \TT_{Q'})\}  \geq  2$ hold generically.
Moreover, 
we claim $\text{rank}_K \{T(Q',\allowbreak \TT_{Q'}) \allowbreak\cdot\text{Diag}(\pp)\}\geq 2$ also holds generically. This is because if all the entries of $\pp$ are positive, which is a generic requirement, then multiplying the invertible diagonal matrix $\text{Diag}(\pp)$ by the matrix $T(Q',\allowbreak \TT_{Q'})$ would not change the Kruskcal rank of the latter. Now apply Lemma \ref{lem-kruskal} and follow a similar argument as the proof of Theorem 4 in \cite{allman2009}, we have the conclusion that the model is generically identifiable up to label swapping. 
Specifically, the label swapping would happen only between those latent classes which have identical ideal response vectors, namely the labels of $\aaa_1$ and $\aaa_2$ could possibly be swapped \textit{only if} $\Gamma_{\Cdot,\aaa_1} = \Gamma_{\Cdot,\aaa_2}$. This is because otherwise, the constraints (2.1)
introduced by the $\Gamma$-matrix would fail to hold.
\end{proof}

\begin{proof}[Proof of Theorem \ref{prop-vio-gen}]
We first prove the conclusion of the part (a), then that of the part (b).

\vspace{3mm}
\noindent
\textbf{Proof of   (a):}
Without loss of generality assume the $Q$-matrix takes the following form
\begin{equation}\label{eq-attr1-nogen}
Q = \left(
\begin{array}{cc}
1 & \vv \\
\mz & Q^*
\end{array}
\right),
\end{equation}
then given any set of valid parameters $(\TT,\pp)$, one can construct another set of model parameters $(\bar\TT,\bar\pp)$ as follows. First set all the item parameters associated items $j\geq 2$ to be the same as the true parameters for this second set of parameters.  For any $\aaa':=\aaa_{2:K}\in\{0,1\}^{K-1}$, choose $\bar\theta_{1,(1,\aaa')}\neq \theta_{1,(1,\aaa')}$ to be any reasonable value in a small neighborhood of $\theta_{1,(1,\aaa')}$. Set $\bar\theta_{1,(0,\aaa')} = \theta_{1,(0,\aaa')}$ and
\[
\begin{cases}
\bar p_{(0,\aaa')} = p_{(0,\aaa') }+ \left(1 - \frac{\theta_{1,(1,\aaa')}}{\bar\theta_{1,(1,\aaa')}} \right) p_{(1,\aaa')}; \\
\bar p_{(1,\aaa')} = \frac{\theta_{1,(1,\aaa')}}{\bar\theta_{1,(1,\aaa')}} p_{(1,\aaa')},
\end{cases}
\]
then we have 
\[\begin{cases}
\bar p_{(0,\aaa')} + \bar p_{(1,\aaa')} = p_{(0,\aaa')} +p_{(1,\aaa')};\\
\bar\theta_{1,(0,\aaa')}\bar p_{(0,\aaa')} + \bar\theta_{1,(1,\aaa')} \bar p_{(1,\aaa')} = \theta_{1,(0,\aaa')} p_{(0,\aaa')} + \theta_{1,(1,\aaa')} p_{(1,\aaa')}.
\end{cases}\]
With these two equations, for any $\rr\in\{0,1\}^J$ 
\[\begin{aligned}
& T_{\rr,\Cdot}(\bar\TT)\bar\pp \\
=& \sum_{\aaa'\in\{0,1\}^{K-1}}\prod_{j>1:r_j = 1}\bar\theta_{j,(0,\aaa')}^{r_j}\Big[ \bar\theta_{1,(0,\aaa')}^{r_1}\bar p_{(0,\aaa')} + \bar\theta_{1,(1,\aaa')}^{r_1}\bar p_{(1,\aaa')} \Big] \\
=&
\sum_{\aaa'\in\{0,1\}^{K-1}}\prod_{j>1:r_j = 1}\theta_{j,(0,\aaa')}^{r_j} \Cdot
\begin{cases}
\big[ \bar p_{(0,\aaa')} + \bar p_{(1,\aaa')} \big]  & \mbox{if }  r_1 = 0; \\
\big[ \bar\theta_{1,(0,\aaa')}\bar p_{(0,\aaa')} + \bar\theta_{1,(1,\aaa')} \bar p_{(1,\aaa')} \big]  & \mbox{if } r_1 = 1, \\
\end{cases} \\
=& ~ T_{\rr,\Cdot}(\TT)\pp.
\end{aligned}\]
This proves the model associated with $Q$ in the form of \eqref{eq-attr1-nogen} is not generically identifiable, since for any valid set of true parameters there exist another set of parameters resulting in the same distribution of the observed responses $\RR$.

\vspace{3mm}
\noindent
\noindent
\textbf{Proof of   (b):}
Part of the proof idea is similar to that of Theorem \ref{thm_cond}.
Since the $(J-2)\times(K-1)$ sub-matrix $Q'$ satisfies conditions (C5) and (C6), Theorem \ref{thm-gen-q} gives that, for generic choice of true parameters $(\TT,\pp)$ in the parameter space, if another set of parameters $(\bar\TT,\bar\pp)$ satisfy $T(\TT)\pp = T(\bar\TT)\bar\pp$, then 
\[\forall j\geq 3,\quad
\theta_{j,\,(0,\aaa_{2:K})} = \bar\theta_{j,\,(0,\aaa_{2:K})},\quad p_{(0,\aaa_{2:K})} + p_{(1,\aaa_{2:K})} = \bar p_{(0,\aaa_{2:K})} + \bar p_{(1,\aaa_{2:K})} .
\]
For any response pattern $\rr=(r_1,r_2,r_3,\ldots,r_j)\in\{0,1\}^J$, \eqref{eq-tpr} for $\rr$ can be equivalently written as
\begin{align}\label{eq-qprime-nn0}
&\sum_{\aaa_{2:K}\in\{0,1\}^{K-1}}\prod_{j>2:\,r_j=1}\theta_{j,\,(0,\aaa_{2:K})}\cdot \mathbb P(R_1\geq r_1,\,R_2\geq r_2,\, \ba_{2:K}=\aaa_{2:K})\\
=
&\sum_{\aaa_{2:K}\in\{0,1\}^{K-1}}\prod_{j>2:\,r_j=1}\bar\theta_{j,\,(0,\aaa_{2:K})}\cdot\overline {\mathbb P}(R_1\geq r_1,\,R_2\geq r_2,\,\ba_{2:K}=\aaa_{2:K}).\notag
\end{align}
Note that the difference of \eqref{eq-qprime-nn0} and \eqref{eq-qprime-n0} is that $\mathcal R^{Q'}$ is replaced by $\{0,1\}^{K-1}$, which is because when considering generic identifiability of multi-parameter $Q$-restricted models, all the $2^{K-1}$ possible proportion parameters resulting from the $Q'$ part are generically identifiable under conditions (C5) and (C6).

Following the same reasoning as in the proof of Theorem \ref{thm_cond}, part (B.2), we have
\begin{eqnarray}\label{eq-nf}
&&\mathbb  P(R_1\geq r_1,\,R_2\geq r_2,\, \ba_{2:K}=\aaa_{2:K})\\
&&= \overline{ \mathbb P}(R_1\geq r_1,\,R_2\geq r_2,\, \ba_{2:K}=\aaa_{2:K}),\notag
\end{eqnarray}
and this yields that for any $\aaa':=\aaa_{2:K}\in\{0,1\}^{K-1}$,
\begin{align}\label{eq-allv}
\begin{cases}
 p_{(0,\aaa')}+  p_{(1,\aaa')} = \bar p_{(0,\aaa')}+\bar  p_{(1,\aaa')};\\
\theta_{1,(0,\aaa')}\cdot p_{(0,\aaa')} + \theta_{1,(1,\aaa')}\cdot p_{(1,\aaa')} = \bar\theta_{1,(0,\aaa')}\cdot\bar p_{(0,\aaa')} + \bar \theta_{1,(1,\aaa')}\cdot\bar p_{(1,\aaa')};\\
\theta_{2,(0,\aaa')}\cdot p_{(0,\aaa')} + \theta_{2,(1,\aaa')}\cdot p_{(1,\aaa')} = \bar\theta_{2,(0,\aaa')}\cdot\bar p_{(0,\aaa')} + \bar \theta_{2,(1,\aaa')}\cdot\bar p_{(1,\aaa')};\\
\theta_{1,(0,\aaa')}\theta_{2,(0,\aaa')}\cdot p_{(0,\aaa')} + \theta_{1,(1,\aaa')}\theta_{2,(1,\aaa')}\cdot p_{(1,\aaa')} 
\\
\qquad\qquad\qquad = \bar\theta_{1,(0,\aaa')}\bar\theta_{2,(0,\aaa')}\cdot\bar p_{(0,\aaa')} + \bar \theta_{1,(1,\aaa')}\bar\theta_{2,(1,\aaa')}\cdot\bar p_{(1,\aaa')}.\\
\end{cases}
\end{align}
First we show that if 
there exist $\aaa'_1$, $\aaa'_2\in\{0,1\}^{K-1}$, $\aaa'_1\neq\aaa'_2$ such that  
\begin{align}\label{eq-s1s2}
\begin{cases}
\theta_{j,(\alpha_1,\aaa_1')} = \theta_{j,(\alpha_1,\aaa_2')}
\mbox{ and }
\bar\theta_{j,(\alpha_1,\aaa_1')} = \bar\theta_{j,(\alpha_1,\aaa_2')},\quad\forall j=1,2,~~\forall\alpha_1 = 0,1;\\[2mm]
\frac{p_{(1,\aaa_1')}}{p_{(0,\aaa_1')}}:= s_1 \neq
s_2 =: \frac{p_{(1,\aaa_2')}}{p_{(0,\aaa_2')}}.
\end{cases}
\end{align}
then one must have
\begin{equation}\label{eq-cg12}
\theta_{j,(\alpha_1,\aaa_1')} = \bar\theta_{j,(\alpha_1,\aaa_1')},\quad\forall j=1,2,~~\forall\alpha_1 = 0,1.
\end{equation}
After some transformations, the system of equations \eqref{eq-allv} yields
\[\begin{cases}
(\theta_{1,(0,\aaa')} - \theta_{1,(1,\aaa')})\cdot(\theta_{2,(0,\aaa')} - \bar\theta_{2,(1,\aaa')})\cdot p_{(0,\aaa')}\\
\qquad\qquad\qquad=
(\bar \theta_{1,(0,\aaa')} - \theta_{1,(1,\aaa')})\cdot(\bar \theta_{2,(0,\aaa')} - \bar\theta_{2,(1,\aaa')})\cdot \bar p_{(0,\aaa')};\\[1mm]
(\theta_{2,(0,\aaa')} - \bar\theta_{2,(1,\aaa')})\cdot  p_{(0,\aaa')}  + (\theta_{2,(1,\aaa')} - \bar\theta_{2,(1,\aaa')})\cdot \bar p_{(1,\aaa')}\\
\qquad\qquad\qquad= (\bar\theta_{2,(0,\aaa')} - \bar\theta_{2,(1,\aaa')})\cdot \bar p_{(0,\aaa')}.
\end{cases}\]
Under a multi-parameter model, $q_{1,1}=q_{2,1}=1$ yields that for generic parameters, $\theta_{i,(0,\aaa')} \neq \bar\theta_{i,(1,\aaa')}$, $i=1,2$, so the left (right) hand side of the first equation above is nonzero. And obviously the right hand side of the second equation above is nonzero.
Taking the ratio of the above two equations gives
\[\begin{aligned}
& \frac{(\theta_{1,(0,\aaa')} - \theta_{1,(1,\aaa')})\cdot(\theta_{2,(0,\aaa')} - \bar\theta_{2,(1,\aaa')})}{(\theta_{2,(0,\aaa')} - \bar\theta_{2,(1,\aaa')})  + (\theta_{2,(1,\aaa')} - \bar\theta_{2,(1,\aaa')})\cdot  p_{(1,\aaa')} /  p_{(0,\aaa')}} \\
=&~ (\bar \theta_{1,(0,\aaa')} - \theta_{1,(1,\aaa')})
:= f(\aaa').
\end{aligned}\]
The right hand side of the above equation does not involve any proportion parameter $\pp$ or $\bar\pp$. So for $\aaa'_1$, $\aaa'_2$ satisfying \eqref{eq-s1s2}, $f(\aaa'\aaa'_1)=f(\aaa'_2)$. Note that the left hand side of the above equation involves a ratio $p_{(1,\aaa')} /  p_{(0,\aaa')}$ depending on $\aaa'$. 
Equality $f(\aaa_1')=f(\aaa'_2)$ along with \eqref{eq-s1s2} imply
\begin{align*}
&(\theta_{2,(1,\aaa_1')} - \bar\theta_{2,(1,\aaa_1')})\cdot  \frac{p_{(1,\aaa_1')}}{p_{(0,\aaa_1')}} 
~=~ (\theta_{2,(1,\aaa_2')} - \bar\theta_{2,(1,\aaa_2')})\cdot  \frac{p_{(1,\aaa_2')}}{p_{(0,\aaa_2')}} \\
=&~ (\theta_{2,(1,\aaa_1')} - \bar\theta_{2,(1,\aaa_1')})\cdot  \frac{p_{(1,\aaa_2')}}{p_{(0,\aaa_2')}},\end{align*} 
and
 \begin{align*}
(\theta_{2,(1,\aaa_1')} - \bar\theta_{2,(1,\aaa_1')}) \cdot \left(\frac{p_{(1,\aaa_1')}}{p_{(0,\aaa_1')}} - \frac{p_{(1,\aaa_2')}}{p_{(0,\aaa_2')}}\right) &= 0,
\end{align*} 
then since $p_{(1,\aaa_1')} /  p_{(0,\aaa_1')} =s_1 \neq s_2=p_{(1,\aaa_2')} /  p_{(0,\aaa_2')}$, by assumption \eqref{eq-s1s2}, we have 
\[
\theta_{2,(1,\aaa_1')} - \bar\theta_{2,(1,\aaa_1')}=0.
\]
By symmetry of the four item parameters $\theta_{1,(0,\aaa')}$, $\theta_{1,(1,\aaa')}$, $\theta_{2,(0,\aaa')}$ and $\theta_{2,(1,\aaa')}$ in \eqref{eq-allv}, equalities \eqref{eq-cg12} therefore hold   following a similar argument.

Next we show that under the condition of the theorem, the  conclusions obtained so far give the generic identifiability of all the item parameters associated with the first two items, and hence proved the generic identifiability of all the model parameters. Since $\vv_1\vee\vv_2\neq\mo$, there must exist some attribute $k$, $k\neq 1$, that is not required by the first two items. Then for any item parameter $\theta_{j,\aaa}$ corresponding to item $j$, $j=1,2$ and attribute profile $\aaa=(\alpha_1,\alpha_2,\ldots,\alpha_K)$, define $\aaa'_1=(\alpha_2,\ldots,\alpha_K)$, and $\aaa'_2=(\alpha'_2,\ldots,\alpha'_K)$, $\alpha_l'=\alpha_l$ for any $l\neq k$ and $\alpha_k' = 1 - \alpha_k$, then
\[
\theta_{j,(\alpha_1,\aaa_1)} = \theta_{j,(\alpha_1,\aaa_2')}
\mbox{ and }
\bar\theta_{j,(\alpha_1,\aaa_1')} = \bar\theta_{j,(\alpha_1,\aaa_2')},\quad\forall j=1,2,~~\forall\alpha_1 = 0,1.
\]
This means we have found $\aaa_1'\neq \aaa'_2$ that satisfy the first equation in \eqref{eq-s1s2}, then as long as $ {p_{(1,\aaa_1')}}/{p_{(0,\aaa_1')}} \neq  {p_{(1,\aaa_2')}}/{p_{(0,\aaa_2')}}$ then $\theta_{j,\aaa} = \bar\theta_{j,\aaa}$ follows for $j=1,2$. Since this inequality constraint of the true parameters is a generic constraint, i.e. the parameters not satisfying this constraint falls in a Lebesgue measure zero set of the parameter space,   the generic identifiability of all the item parameters holds. 
Considering the fact $\theta_{j,(0,\aaa')}\neq \theta_{j,(1,\aaa')}$ generically, identifiability of the item parameters combined with \eqref{eq-allv} further gives the generic identifiability of   the proportion parameters $\pp$.
This completes the proof of part (b).
\end{proof}

\begin{proof}[Proof of Proposition \ref{prop-con1} and Proposition \ref{prop-con2}]
To prove Proposition \ref{prop-con2},  suppose the identifiability conditions for $\pp$-partial identifiability are satisfied.
  We introduce  a $2^J$-dimensional empirical response vector 
\begin{eqnarray*}
	\boldsymbol\gamma
	&= & \biggr\{
1,~
{N}^{-1}\sum_{i=1}^NI(\RR_i \succeq \ee_1),~
\cdots,~
{N}^{-1}\sum_{i=1}^NI(\RR_i \succeq \ee_J),\\
&&\quad 
{N}^{-1}\sum_{i=1}^NI(\RR_i \succeq \ee_1+\ee_2),~
\cdots,~
{N}^{-1}\sum_{i=1}^NI(\RR_i \succeq \sum_{j=1}^J\ee_j)\biggr\}^\top,
\end{eqnarray*} 
where elements of $\boldsymbol\gamma$ are indexed by  all the $|\{0,1\}^J|=2^J$ possible response patterns and they are in the same order as that of the columns of the $T$-matrix.
 First, by the definition of the $T$-matrix and the strong law of large numbers, we have $\boldsymbol\gamma\to T(\TT^0) \pp^0$ almost surely as $N\to\infty$. Second, the maximum likelihood estimators $\widehat\TT$ and $\widehat\pp$ satisfy 
$\|\boldsymbol\gamma - 
T(\widehat\TT) \widehat\nnu\| \to 0 , $ where $\|\cdot\|$ denotes the $L_2$ norm.
Therefore, combining these two gives
$$\|T(\TT^0) \nnu^0 - T(\widehat\TT) \widehat\nnu\| \to 0$$  
almost surely as $N\to\infty$.  
Then since the identifiability conditions are satisfied, we have that $T(\TT^0) \nnu^0 = T(\widehat\TT)\widehat\nnu$ indicates $(\TT^0,\nnu^0) = (\widehat\TT,\widehat\nnu)$. Therefore we obtain the consistency result  that $(\widehat\TT,\widehat\nnu)\to (\TT^0,\nnu^0)$ almost surely as $N\to\infty$. This proves Proposition \ref{prop-con1}.

To prove Proposition \ref{prop-con2}, suppose the identifiability conditions for generic identifiability are satisfied. Then according to Definition \ref{def-gen} of generic identifiability, there exists a proper algebraic subvariety $\mathcal V$ of $\mathcal T$, such that  $(\TT,\pp)$ are strictly identifiable on $\mathcal T\setminus \mathcal V$, and subvariety $\mathcal V$ has Lebesgue measure zero in the parameter space. If the true parameters $(\TT^0,\pp^0)$ belong to $\mathcal T\setminus \mathcal V$, then for any other valid set of parameters $(\bar\TT,\bar\pp)$, the equalities $T(\TT^0)\pp^0 = T(\bar\TT)\bar\pp$ indicate $(\TT^0,\pp^0) = (\bar\TT,\bar\pp)$. Similarly to the proof of Proposition \ref{prop-con1} in the last paragraph, we have
$$\|T(\TT^0) \pp^0 - T(\widehat\TT) \widehat\pp\| \to 0$$
almost surely as $N\to\infty$. And the identifiability of $(\TT^0,\pp^0)\in\mathcal T\setminus \mathcal V$ guarantees that $(\widehat\TT,\widehat\pp)\to (\TT^0,\pp^0)$ almost surely as $N\to\infty$. This proves Proposition \ref{prop-con2}.
\end{proof}

\medskip
\section*{Section D: Proof of Results in Section 5}\label{sec-lemma}
\medskip
\begin{proof}[Proof of Corollary \ref{prop-ao}]
If the $\Gamma$-matrix constructed as in the corollary is separable and contains distinct columns, 
then each attribute pattern $\aaa\in\ma$ corresponds to a unique equivalence class and $\nnu = \pp$, where $\nnu$ represents the grouped proportion parameters of $\Gamma$-matrix-induced equivalence classes introduced in Section \ref{sec-two-para}.
Further, the general constraints \eqref{eq-constraints} are satisfied for each item $j$.
Note that the proof of Theorem 1 only use the information that each item $j$ has two levels of item parameters $\theta_j^+$, $\theta_j^-$ which satisfy \eqref{eq-constraints}, and that proof does not depend on whether each item is specified as conjunctive (DINA) or disjunctive (DINO). 
Therefore Theorem 1 can be directly applied here. Given that $\Gamma$ is separable, conditions (C1) and (C2) lead to strict identifiability of the model parameters $(\ttt^+,\ttt^-,\pp)$. This concludes the proof.
\end{proof}


\begin{proof}[Proof of Corollary \ref{prop-mix-multi}  $(a)$]

To prove part (a), we first point out that the $\Gamma$-matrix defined in part (a) ensures the model parameters $(\TT,\pp)$ satisfy the general constraints \eqref{eq-constraints} for each item $j\in\ms$. The constraint set $\mc_j$ is just defined as $\mc_j = \{\aaa\in\ma:\, \Gamma_{j,\aaa}=1\}$. Then because the proofs of Theorem \ref{thm-order} and Proposition \ref{new} do not depend on the specific model assumption of each item, but only use the information that the constraints \eqref{eq-constraints} are satisfied for each $j$, the conclusions of Theorem \ref{thm-order} and Proposition \ref{new} still hold in the currently considered scenario. This proves part (a).
\end{proof}

\begin{proof}[Statement and Proof of Corollary \ref{prop-mix-multi} $(b)$]

We first introduce  the condition (E2)  needed in part (b).
For a binary vector $\boldsymbol a$, we say another binary vector $\boldsymbol b$ of the same length is a unit-shrinkage of $\boldsymbol a$, if $\boldsymbol b\preceq \boldsymbol a$ and $\boldsymbol b = \ee_k$ for some $k$.
Further, for a binary matrix $Q$, we say another binary matrix $\widetilde Q$ of the same size is a unit-shrinkage of $Q$, if for each $j$, the $j$th row vector of $\widetilde Q$ is either equal to, or a unit-shrinkage of, the $j$th row vector of $Q$.
The following condition (E2)  ensures the generic identifiability of  $(\TT,\pp)$.
\begin{itemize}

\item[(E2)] 
There exists a decomposition of $Q_{mult} = (Q_{mult,1}^\top, Q_{mult,2}^\top)^\top $ such that the submatrices $Q_{mult,1}$ and  $Q_{mult,2}$ satisfy the following conditions.
\begin{itemize}
\item[(E2.a)] There exists a ``unit-shrinkage" $\widetilde{Q}_{mult,1}$ of $Q_{mult,1}$ such that the    matrix
$\widetilde \Gamma = (\Gamma^{disj}(Q_{disj},\ma)^\top,~ \allowbreak\Gamma^{conj}(Q_{conj},\ma)^\top,~ \allowbreak\Gamma^{conj}(\widetilde Q_{mult,1},\ma)^\top)$  contains two disjoint separable submatrices $\Gamma_1$ and $ \Gamma_2$.

\item[(E2.b)] Each attribute is required by at least one item in $Q_{mult,2}$.
\end{itemize}
\end{itemize}


 Before proving Corollary \ref{prop-mix-multi} (b), we use an example to illustrate how to check its conditions (E2). 
 \begin{example}
 \normalfont{
 Consider  the following $Q$-matrix with items 1, 4 being two-parameter conjunctive, items 2, 5 being two-parameter disjunctive, and items 3, 6, 7 being multi-parameter. 
 \begin{equation*}
 Q = 
\begin{blockarray}{ccc}
\begin{block}{c(cc)}
conj & 1 & 0 \\
disj & 1 & 1 \\
mult & 1 & 1 \\
\cline{2-3}
conj & 1 & 0 \\
disj & 1 & 1 \\
mult & 1 & 1 \\
\cline{2-3}
mult & 1 & 1 \\
\end{block}
\end{blockarray}~\Rightarrow{}
\widetilde{Q} = 
\begin{blockarray}{cc}
\begin{block}{(cc)}
 1 & 0 \\
 1 & 1 \\
 \mathbf0 & 1 \\
\cline{1-2}
 1 & 0 \\
 1 & 1 \\
 \mathbf0 & 1 \\
\cline{1-2}
 1 & 1 \\
\end{block}
\end{blockarray}
~~\Rightarrow{}
\widetilde\Gamma=
\begin{blockarray}{cccc}
 (0,0) & (0,1) & (1,0) & (1,1) \\
\begin{block}{(cccc)}
  0 & 0 & 1 & 1\\
  0 & 1 & 1 & 1\\
  0 & 1 & 0 & 1\\
  \cline{1-4}
  0 & 0 & 1 & 1\\
  0 & 1 & 1 & 1\\
  0 & 1 & 0 & 1\\
  \cline{1-4}
  0 & 0 & 0 & 1\\
\end{block}
\end{blockarray}~.
\end{equation*}
  Then $\widetilde Q$ is a unit-shrinkage of $Q$, and $\widetilde\Gamma$ corresponds to $\widetilde Q$.
We can see that in $\widetilde Q$ items 1, 2, 3 give a separable $\Gamma_1$; items 4, 5, 6 give a separable $\Gamma_2$; and item 7 alone forms $Q_{multi,2}$ which requires both attributes. So (E2.a) and (E2.b) are satisfied and $(\TT,\pp)$ are generically identifiable.
}
\end{example}

\textsc{Proof of Corollary \ref{prop-mix-multi} $(b)$.}
First consider those multi-parameter items in the model. If item $j$ conforms to a multi-parameter model, then by our definition in the end of Section \ref{sec-cdm}, it could be a main-effect model or an all-effect model. Whichever multi-parameter model item $j$ follows, the item parameters $\theta_{j,\aaa}$ depend on the main effects of those required attributes of item $j$, so $\theta_{j,\aaa}$ can be written in the form of  \eqref{eq-tja} with some link function $f$. Now under condition (E2.a), since $\widetilde{Q}_{multi,1}$ is a unit-shrinkage of $Q_{multi,1}$, we denote 
\begin{align*}
S_u=\{j\in\ms:~
&j~\mbox{belongs to the}~ \widetilde Q_{multi,1}~\mbox{ part;}\\
&\widetilde \bq_j~\mbox{in}~ \widetilde Q_{multi,1}~\mbox{ is a unit-shrinkage of}~\bq_j\}.
\end{align*}
Then for each $j\in S_u$, there exists some $k_j\in\{1,\ldots,K\}$ such that $\widetilde{q}_{j,k_j}  = q_{j,k_j} = \ee_{k_j}$.
We claim that the $J\times |\ma|$ matrix $\widetilde{\TT} = (\widetilde{\ttt}_{j,\aaa})$ 
defined as follows actually give item parameters that form a submodel of the original model being considered.
\begin{align}\label{eq-trans}
\widetilde\theta_{j,\aaa} = 
\begin{cases}
f(\beta_{j,0} + \sum_{k=1}^K \beta_{j,k}\, \widetilde{q}_{j,k} \,\alpha_k) 
= f(\beta_{j,0} +  \beta_{j,k_j}  \,\alpha_{k_j}),& j\in S_u,~\aaa\in\ma;\\
\theta_{j,\aaa},&j\not\in S_u,~\aaa\in\ma.
\end{cases}
\end{align}
 In other words, $(\widetilde{\TT},\pp)$ are a valid set of parameters under the original $Q$-matrix and original model assumption. This is because setting all the interaction-effect coefficients and all the main-effect coefficients in \eqref{eq-tja} other than $\{\beta_{j,k_j}:\, j\in S_u\}$ to zero gives \eqref{eq-trans}. Note that for each item $j$ with $\bq$-vector $\widetilde\bq_j = \ee_{k_j}$, \eqref{eq-trans} actually defines a two-parameter conjunctive model for item $j$, with the two levels of item parameters being $\widetilde \theta^+_j = f(\beta_{j,0}+\beta_{j,k_j})$ and $\widetilde \theta^-_j = f(\beta_{j,0})$. 
 Now we claim that given the $\widetilde\TT$ constructed in \eqref{eq-trans}, and given the two separable matrices $\Gamma_1$ and $\Gamma_2$ described in (E2.a),  $\wt \TT_{\Gamma_1}$ and  $\wt \TT_{\Gamma_2}$ both have full column rank.
 
 In summary, the above reasoning from (E2.a) indicates that the two $T$-matrices $T(\TT_{\Gamma_1})$ and $T(\TT_{\Gamma_2})$ are both generically full-column-rank. Combining condition (E2.b) that each column contains at least one entry of ``1" in the submatrix $Q_{multi,2}$, a similar argument as that in the proof of Theorem \ref{thm-gen-q}  gives that the entire model is generically identifiable. 
\end{proof}

\begin{proof}[Proof of Proposition \ref{prop-poly}]
We introduce a useful lemma before proving the proposition.
\begin{lemma}\label{lem-high}
Under a restricted latent class model with categorial responses $\RR\in\prod_{j=1}^J \{0,1,\ldots,L_j-1\}$, if two sets of parameters $(\TT^{\text{cat}},\pp)$ and $(\bar\TT^{\text{cat}},\bar\pp)$ satisfy
\begin{equation}\label{eq-catdef}
\mathbb P(\RR\mid \TT^{\text{cat}},\pp) = \mathbb P(\RR\mid \bar\TT^{\text{cat}},\bar\pp),
\end{equation}
then for any response pattern $\rr^{H}=(r^H_1,\ldots,r^H_J)\in\prod_{j=1}^J \{1,\ldots,L_j-1\}$ that consists of higher-level responses (higher than the basic level-0) to all the items, we have the following $2^J$ equalities
\begin{equation}\label{eq-high}
\sum_{\aaa\in\ma} p_{\aaa} \prod_{j:\, r_j= r_j^H}\theta^{(r_j^H)}_{j,\aaa}=
\sum_{\aaa\in\ma}\bar p_{\aaa} \prod_{j:\, r_j = r_j^H}\bar\theta^{(r_j^H)}_{j,\aaa},\quad \forall \rr\in\prod_{j=1}^J\{0, r^H_{j}\}.
\end{equation}
\end{lemma}

We now continue with the proof of Proposition \ref{prop-poly}.
Given any higher-level response pattern $\rr^H$, we can define a generalized $T$-matrix $T^{\rr^H}$ of size $2^J\times m$, with the $(\rr,\aaa)$th entry being
$$
\{T^{\rr^H}(\TT^{\text{cat}})\}_{\rr,\aaa} = \sum_{\aaa\in\ma} p_{\aaa} \prod_{j:\, r_j= r_j^H}\theta^{(r_j^H)}_{j,\aaa},\quad
\rr\in\prod_{j=1}^J\{0, r^H_{j}\}.
$$
Then \eqref{eq-high} in Lemma \ref{lem-high} can be rewritten as
\begin{equation}\label{eq-tcat}
T^{\rr^H}(\TT^{\text{cat}})\pp = T^{\rr^H}(\bar\TT^{\text{cat}})\bar\pp.
\end{equation}
which has the same form as \eqref{eq1}, $T(\TT)\pp=T(\bar\TT)\bar\pp$. Now consider all the proposed sufficient conditions for strict (or $\pp$-partial, generic) identifiability in Sections \ref{sec-two-para} and \ref{sec-multi-para}. In those proofs, we always start with assuming \eqref{eq1} holds and then show $(\TT,\pp)=(\bar\TT,\bar\pp)$ under those sufficient conditions. In the current case of categorical responses, under the same set of sufficient conditions as those in Sections  \ref{sec-two-para} and \ref{sec-multi-para}, assuming \eqref{eq-tcat} holds leads to $\pp=\bar\pp$ and
$$
\theta^{(r_j^H)}_{j,\aaa}=\bar \theta^{(r_j^H)}_{j,\aaa},\quad \forall \aaa\in\ma,~ j\in\{1,\ldots,J\},
$$
for the specific $\rr^H$. Since $\rr^H$ is arbitrary, we obtain 
$$\theta^{(r_j)}_{j,\aaa}=\bar \theta^{(r_j)}_{j,\aaa},
\quad \forall \aaa\in\ma,~ j\in\{1,\ldots,J\},~r_j\in\{1,\ldots,L_j-1\}.
$$
This further gives $\theta^{(0)}_{j,\aaa}=1-\sum_{l>0}\theta^{(l)}_{j,\aaa}=1-\sum_{l>0}\bar\theta^{(l)}_{j,\aaa}=\bar \theta^{(0)}_{j,\aaa}$ for any item $j$. By far we have shown if \eqref{eq-catdef} holds and the previously proposed sufficient identifiability conditions are satisfied, then $(\TT^{\text{cat}},\pp)=(\bar\TT^{\text{cat}},\bar\pp)$ hold. This concludes the proof of the proposition.
\end{proof}

\begin{proof}[Proof of Proposition \ref{prop-rbm}]
We rewrite the probability distribution function of a RBM as
\begin{align}\label{eq-rbmg}
\mathbb P(\RR, \aaa^{(1)},\cdots) = \frac{1}{Z}\exp\Big( 
- \RR^\top \bo W^Q \aaa^{(1)} - (\aaa^{(1)})^\top \bo U \aaa^{(2)} - \cdots \Big),
\end{align}
where the ``$\cdots$" part denote deeper latent layers $\aaa^{(2)}$, $\aaa^{(3)}$, etc.
The conditional distribution of $R_j$ given $\aaa^{(1)}$ can be written as
\begin{align}\label{eq-cond}
\mathbb P(R_j=1\mid\aaa^{(1)}) 
=&~ \frac{\exp\Big(W^Q_{j,\Cdot}\aaa^{(1)} + a_j\Big)}{1+\exp\Big(W^Q_{j,\Cdot}\aaa^{(1)} + a_j\Big)} \\ \notag
=&~ \sigma\Big(W^Q_{j,\Cdot}\aaa^{(1)} + a_j\Big),
\end{align}
where $\sigma(x)=e^x/(1+e^x)$ denotes the sigmoid function.
Denote the length of $\aaa^{(1)}$ by $K_1$. Since $\aaa^{(1)}\in\{0,1\}^{K_1}$ can be viewed as a latent  attribute pattern, we  denote $\aaa^{(1)}=(\alpha^{(1)}_1,\ldots,\alpha^{(1)}_{K_1})$ and further write \eqref{eq-cond} as
$$
\theta_{j,\aaa^{(1)}} = \sigma\Big(\sum_{k:\,W^Q_{j,k}\neq 0} W^Q_{j,k} \alpha^{(1)}_k\Big).
$$
Now it is clear from the above display that the RBM defined in \eqref{eq-rbmg} can be viewed as a multi-parameter main-effect restricted latent class model with $J$ items and $K_1$ latent attributes, with a $Q$-matrix resulting from the sparse bipartite structure $\bo W^Q$. Therefore, part (a) of the theorem follows from the generic identifiability result of the unrestricted latent class models \citep{allman2009} that $J\geq 2K_1+1$ suffices for generic identifiability of the item parameters $\TT$, and hence $\bo W^Q$. Also, part (b) of the theorem holds because when $Q$ satisfies the sufficient conditions for strict or generic identifiability under a multi-parameter restricted latent class model, the item parameters $\TT=(\theta_{j,\aaa})$ are strictly or generically identifiable.
This completes the proof of the theorem.
\end{proof}

\color{black}

\medskip
\section*{Section E: Proof of Technical Lemmas}\label{sec-lemma}
\medskip
\begin{proof}[Proof of Lemma \ref{lem-rk} (on page \pageref{lem-rk}, Section B)]
Without loss of generality, assume $\Gamma^S$ is separable with the item set $S=\{1,\ldots,J\}$.
Define $\ttt^* = \sum_{j\in S}\theta_{\ee_j,\mo}$.
The aim is to find response patterns $\rr_1,~\ldots,~\rr_{m-1}$ such that the corresponding row vectors of $\rr_0:=\mathbf0,~\rr_1,~\ldots,~\rr_{m-1}$ in the transformed $T(\TT-\ttt^*\mo^\top)$ form a $m\times m$ lower triangular matrix 
with nonzero diagonal elements, which will prove the conclusion that $T(\TT)$ has full column rank $m$. 

Since $\Gamma^S$ is separable and every two different column vectors of it are distinct,
without loss of generality, assume the $m$ column vectors in the ideal response matrix $\Gamma^{S}$ are arranged in a lexicographic order, where the first column is an all-zero column corresponding to the universal least capable class $\aaa_0$. In other words, for any $0\leq k < h\leq m-1$, $\Gamma_{\cdot,\,\aaa_k}^{S}$ is of smaller lexicographical order than $\Gamma_{\cdot,\,\aaa_h}^{S}$. In the following proof denote $\Gamma:=\Gamma^{S}$ to simplify notations. Define response patterns $\rr_1,\ldots,\rr_{m-1}$ to be 
\[
\rr_k = \sum_{j:\,\Gamma_{j,\,\aaa_k} = 0}\ee_j,\quad k = 1,\ldots,C-1,
\]
and define a sub-matrix $T^{sub}$ of $T(\TT)$ whose $m$ rows corresponding to response patterns $\rr_0,\rr_1,\ldots,\rr_{m-1}$ and $m$ columns corresponding to class profiles $\aaa_0,\aaa_1,\ldots,\aaa_{m-1}$. We claim that $T^{sub}(\TT-\ttt^*\mo^\top)$ is a lower triangular square matrix of full rank $m$. This is because for any $0\leq k \leq m-1$, the row vector corresponding to $\rr_k$ in $T^{sub}(\TT-\ttt^*\mo^\top)$ is
\begin{equation}\label{eq-lem2}
T^{sub}_{\rr_k,\Cdot}(\TT-\ttt^*\mo^\top)
 = \bigodot_{j:\,\Gamma_{j,\,\aaa_k}=0} T_{\ee_j,\Cdot}\biggr(\TT-\Big(\sum_{j=1}^J\theta_{\ee_j,\mo}\ee_j\Big)\mo^\top\biggr).
\end{equation}
For any $h>k$, there must exist an item $j$ such that $\Gamma_{j,\,\aaa_k}=0$ and $\Gamma_{j,\,\aaa_h}=1$. Existence of such $j$ means that $\aaa_h$ is capable of at least one item not mastered by $\aaa_k$, and guarantees that the $\aaa_h$-entry of the above row vector (\ref{eq-lem2}) is zero. We have shown $T^{sub}_{\rr_k,\aaa_h}(\TT-\ttt^*\mo^\top)=0$ for arbitrary $0\leq k < h \leq m-1$, so $T^{sub}(\TT-\ttt^*\mo^\top)$ is a lower triangular square matrix. Moreover, the diagonal entries are
\[\begin{aligned}
T^{sub}_{\rr_k,\aaa_k}(\TT-\ttt^*\mo^\top)
&= \prod_{j:\,\Gamma_{j,\aaa_k}=0} (\theta_{\ee_j,\aaa_k} - \theta_{\ee_j,\mo})
\neq 0,
\end{aligned}\]
so $T^{sub}(\TT-\ttt^*\mo^\top)$ is of rank $m$, with the shape
$$
{\small \left|\begin{array}{c ccc}
\prod\limits_{j\in S_{\aaa_0}} (\theta_{\ee_j,\aaa_0} - \theta_{\ee_j,\mo}) & 0 & \cdots & 0 \\
\prod\limits_{j\in S_{\aaa_1}} (\theta_{\ee_j,\aaa_0} - \theta_{\ee_j,\mo})
& \prod\limits_{j\in S_{\aaa_1}} (\theta_{\ee_j,\aaa_1} - \theta_{\ee_j,\mo})
 & \cdots & 0 \\
\vdots & \vdots & \ddots & \vdots \\
\prod\limits_{j\in S_{\aaa_{m-1}}} (\theta_{\ee_j,\aaa_0} - \theta_{\ee_j,\mo})
 & * & \cdots & \prod\limits_{j\in S_{\aaa_{m-1}}} (\theta_{\ee_j,\aaa_{m-1}} - \theta_{\ee_j,\mo}) \\
\end{array}\right| 
}$$
where $S_{\aaa_i} := \{j:\,\Gamma_{j,\aaa_i}=0\}$ for $i=0,1,\ldots,m-1$.
The proof of Lemma \ref{lem-rk} is complete.
\end{proof}

\smallskip
\begin{proof}[Proof of Lemma \ref{lemmaS2} (on page \pageref{lemmaS2}, Section B)]
Since $\theta^+_j > \theta^-_j$ for each item $j$ from the definition constraints of restricted latent class models, 
$T_{\ee_j,\Cdot}(\TT)\pp =  T_{\ee_j,\Cdot}(\bar \TT)\bar\pp$ indicates
\[
\theta_j^+ = \sum_{\aaa\in\ma} \theta_j^+ p_{\aaa} \geq 
\sum_{\aaa\in\ma} \theta_{j,\aaa} p_{\aaa} = \sum_{\aaa\in\ma} \bar\theta_{j,\aaa}\bar p_{\aaa} \geq \sum_{\aaa\in\ma}\bar\theta_j^-\bar p_{\aaa} = \bar\theta_j^-,
\]
where among the two ``$\geq$" there is at least a strict ``$>$". This is because the first ``$\geq$" is an equality sign only if all the latent classes are capable of item $j$, namely $\Gamma_{j,\aaa}=1$ for all $\aaa\in\ma$, and in this case, $\sum_{\aaa\in\ma} \bar\theta_{j,\aaa}\bar p_{\aaa} = \bar\theta_j^+ > \bar\theta_j^-$ and therefore  the second ``$\geq$" must be a strict ``$>$". Similarly,   the second ``$\geq$" is an equality sign only if all the latent classes are incapable of item $j$, and in this case, $\theta_j^+>\theta_j^-=\sum_{\aaa\in\ma} \theta_{j,\aaa} p_{\aaa}$ and therefore the first ``$\geq$" must be a strict ``$>$".
This proves that $\theta_j^+ > \bar\theta_j^-$ for all $j$, and similarly we have $\theta_j^- < \bar\theta_j^+$ for all $j$.
\end{proof}

\smallskip
\begin{proof}[Proof of Lemma \ref{lem-neq} (on page \pageref{lem-neq}, Section C)]
Since the sub-matrix  $T(\TT_{S_1})$ has full column rank $m$, there exists a vector $\mm_{\aaa}$ such that
\[
\mm_{\aaa}^\top\cdot T(\TT_{S_1}) = (\mz, \underbrace{1}_\text{column $\aaa$},\mz),
\]
On the other hand, Equation (\ref{eq1}) implies $\mm_{\aaa}^\top\cdot T(\TT_{S_1})\pp = \mm_{\aaa}^\top\cdot T(\bar\TT_{S_1})\bar\pp$, which further indicates the row vector $\mm_{\aaa}^\top\cdot T(\bar\TT_{S_1})$ also contains at least one nonzero element in some column. Denote such a column by $\aaa^*$ and the nonzero value by $\bar x_{\aaa^*}$. Since the sub-matrix  $T(\bar\TT_{S_2})$ also has full column rank, there exists another vector $\nn_{\aaa^*}$ such that
\[
\nn_{\aaa^*}^\top\cdot T(\bar\TT_{S_2}) = (\mz, \underbrace{1}_\text{column $\aaa^*$},\mz),
\]
Again from Equation (\ref{eq1}) we have that the $\aaa$-th entry of the row vector $\nn_{\aaa^*}^\top T(\TT_{S_2})$ is also nonzero. We denote this nonzero value by $y_{\aaa}$. Then we have
\[\begin{aligned}
\{\mm_{\aaa}^\top\cdot T(\TT_{S_1})\}\odot \{\nn_{\aaa^*}^\top\cdot T(\TT_{S_2})\} = (\mz,\underbrace{y_{\aaa}}_\text{column $\aaa$},\mz), \\
\{\mm_{\aaa}^\top\cdot T(\bar\TT_{S_1})\}\odot \{\nn_{\aaa^*}^\top\cdot T(\bar\TT_{S_2})\} = (\mz,\underbrace{\bar x_{\aaa^*}}_\text{column $\aaa^*$},\mz).
\end{aligned}\]
Now consider one more row $\ee_j$ for an arbitrary $j>M_1+M_2$ in the $T$-matrix, we have
\[\begin{aligned}
T_{\ee_j,\cdot}(\TT) \odot \{\mm_{\aaa}^\top\cdot T(\TT_{S_1})\}\odot \{\nn_{\aaa^*}^\top \cdot T(\TT_{S_2})\} &= (\mz,\underbrace{\theta_{\ee_j,\aaa} \cdot y_{\aaa}}_\text{column $\aaa$},\mz), \\
T_{\ee_j,\cdot}(\bar\TT) \odot \{\mm_{\aaa}^\top\cdot T(\bar\TT_{S_1})\}\odot \{\nn_{\aaa^*}^\top \cdot T(\bar\TT_{S_2})\} &= (\mz,\underbrace{\bar\theta_{\ee_j,\aaa^*}\cdot \bar x_{\aaa^*}}_\text{column $\aaa^*$},\mz).
\end{aligned}\]
The above four equations along with Equation (\ref{eq1}) imply 
\begin{equation}\label{eq-permu}
\theta_{\ee_j,\aaa} = \bar\theta_{\ee_j,\aaa^*},\quad \forall j>M_1+M_2.
\end{equation}
Now that $(\theta_{\ee_j,\aaa},j>M_1+M_2) = (\bar\theta_{\ee_j,\aaa^*},j>M_1+M_2)$, condition (C4) implies that there exists a vector $\boldsymbol s_{\aaa}$ such that
\[\begin{aligned}
\boldsymbol s_{\aaa}^\top\cdot T(\TT_{(M_1+M_2+1):J}) &= (0,*,\ldots,*,\underbrace{1}_\text{column $\aaa$},*,\ldots,*),\\
\boldsymbol s_{\aaa}^\top\cdot T(\bar\TT_{(M_1+M_2+1):J}) &= (0,*,\ldots,*,\underbrace{1}_\text{column $\aaa^*$},*,\ldots,*).\\
\end{aligned}\]
Next redefine 
\[
\ttt^* = \sum_{h\in S_1:\Gamma_{h,\aaa}=0}\theta_{\ee_h,\mo}\ee_h \mbox{ and } \rr^* = \sum_{h\in S_1:\Gamma_{h,\aaa}=0}\ee_h,
\]
then we have
\begin{align}\label{eq7}
  \{\boldsymbol s_{\aaa}^\top\cdot T(\TT_{(M_1+M_2+1):J})\}
&\odot \{ \nn_{\aaa^*}^\top\cdot T(\TT_{S_2}) \}
\odot \{ T_{\rr^*,\cdot} (\TT_{S_1} - \ttt^*\mo^\top) \} \notag\\
&=   \Big( \mz, \underbrace{y_{\aaa}\prod_{h\in S_1:\Gamma_{h,\aaa}=0}(\theta_{\ee_h,\aaa} - \theta_{\ee_j,\mo})}_\text{column $\aaa$}, \mz \Big),
 \end{align}
and
\begin{align}\label{eq8}
 \{\boldsymbol s_{\aaa}^\top\cdot T(\bar\TT_{(M_1+M_2+1):J})\}
 &\odot \{ \nn_{\aaa^*}^\top\cdot T(\bar\TT_{S_2}) \}
\odot \{ T_{\rr^*,\cdot} (\bar\TT_{S_1} - \ttt^*\mo^\top) \} \notag\\
&=  \Big( \mz, \underbrace{\prod_{h\in S_1:\Gamma_{h,\aaa}=0}(\bar\theta_{\ee_h,\aaa} - \theta_{\ee_j,\mo})}_\text{column $\aaa^*$}, \mz \Big).
\end{align} 
Since the $\aaa$-entry of (\ref{eq7}) is nonzero, the $\aaa^*$-entry of (\ref{eq8}) must also be nonzero since by (\ref{eq1}) we have $(\ref{eq7})\cdot\pp = (\ref{eq8})\cdot\bar\pp$.
Further consider row $j\in S_1$ such that $\Gamma_{j,\aaa}=1$. Obviously $\ee_j$ does not appear in the summation of the previously defined $\rr^*$, so we have
\begin{align}\label{eq9}
 \{\boldsymbol s_{\aaa}^\top\cdot T(\TT_{(M_1+M_2+1):J})\}
 &\odot \{ \nn_{\aaa^*}^\top\cdot T(\TT_{S_2}) \}
\odot \{ T_{\rr^*+\ee_j,\cdot} (\TT_{S_1} - \ttt^*\mo^\top) \} \notag\\
&=   \Big( \mz, \underbrace{\theta_{\ee_j,\aaa} y_{\aaa}\prod_{h\in S_1:\Gamma_{h,\aaa}=0}(\theta_{\ee_h,\aaa} - \theta_{\ee_j,\mo})}_\text{column $\aaa$}, \mz \Big),  
\end{align}
and
\begin{align}\label{eq10}
  \{\boldsymbol s_{\aaa}^\top\cdot T(\bar\TT_{(M_1+M_2+1):J})\}
 &\odot \{ \nn_{\aaa^*}^\top\cdot T(\bar\TT_{S_2}) \}
\odot \{ T_{\rr^*+\ee_j,\cdot} (\bar\TT_{S_1} - \ttt^*\mo^\top) \} \notag\\
&=   \Big( \mz, \underbrace{\bar\theta_{\ee_j,\aaa^*}\prod_{h\in S_1:\Gamma_{h,\aaa}=0}(\bar\theta_{\ee_h,\aaa} - \theta_{\ee_j,\mo})}_\text{column $\aaa^*$}, \mz \Big),
\end{align} 
and therefore
\[
\theta_{\ee_j,\mo} = \theta_{\ee_j,\aaa} = \frac{(\ref{eq9})\cdot\pp}{(\ref{eq7})\cdot\pp} = \frac{(\ref{eq10})\cdot\bar\pp}{(\ref{eq8})\cdot\bar\pp} = 
\bar\theta_{\ee_j,\aaa^*}, \quad \forall j\in S_1\text{ s.t. } \Gamma_{j,\aaa}=1 .
\]
Therefore for any $j\in S_1$ and any $\aaa'$ such that $\Gamma_{j,\aaa'}=0$, as long as there exists some $\aaa$ such that $\Gamma_{j,\aaa}=1$, we have $\theta_{\ee_j,\mo} = \bar\theta_{\ee_j, \aaa^*}$ from the above proof. Then the following inequality holds
\[
\forall j\in S_1,\quad \forall \aaa', \text{ s.t. } \Gamma_{j,\aaa'=0},
\quad
\theta_{\ee_j,\aaa'} < \theta_{\ee_j, \aaa} = \theta_{\ee_j,\mo} = \bar\theta_{\ee_j, \aaa^*}
\leq \bar \theta_{\ee_j, \mo}.
\]
Similarly we also have $\theta_{\ee_j,\aaa'} < \bar\theta_{\ee_j,\mo}$ for any $j\in S_2$ and $\Gamma_{j,\aaa'}=0$; and $\theta_{\ee_j,\mo} > \bar\theta_{\ee_j,\aaa'}$ for any $j\in S_1\cup S_2$ and $\Gamma_{j,\aaa'}=0$.
The proof of Lemma \ref{lem-neq} is complete.
\end{proof}

\smallskip
\begin{proof}[Proof of Lemma \ref{lem-order-a} (on page \pageref{lem-order-a}, Section C)]
We focus on $T(\TT_{S_2})$ first.
As shown in the proof of Lemma \ref{lem-neq}, for any $\aaa$ there exists some $\aaa^*$, which depends on $\aaa$, such that
\[\begin{aligned}
&\mm_{\aaa}^\top\cdot T(\TT_{S_1}) = (\mz, \underbrace{1}_\text{column $\aaa$},\mz),\\
&\nn_{\aaa^*}^\top\cdot T(\bar\TT_{S_2}) = (\mz, \underbrace{1}_\text{column $\aaa^*$},\mz),\\
&\{\mm_{\aaa}^\top\cdot T(\TT_{S_1})\}\odot \{\nn_{\aaa^*}^\top\cdot T(\TT_{S_2})\} = (\mz,\underbrace{y_{\aaa}}_\text{column $\aaa$},\mz),\quad y_{\aaa}\neq 0 \\
&\{\mm_{\aaa}^\top\cdot T(\bar\TT_{S_1})\}\odot \{\nn_{\aaa^*}^\top\cdot T(\bar\TT_{S_2})\} = (\mz,\underbrace{\bar x_{\aaa^*}}_\text{column $\aaa^*$},\mz),\quad \bar x_{\aaa^*}\neq 0
\end{aligned}\]
for some vectors $\mm_{\aaa}$ and $\nn_{\aaa^*}$. And based on these constructions we proved 
\[\theta_{\ee_j,\aaa} = \bar\theta_{\ee_j,\aaa^*},\quad \forall j>M_1+M_2.\]
Clearly from the constructions we have
\[
\{\nn_{\aaa^*}^\top \cdot T(\TT_{S_2})\}_{\aaa} \neq 0,
\]
then we furthermore claim that under condition (C4*), $\nn_{\aaa^*}$ also has the following property
\begin{equation}\label{eq-claim}
\{\nn_{\aaa^*}^\top \cdot T(\TT_{S_2})\}_{\aaa'} = 0,\quad \forall\aaa'\precneqq_{S_1}\aaa.
\end{equation}
Since otherwise if $\{\nn_{\aaa^*}^\top \cdot T(\TT_{S_2})\}_{\aaa'} = 0$ for some $\aaa'\preceq_{S_1}\aaa$, we would have
\begin{equation}\label{eq-t7-1}
\{\mm_{\aaa'}^\top\cdot T(\TT_{S_1})\}\odot \{\nn_{\aaa^*}^\top\cdot T(\TT_{S_2})\}
= (\mz,\underbrace{s_{\aaa'}}_\text{column $\aaa'$},\mz),\quad s_{\aaa'}\neq 0,
\end{equation}
\begin{equation}\label{eq-t7-2}
\{\mm_{\aaa'}^\top\cdot T(\bar\TT_{S_1})\}\odot \{\nn_{\aaa^*}^\top\cdot T(\bar\TT_{S_2})\}
= (\mz,\underbrace{\bar t_{\aaa^*}}_\text{column $\aaa^*$},\mz), \quad \bar t_{\aaa^*}\neq 0,
\end{equation}
then using similar argument as that in Lemma \ref{lem-neq}, for any $j\in (S_1\cup S_2)^c$  we would have
\[
\theta_{\ee_j,\aaa'} = 
\frac{\theta_{\ee_j,\aaa'}\cdot(\ref{eq-t7-1})\cdot\pp}{(\ref{eq-t7-1})\cdot\pp}
= \frac{\bar\theta_{\ee_j,\aaa'}\cdot(\ref{eq-t7-2})\cdot\bar\pp}{(\ref{eq-t7-2})\cdot\bar\pp}
= \bar\theta_{\ee_j,\aaa^*} = \theta_{\ee_j,\aaa},
\]
which contradicts Condition (C4) that $\Gamma^{(S_1\cup S_2)^c}_{\Cdot,\aaa} = \Gamma^{(S_1\cup S_2)^c}_{\Cdot,\aaa'}$ for any $\aaa'\precneqq_{S_1} \aaa$, since (C4) naturally leads to 
$(\theta_{j,\aaa},j\in (S_1\cup S_2)^c)\neq (\theta_{j,\aaa'},j\in (S_1\cup S_2)^c)$ for any $\aaa'\precneqq_{S_1} \aaa$.
So the claim (\ref{eq-claim}) must hold. By far we have found $\vv_{\aaa}:=\nn_{\aaa^*}$ that satisfies the first equation in (\ref{eq-order}) for each $\aaa$. By symmetry between $(\TT,\pp)$ and $(\bar\TT,\bar\pp)$, using exactly the same techniques will lead to $\uu_{\aaa}$ for each $\aaa$ that satisfies the second equation in (\ref{eq-order}). 
\end{proof}

\begin{proof}[Proof of Lemma \ref{lem-high} (on page \pageref{lem-high}, Section D)]
Equation \eqref{eq-high} for $\RR=\rr$ can be written as 
\begin{align}\label{eq-rh}
&\sum_{\aaa\in\ma} p_{\aaa} \prod_{j:\, r_j=r_j^H} \theta^{(r_j^H)}_{j,\aaa}\prod_{j:\, r_j\neq r_j^H}\theta^{(r_j)}_{j,\aaa}=
\sum_{\aaa\in\ma} \bar p_{\aaa} \prod_{j:\, r_j=r_j^H}  \bar \theta^{(r_j^H)}_{j,\aaa}\prod_{j:\, r_j\neq r_j^H} \bar \theta^{(r_j)}_{j,\aaa}.
\end{align}
We denote $\{0,\ldots,L_j-1\}$ by $[L_j]$ for simplicity.
Now consider an arbitrary item set $S\subseteq\ms$, and we write $\rr_{S} = \rr^H_S$ if $r_j=r^H_j$ for any $j\in S$. 
 For this $S$, we sum \eqref{eq-rh} over all response patterns $\rr$ for which $r_j= r_j^H$ if and only if $j\in S$ (i.e., $\rr$ satisfies $\rr_S = \rr_S^H$ and $\rr_{S^c}\in  \prod_{j\notin S}[L_j]\setminus[r_j^H]$), then the left hand side (LHS) of the new equation is
\begin{align*}
&~ \sum_{\rr: \,\rr_S = \rr^H_S,\, \atop\rr_{S^c}\in \prod_{j\neq S}[L_j]\setminus[r_j^H]} \Big( \sum_{\aaa\in\ma} p_{\aaa} \prod_{j:\, r_j=r_j^H} \theta^{(r_j^H)}_{j,\aaa}\prod_{j:\, r_j\neq r_j^H}\theta^{(r_j)}_{j,\aaa}\Big) \\
=&~ \sum_{\aaa\in\ma} p_{\aaa} 
\prod_{j\in  S} \theta^{(r_j^H)}_{j,\aaa}
\sum_{\rr: \,\rr_S = \rr^H_S,\, \atop\rr_{S^c}\in \prod_{j\neq S}[L_j]\setminus[r_j^H]}
\prod_{j\notin S}\theta^{(r_j)}_{j,\aaa} \\
=&~ \sum_{\aaa\in\ma} p_{\aaa} 
\prod_{j\in  S} \theta^{(r_j^H)}_{j,\aaa}
\prod_{j\notin S}\Big(\sum_{r_j\neq r_j^H} \theta^{(r_j)}_{j,\aaa} \Big)\\
=&~ \sum_{\aaa\in\ma} p_{\aaa} 
\prod_{j\in  S} \theta^{(r_j^H)}_{j,\aaa}
\prod_{j\notin S}\Big(1 - \theta^{(r_j^H)}_{j,\aaa} \Big),
\end{align*}
so from \eqref{eq-rh} we have 
\begin{equation}\label{eq-subs}
\sum_{\aaa\in\ma} p_{\aaa} 
\prod_{j\in  S} \theta^{(r_j^H)}_{j,\aaa}
\prod_{j\notin S}\Big(1 - \theta^{(r_j^H)}_{j,\aaa} \Big)
=
\sum_{\aaa\in\ma}\bar p_{\aaa} 
\prod_{j\in  S} \bar \theta^{(r_j^H)}_{j,\aaa}
\prod_{j\notin S}\Big(1 - \bar \theta^{(r_j^H)}_{j,\aaa} \Big)
\end{equation}
holds for any $S\subseteq\ms$. By far we have shown the system of $2^J$ equations \eqref{eq-subs} hold for any $\rr^H$.
Note that \eqref{eq-subs} can be viewed as probability of a response pattern consisting of binary responses, where for each item $j$ and each latent class $\aaa$, there are two possible responses with probabilities $\theta^{(r_j^H)}_{j,\aaa}$ and $1 -  \theta^{(r_j^H)}_{j,\aaa}$ respectively.
Then similar to the proof of Proposition \ref{prop1} which establishes equivalence between equality of probability mass functions and equality of marginal probabilities, \eqref{eq-subs} is equivalent to \eqref{eq-high} in the lemma.
This completes the proof of Lemma \ref{lem-high}.
\end{proof}

\end{appendix}

\end{document}